\documentclass[10pt]{article}
\usepackage{xcolor}
\usepackage{amsmath}
\usepackage{amssymb}
\usepackage{amsthm}
\usepackage{fullpage}
\usepackage{graphics}
\usepackage{graphicx}
\usepackage{euscript,mathrsfs}
\usepackage{subcaption}
\usepackage{caption}
\usepackage[font=footnotesize,labelfont=bf]{caption}
\usepackage[draft]{todonotes} 
\usepackage{mathtools}
\usepackage{bbm}
\usepackage{multicol}
\usepackage[shortlabels]{enumitem}
\usepackage{hyperref}
\usepackage{cleveref}

\newlist{abbrv}{itemize}{1}
\setlist[abbrv,1]{label=,labelwidth=1in,align=parleft,itemsep=0.1\baselineskip,leftmargin=!}

\makeatletter
\renewcommand\paragraph{\@startsection{paragraph}{4}{\z@}%
	{-2.5ex\@plus -1ex \@minus -.25ex}%
	{1.25ex \@plus .25ex}%
	{\normalfont\normalsize\bfseries}}
\makeatother
\newtheorem{thm}{Theorem}
\newtheorem{lm}[thm]{Lemma}
\newtheorem{defn}[thm]{Definition}
\newtheorem{prop}[thm]{Proposition}
\newtheorem{rmk}[thm]{Remark}

\newtheorem{notations}[thm]{Notations}
\newtheorem{assum}[thm]{Assumption}

\newcommand{\beq}{ \begin{equation} }
\newcommand{\eeq}{ \end{equation} }
\newcommand{\bes}{ \begin{split} }
	\newcommand{\ees}{ \end{split} }
\newcommand{\beqq}{ \begin{equation*} }
\newcommand{\eeqq}{ \end{equation*} }

\newcommand{\dd}{{\mathrm d}}
\newcommand{\ii}{\mathrm{i}}
\newcommand{\ddbar}[2]{\frac{{\mathrm d}#1}{2\pi {\mathrm i}#2}}

\newcommand{\prob}{\mathbb{P}}
\newcommand{\intZ}{\mathbb{Z}}
\newcommand{\realR}{\mathbb{R}}
\newcommand{\complexC}{\mathbb{C}}

\newcommand{\polylog}{\mathrm{Li}}

\newcommand{\dist}{{\rm{dist}}\,}

\newcommand{\gftn}{\mathcal{G}}
\newcommand{\gftnv}{{\tilde{\mathcal{G}}}}
\newcommand{\ggftn}{\mathrm{G}}

\newcommand{\energy}{\mathcal{E}}
\newcommand{\energyv}{{\tilde{\mathcal{E}}}}
\newcommand{\ich}{\mathrm{ch}}
\newcommand{\ichv}{{\tilde{\mathrm{ch}}}}
\newcommand{\Fic}{\mathbb{F}_{\mathrm{ic}}}

\newcommand{\FS}{\mathbb{F}_{\mathrm{step}}} 

\newcommand{\FU}{\mathbb{F}_{\mathrm{uniform}}}
\newcommand{\FUS}{\mathbb{F}_{\mathrm{us}}}
\newcommand{\us}{\mathrm{us}}

\newcommand{\conf}{\mathcal{X}}

\newcommand{\rr}{\mathbbm{r}_0}

\newcommand{\roots}{\mathcal{S}}
\newcommand{\rootsR}{\mathcal{R}}
\newcommand{\rootsL}{\mathcal{L}}

\newcommand{\inodesR}{\mathrm{R}}
\newcommand{\inodesL}{\mathrm{L}}

\newcommand{\SU}{\mathcal{U}}

\newcommand{\oout}{{\rm out}}

\newcommand{\upp}{{\rm up}}
\newcommand{\down}{{\rm down}}

\newcommand{\LL}{{\rm L}}
\newcommand{\RR}{{\rm R}}

\newcommand{\mb}{\mathbf}
\newcommand{\ds}{\displaystyle}

\numberwithin{equation}{section} 
\numberwithin{thm}{section}

\newcommand{\Z}{\mathbb{Z}}
\newcommand{\R}{\mathbb{R}}
\newcommand{\C}{\mathbb{C}}

\newcommand{\step}{\mathrm{step}}
\renewcommand{\flat}{\mathrm{flat}}
\newcommand{\stof}{\mathrm{sf}} 

\newcommand{\scrC}{\mathscr{C}}
\newcommand{\scrD}{\mathscr{D}}
\newcommand{\scrS}{\mathscr{S}}
\newcommand{\scrK}{\mathscr{K}}

\newcommand{\bz}{\boldsymbol{z}}
\newcommand{\bn}{\boldsymbol{n}}

\newcommand{\caC}{\mathcal{C}}
\newcommand{\caD}{\mathcal{D}}

\newcommand{\lemS}{S}
\newcommand{\lemSL}{L}
\newcommand{\lemSR}{R}
\newcommand{\lemK}{K}
\newcommand{\lemf}{\mathrm{p}}
\newcommand{\lemF}{\mathrm{P}}
\newcommand{\lemG}{\mathrm{Q}}

\newcommand{\lemg}{\mathrm{q}}
\newcommand{\lemh}{\mathrm{h}}

\newcommand{\lemE}{\mathcal{G}}
\newcommand{\lemFS}{\mathbf{S}} 
\newcommand{\lemBlk}{B}

\newcommand{\lemq}{r}

\renewcommand{\Re}{\mathrm{Re}}
\renewcommand{\Im}{\mathrm{Im}}

\newcommand{\bx}{\mathbf{x}} 
\newcommand{\height}{\mathbf{h}} 
\DeclareMathOperator{\PTASEP}{PTASEP}
\newcommand{\scrCstep}{\mathscr{C}_{\step}}
\newcommand{\scrDstep}{\mathscr{D}_{\step}}
\newcommand{\scrKstep}{\mathscr{K}^{\step}}
\newcommand{\scrKone}{\mathscr{K}_1^Y}
\newcommand{\scrKtwo}{\mathscr{K}_2^Y}
\newcommand{\scrKY}{\mathscr{K}^Y}

\newcommand{\limKstep}{\mathrm{K}^{\step}}
\newcommand{\limKstepo}{\mathrm{K}^{\step}_1}
\newcommand{\limKstept}{\mathrm{K}^{\step}_2}
\newcommand{\limC}{\mathrm{C}_{\mathrm{ic}}}
\newcommand{\limCstep}{\mathrm{C}_{\step}}

\newcommand{\limDstep}{\mathrm{D}_{\step}}
\newcommand{\limz}{\mathbf{z}}
\newcommand{\limSo}{\mathrm{S}_1}
\newcommand{\limSt}{\mathrm{S}_2}
\newcommand{\lz}{\mathrm{z}}

\newcommand{\sfrr}{\rr^\stof}

\newcommand{\newz}{\mathfrak{s}} 

\newcommand{\Omegac}{\mathrm{D}}  

\author{Jinho Baik\footnote{Department of Mathematics, University of Michigan,
Ann Arbor, MI, 48109. Email: \texttt{baik@umich.edu}} 
 and Zhipeng Liu\footnote{Department of Mathematics, University of Kansas, Lawrence, KS 66045. Email: \texttt{zhipeng@ku.edu}}}
\date{\today}

\begin{document}

\title{Periodic TASEP with general initial conditions}

\date{\today}

\maketitle
	
\begin{abstract}
	We consider the one-dimensional totally asymmetric simple exclusion process with an arbitrary initial condition in a spatially periodic domain, and obtain explicit formulas for the multi-point distributions in the space-time plane.
	The formulas are given in terms of an integral involving a Fredholm determinant. We then evaluate the large-time, large-period limit in the relaxation time scale, which is the scale such that the system size starts to affect the height fluctuations. 
	The limit is obtained assuming certain conditions on the initial condition, which we show that the step, flat, and step-flat initial conditions satisfy. 
	Hence, we obtain the limit theorem for these three initial conditions in the periodic model, extending the previous work on the step initial condition. 
	We also consider uniform random and uniform-step random initial conditions. 
\end{abstract}

\setcounter{tocdepth}{1}  
\tableofcontents

\section{Introduction}

In the last two decades, many limit theorems for the height function of the interacting particle systems in the Kardar-Parisi-Zhang (KPZ) universality class were established when the domain is infinite  \cite{Baik-Deift-Johansson99,Johansson00,Johansson03,Borodin-Ferrari-Prahofer-Sasamoto07,Tracy-Widom08,Tracy-Widom09,Amir-Corwin-Quastel11,Borodin-Corwin14,Matetski-Quastel-Remenik17,JohanssonTwoTime,Dauvergne-Ortmann-Virag18,Johansson-Rahman19,Liu19} or half-infinite \cite{Baik-Rains01,Sasamoto-Imamura04,Baik-Barraquand-Corwin-Suidan18,Barraquand-Borodin-Corwin-Wheeler18}. 
Recently, similar theorems were also beginning to be obtained for the case of the periodic domain.  
When the domain is periodic, all particles are strongly correlated if the period is too small compared with the time. Hence, it is natural to study the situation when the time $t$ and the period $L$ both become large and all particles are critically correlated. 
The critical case occurs when $t=O(L^{3/2})$, which is called the \emph{relaxation time scale}.
In this scale, the spatial correlation length and the period of the domain  are of the same order of magnitude.

One of the fundamental models in the KPZ universality class is the totally asymmetric simple exclusion process.
We call the process in a periodic domain  (equivalently the spatially periodic process) the \emph{periodic TASEP} or \emph{PTASEP} for short. 
On the other hand, we call the process on the infinite domain simply the \emph{TASEP}.
For the PTASEP, a few limit theorems  in the relaxation time scale were obtained in the past three years.\footnote{See \cite{Gwa-SpohnBethe}, \cite{Derrida-Lebowitz98}, \cite{Priezzhev2003}, \cite{Golinelli-Mallick04}, \cite{Golinelli-Mallick05}, \cite{Brankov-Papoyan-Poghosyan-Priezzhev06} for the earlier work for the other properties of periodic models.}
The papers \cite{Prolhac16}, \cite{Baik-Liu16}, \cite{Liu16} 
evaluated the limit of the one-point distribution function for three initial conditions; step, flat, and uniformly random initial conditions. For the step initial condition, the result was further extended to multi-point distributions at non-equal times in \cite{Baik-Liu17}. 

The goal of this paper is to evaluate the multi-point distributions for other initial conditions.
The main results are the following.

\begin{enumerate}[1)]
	\item For the PTASEP with an arbitrary initial condition, we obtain an explicit formula for the multi-point distribution at arbitrary finite times; see Theorem \ref{thm:Fredholm}.  
	The formula is given in terms of an integral involving a Fredholm determinant. 
	
	\item We evaluate the large time limit of the multi-point distribution in the relaxation time scale when the initial condition satisfies certain assumptions; see Theorem \ref{thm:main}. 
	
	\item We show that the flat and step-flat initial conditions satisfy the assumptions mentioned above. Hence, we obtain limit theorems for the multi-point distribution for these two initial conditions; see Theorem \ref{thm:special_IC}.  
	
	\item We also obtain similar results for two random initial conditions: uniformly random initial conditions and uniform-step random initial conditions. 
	See Theorems \ref{thm:Fredholm_Uniform_IC} and \ref{thm:Fredholm_RandomIC} for finite-time formulas and Theorem \ref{thm:Uniform} and \ref{thm:step_uniform} for limit theorems.  
\end{enumerate}

Two key features are:  
(a) we evaluate multi-point distributions at non-equal times (i.e., multi-time distributions), and 
(b) we consider general initial conditions. 

\bigskip

Unlike the one-point distribution and the multi-point distribution in the spatial direction, the multi-time distributions were obtained only recently even for the infinite domain case.  
For the models in the infinite domain, Johansson computed the two-time distribution for a directed last passage percolation model in \cite{JohanssonTwoTime} and for the discrete-time TASEP in \cite{Johansson18}. 
Most recently, in the span of about one month, the multi-time distribution was evaluated in \cite{Johansson-Rahman19} for a directed last passage percolation model and, independently in \cite{Liu19} for the (continuous-time) TASEP. 
The work  \cite{Johansson-Rahman19} was for the step initial condition while the work \cite{Liu19} was for both the step or flat initial conditions. 
We mention that the formulas of these two papers are different, and it still remains to show that they are equal.

For the periodic domain, the multi-time distribution was computed in \cite{Baik-Liu17} for the PTASEP with the step initial condition. 
This paper extends the work of \cite{Baik-Liu17} to other initial conditions. The limiting space-time fluctuation field of the height function in the relaxation time limit, which is the periodic analog of the KPZ fixed point, depends on the initial condition. In this paper, we describe this dependence by studying the general initial condition of the PTASEP. We prove the convergence of the multi-point distribution under certain assumptions on the initial conditions, which we verify for two specific initial conditions, namely the flat and step-flat initial conditions.

\bigskip

The limits in this paper are taken in the relaxation time scale; $t, L\to \infty$ while $\tau= tL^{-3/2}$ is finite.
Hence, the limiting multi-time distributions depend on the relaxation parameter, $\tau$. 
Since the PTASEP becomes the TASEP if $L\to \infty$ with $t$ kept finite, we expect that 
limiting multi-time distributions of the PTASEP converges to limiting multi-time distributions of the TASEP if we take $\tau\to 0$. 
Unfortunately, this limit is delicate to take, and we were not able to compute the limit in this paper. 
However, in \cite{Liu19}, the author re-expressed 
the finite-time formula of this paper 
when $L$ is larger than some fixed number, which is equivalent of taking 
$L\to \infty$ with $t$ fixed. 
The computation involves a technical calculation, which shows that many terms of the expansion of a Fredholm determinant cancel out upon integration, and the remaining terms still sum up to a Fredholm determinant on a different space. 
As a result, the paper obtained a finite-time formula for multi-time distributions for the TASEP. 
By taking the large time limit of this formula, the author obtained the limiting multi-time distribution of the TASEP on the infinite domain. 
Adapting (and simplifying) the method of \cite{Liu19} directly to the $\tau\to 0$ limit of the limiting multi-point distribution is left as a future project. 

We remark that the limit of the multi-point distribution obtained in \cite{Baik-Liu17} and this paper should be the finite-dimensional distributions for the universal field for the models in the KPZ universality class when we consider them in periodic domains. 

\bigskip

We now explain how we prove the results of this paper. 
For the TASEP (on the infinite domain), one way of computing the one-point distribution is the following. 
One first computes the transition probabilities using the coordinate Bethe ansatz method.
For the TASEP, this was obtained by Sch\"utz in \cite{Schutz97}. 
The one-point distribution is a sum of the transition probabilities. This sum was simplified by  \cite{Rakos-Schutz05}, which re-derived the  the Fredholm determinant formula obtained previously by \cite{Johansson00} using a different method. 
The Fredholm determinant formula is suitable for the asymptotic analysis, and one can find a large time limit from it. 
This method was generalized to evaluate multi-point distributions in spatial directions in \cite{Borodin-Ferrari-Prahofer-Sasamoto07} for several classical initial conditions. 
The asymptotic analysis was further extended to general initial conditions in \cite{Matetski-Quastel-Remenik17}.

For the PTASEP, we computed the transition probabilities in \cite{Baik-Liu16} and evaluated the one-point distribution using them. 
The multi-point distribution is a multi-sum involving the transition probabilities. 
The difficult task is to simplify the sum to a formula suitable for the asymptotic analysis. 
In \cite{Baik-Liu17}, we found a formula for the multi-point distribution (in the space-time coordinates) in terms of an integral involving a certain determinant. This formula was obtained for general initial conditions. The involved determinant is not of the Fredholm type but instead an extension of a Toeplitz/Hankel determinant, which is not easy to analyze asymptotically. 
For the special case of the step initial condition, we proved 
an algebraic identity which relates the Toeplitz-like determinant to a Fredholm determinant, which was then analyzed asymptotically. 
One of the main technical results of this paper is an extension of this identity between Toeplitz-like determinants and Fredholm determinants for general initial conditions. This result is given in Proposition \ref{lm:key_lm}.
As a consequence, we obtain a formula for the general initial conditions in terms of an integral involving a Fredholm determinant. 
The formula does not change much from the step initial condition case. 
Indeed, the information on the initial condition appears in only two explicit factors. 
Hence, the asymptotic analysis for the step initial condition goes through without any changes if we assume that the new factors satisfy certain assumptions. However, checking the assumptions highly depends on the initial condition. For the flat initial condition, it is relatively easy to check that the assumptions are satisfied. The step-flat initial condition is more difficult to check, and this is another technical part of this paper. 

We also consider two types of random initial conditions. 
Since we already have a finite-time formula for the multi-point distribution for general initial conditions, the result for random initial conditions is obtained by taking a weighted sum. 
For uniformly random initial condition and a partially uniformly random initial conditions, the weighted sum becomes simple
and we can do asymptotic analysis.

\bigskip

This paper is organized as follows. 
The PTASEP is defined in Section \ref{sec:modeldef}. 
We state the results for finite-time distributions in Section \ref{sec:Fredholm_finite_time}. 
A general identity between a Toeplitz-like determinant and a Fredholm determinant is discussed in Section \ref{sec:proof_lemma}. Using this identity, we prove the finite-time distribution formulas in Section \ref{sec:proofofal}.
The limit theorem for the multi-point distribution for a general initial condition is given in Section \ref{eq:limittheorem}.
We then discuss two special initial conditions, flat and step-flat initial conditions. 
We mention a few properties of these initial conditions in Section \ref{sec:flatandstepflat} and then in Section \ref{sec:asymptotis_products}, \ref{sec:stepflatassumtions} and \ref{sec:proof_special_IC}  we show that these initial conditions satisfy the assumptions for the general limit theorem. 
Section \ref{sec:random_IC} and \ref{sec:randomIClimit} are about the random initial conditions. 
Finally, we discuss a certain case of step-flat initial condition from a probability perspective (instead of algebraic/analytic perspective) in Appendix \ref{sec:prob_step_flat}.

\subsubsection*{Acknowledgments}
The work of Jinho Baik was supported in part by NSF grant DMS-1664531  and DMS-1664692. The work of Zhipeng Liu was supported by the University of Kansas Start Up Grant, the University of Kansas New Faculty General Research Fund, and Simons Collaboration Grant No. 637861.

\section{PTASEP} \label{sec:modeldef}

Let $N<L$ be two positive integers. 
$L$ denotes the period and $N$ denotes the number of particles per period. 
We label particles from left to right and let $\bx_k(t)$ be the location of the particle with label $k$ at time $t$ so that $\cdots< \bx_1(t)< \bx_2(t)<\cdots$. 
The PTASEP of period $L$ with $N$ particles per period is the interacting particle system following the usual TASEP rule (in which the particles move to the right) except that there is an additional periodicity condition: 
\beqq
\bx_{k+N}(t)= \bx_k(t)+L \quad \text{for all $k$ and $t$.}
\eeqq
Define the set
\begin{equation*}
\conf_N(L):=\{(\alpha_1,\cdots, \alpha_N)\in\intZ^N: \alpha_1<\cdots< \alpha_N< \alpha_1+L\}.
\end{equation*}

\begin{defn} \label{def:PTASEP}
	Let $N<L$ be two positive integers. Let 
	\begin{equation*} 
	Y = (y_1,\cdots,y_N) \quad \text{with $Y \in \conf_N(L)$.} 
	\end{equation*}
	We denote by $\PTASEP(L,N,Y)$ the PTASEP of period $L$ with $N$ particles per period which started with the initial condition 
	\begin{equation*}
	\bx_{k+\ell N}(0) = y_k +\ell L \quad \text{for all $k, \ell \in \intZ$.}
	\end{equation*}
	We use the notation $\PTASEP(L,N)$ when we do not specify the initial condition. 
\end{defn}

$\PTASEP(L,N)$ is equivalent to the TASEP on a ring of size $L$ with $N$ particles if we keep track of the winding numbers on the ring.

\bigskip

We are interested in the joint distribution of the particle locations $\bx_{k_i}(t_i)$ for $1\le i\le m$ where $t_i$ are allowed to be different. 
The PTASEP can also be described by the height function $\height(\ell, t)$. 
The height function $\height(\ell,t)$  for $(\ell, t)\in \Z\times \R_+$ is defined by
\begin{equation}
\label{eq:height}
\height(\ell,t):= 
\begin{dcases}
2J_0(t) + \sum_{j=1}^\ell (1- 2\eta_j(t)), & \ell \ge 1,\\
2 J_0(t), & \ell=0,\\
2 J_0(t) -\sum_{j=\ell+1}^0 (1-2\eta_j(t)), &\ell\le -1,
\end{dcases}
\end{equation}
where $J_0(t)$ is the number of particles that pass $0$ to $1$ during the time interval $[0,t]$, and $\eta_j(t)$ is the occupation function defined by $\eta_j(t)=1$ if the site $j$ is occupied at time $t$ or $\eta_j(t)=0$ if it is unoccupied at time $t$.
The joint distribution of the particle locations is equivalent to the finite-dimensional distribution of the two-dimensional field of the height function. 
Here we used the convention that $\height(0,0)=0$. 
The height function is related to the particle locations in a simple way. 
The following fact, which holds for both TASEP and PTASEP, is well-known and it is easy to check. 

\begin{lm}[Height function and particle locations]
	Let $K$ be the integer such that $\bx_K(0)\le 0<\bx_{K+1}(0)$. Then, 
	\begin{equation}
	\label{eq:relation_height_particle_locations}
	\height(\ell,t)\ge b \quad \text{if and only if}\quad \bx_{K-\frac{b-\ell}{2}+1}(t)\ge \ell+1
	\end{equation}
	for all $b,\ell\in\intZ$ satisfying $b-\ell\in2\intZ$  and $b\ge \height(\ell,0)$.
\end{lm}

The periodicity of the PTASEP implies that the height function satisfies 
\beqq
\height(\ell+L,t)=\height(\ell,t) + (L-2N) \qquad \text{for $\ell\in\intZ$ and $t\ge 0$.}
\eeqq

\bigskip

We will consider the large time, large period limit of the two-dimensional field $\height(\ell, t)$. 
Since we take $L\to \infty$, we consider a sequence of PTASEP parameterized by $L$. 
For each $L$,  let $N=N_L$ be the number of particles per period and we assume that the average density 
\begin{equation*}
\rho= \rho_L:= \frac{N}{L}
\end{equation*}
stays in a compact subset of $(0,1)$ for all $L$. 
We consider the multi-point distribution of the height function at points $(\ell_i, t_i)$ for $1\le i\le m$, where we take the scale that 
\beqq
t_i = O(L^{3/2}) \quad \text{and} \quad \ell_i= O(L).
\eeqq
The time scale $t_i=O(L^{3/2})$ is the relaxation time scale at which all locations are non-trivially correlated. 
We then consider the fluctuation of the height function under the scale $L^{1/2}$.  
Note that time, location, and height scale as $L^{3/2}$, $L$, and $L^{1/2}$, which is consistent with the KPZ exponent $3:2:1$.

\bigskip

As for the initial condition, we take a sequence
\beqq
Y_L=(y_1^{(L)}, \cdots, y_N^{(L)}) \quad \text{with $Y_L \in \conf_N(L)$}
\eeqq
and consider the sequence $\PTASEP(L, N_L, Y_L)$.
Naturally, we need to impose conditions for $Y_L$ so that the law of the height function converges. 
We will discuss a sufficient condition later. 
Three important examples are the following. 

\begin{defn} \label{def:special_IC}
	The step, flat, and step-flat initial conditions for $\PTASEP(L,N)$ are the following. 
	\begin{enumerate}[(i)]
		\item (Step) Set $y_i= i-N$ for $1\le i\le N$.
		\item (Flat) Assume that $L=dN$ for an integer $d\ge 2$ and set $y_i=(i-N)d$ for $i=1,\cdots, N$. 
		\item (Step-flat) Assume $L=dN + L_s$ for integers $d\ge 2$ and $0<L_s<L$, and set $y_i=(i-N)d$ for $i=1,\cdots, N$.
	\end{enumerate}
\end{defn}

See Figure~\ref{fig:ICpicture}. See also Figure~\ref{fig:step_flat_IC} for the associated height function.

\begin{figure}
	\centering
	\begin{tikzpicture}[scale=0.45]
	\draw[thick] (-4.5, 0) -- (17.5, 0);
	\foreach \x in {-4, ..., 17} {
		\draw (\x, 0.15) -- (\x, -0.15) node[anchor=north]{\footnotesize $ $};
	}
	\fill (-3, 0) circle(0.2);
	\fill (-2, 0) circle(0.2);
	\fill (-1, 0) circle(0.2);
	\fill (0, 0) circle(0.2);
	\fill (13, 0) circle(0.2);
	\fill (14, 0) circle(0.2);
	\fill (15, 0) circle(0.2);
	\fill (16, 0) circle(0.2);
	\draw[thin, -] (0.6, 0.7)--(0.6, -0.7)-- (16.4, -0.7) -- (16.4, 0.7)--(0.6, 0.7);
	\draw[thin, -] (17.5,0.7)--(16.6,0.7)--(16.6, -0.7)-- (17.5, -0.7);
	\draw[thin, -] (-4.5,0.7)--(0.4,0.7)--(0.4, -0.7)-- (-4.5, -0.7);
	\end{tikzpicture} 
	\vspace{0.3cm}
	\centering \vspace{0.3cm}
	\begin{tikzpicture}[scale=0.45]
	\draw[thick] (-12.5, 0) -- (6.5, 0);
	\foreach \x in {-12, ..., 6} {
		\draw (\x, 0.15) -- (\x, -0.15) node[anchor=north]{\footnotesize $ $};
	}
	\fill (-12, 0) circle(0.2);
	\fill (-9, 0) circle(0.2);
	\fill (-6, 0) circle(0.2);
	\fill (-3, 0) circle(0.2);
	\fill (0, 0) circle(0.2);
	\fill (3, 0) circle(0.2);
	\fill (6, 0) circle(0.2);
	\draw[thin, -] (-11.4, 0.7)--(-11.4, -0.7)-- (0.4, -0.7) -- (0.4, 0.7)--(-11.4, 0.7);
	\draw[thin, -] (6.5,0.7)--(0.6,0.7)--(0.6, -0.7)-- (6.5, -0.7);
	\draw[thin, -] (-12.5,0.7)--(-11.6,0.7)--(-11.6, -0.7)-- (-12.5, -0.7);
	\end{tikzpicture}
	\vspace{0.3cm}
	\centering \vspace{0.3cm}
	\begin{tikzpicture}[scale=0.45]
	\draw[thick] (-16.5, 0) -- (10.5, 0);
	\foreach \x in {-16, ..., 10} {
		\draw (\x, 0.15) -- (\x, -0.15) node[anchor=north]{\footnotesize $ $};
	}
	\fill (-16, 0) circle(0.2);
	\fill (-9, 0) circle(0.2);
	\fill (-6, 0) circle(0.2);
	\fill (-3, 0) circle(0.2);
	\fill (0, 0) circle(0.2);
	\fill (7, 0) circle(0.2);
	\fill (10, 0) circle(0.2);
	\draw[thin, -] (-15.4, 0.7)--(-15.4, -0.7)-- (0.4, -0.7) -- (0.4, 0.7)--(-15.4, 0.7);
	\draw[thin, -] (10.5,0.7)--(0.6,0.7)--(0.6, -0.7)-- (10.5, -0.7);
	\draw[thin, -] (-16.5,0.7)--(-15.6,0.7)--(-15.6, -0.7)-- (-16.5, -0.7);
	\end{tikzpicture}
	\caption{Examples of step, flat, and step-flat initial conditions. A rectangle represents one period.}
	\label{fig:ICpicture}
\end{figure}
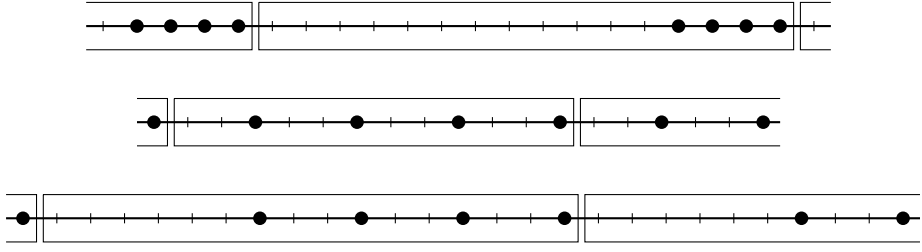

\begin{figure}
	\centering
	\includegraphics[scale=0.32]{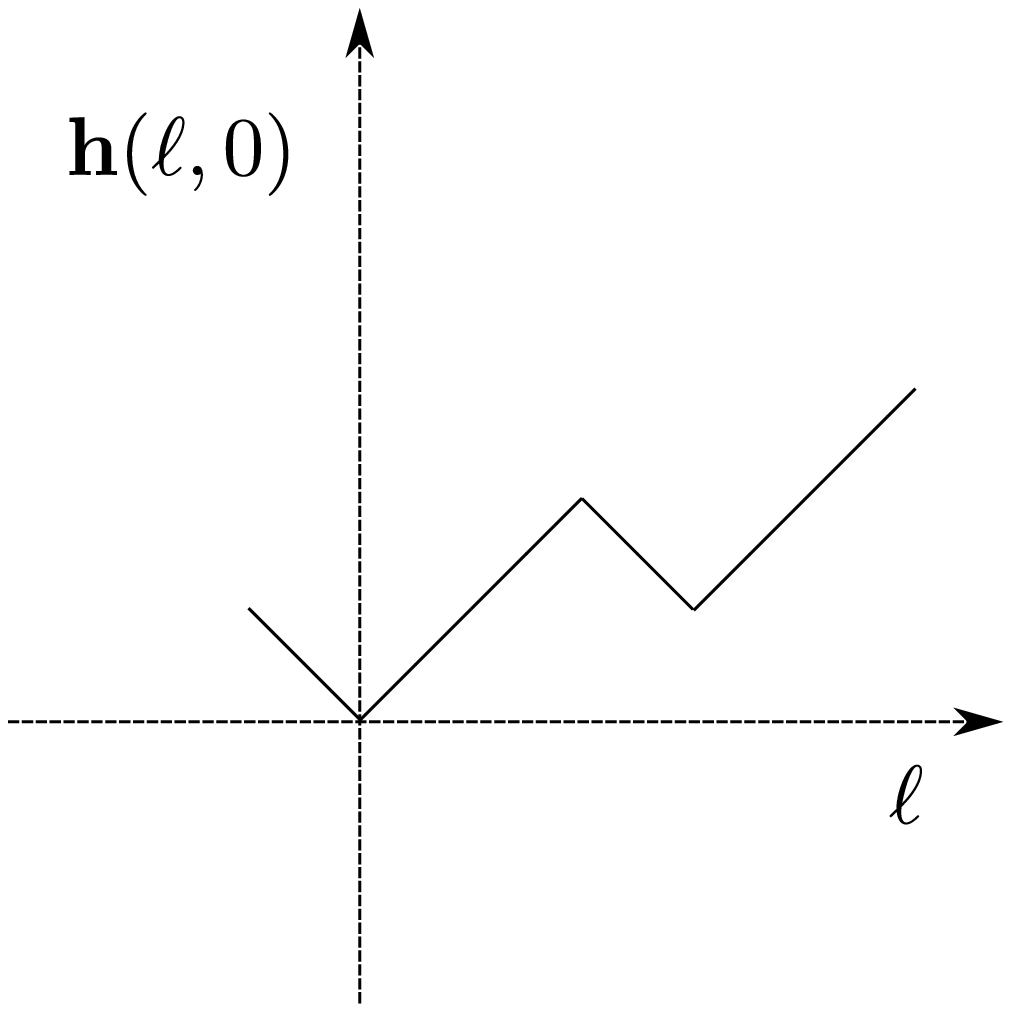}\quad
	\includegraphics[scale=0.32]{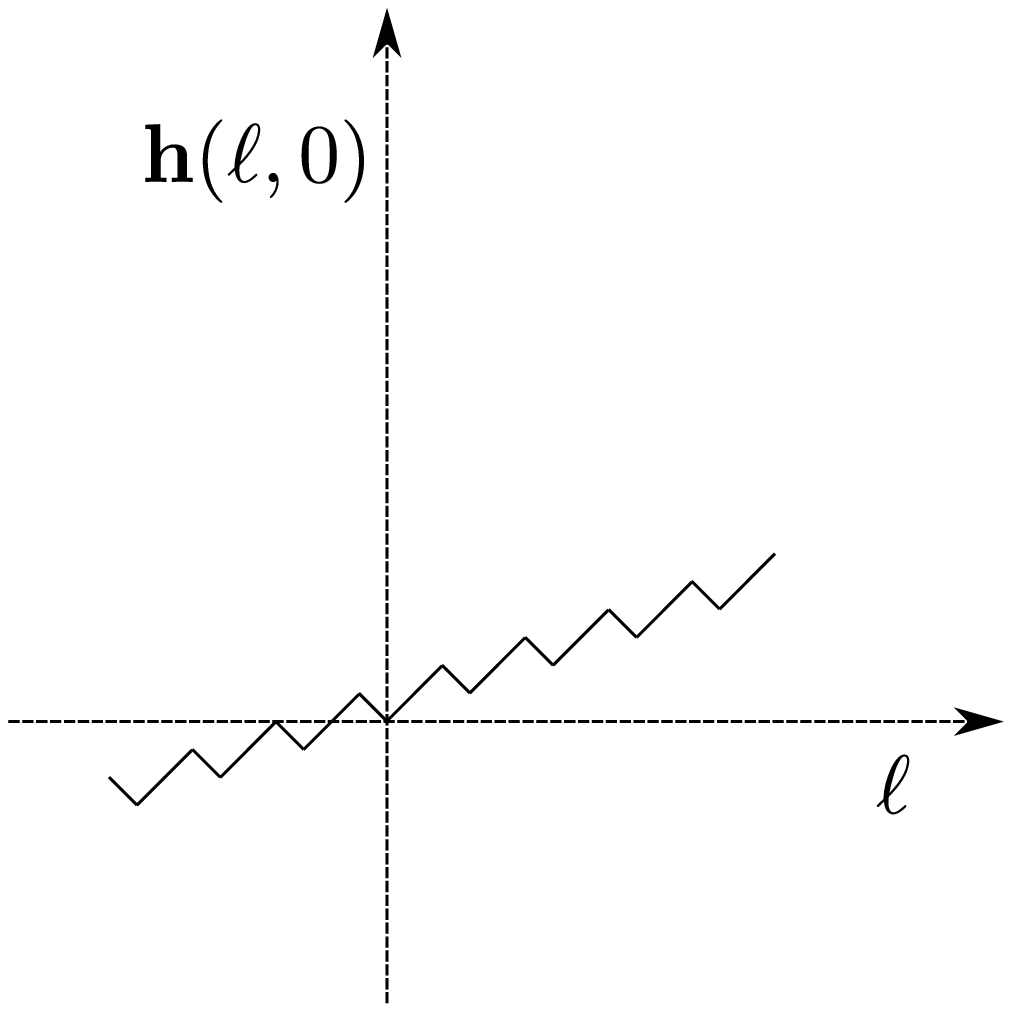}\quad\includegraphics[scale=0.32]{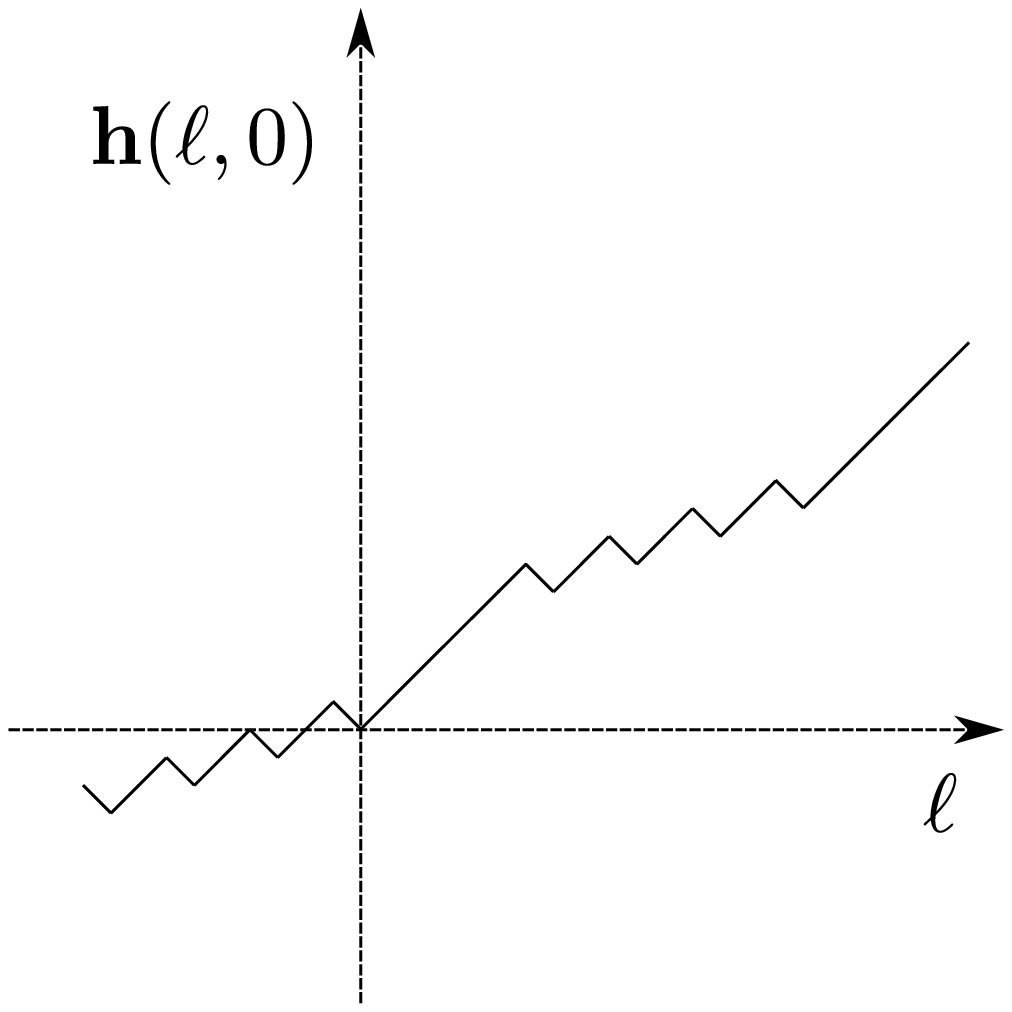}
	\caption{The height function for the step, flat, and step-flat initial conditions. We use $N=4, L=16$ (step), $N=4, d=3, L=12$ (flat), and $N=4, d=3, L=16$ (step-flat). The graphs include two periods. 
	}
	\label{fig:step_flat_IC}
\end{figure}

\section{Multi-point distributions}
\label{sec:Fredholm_finite_time}

\subsection{Finite-time formula for multi-point distribution}
\label{sec:Fredholm_deterministic_finite_time}

The following is the first main result of this paper. 
We obtain a finite-time formula for the multi-point distribution for an arbitrary initial condition $Y$. 
The proof is given in Section \ref{sec:proofofal}.

\begin{thm}[Finite-time formula] \label{thm:Fredholm}
	Let $N<L$ be positive integers. 
	Consider $\PTASEP(L,N, Y)$ for an arbitrary initial condition $Y\in \conf_N(L)$. 
	Set
	\begin{equation} \label{eq:def_rr}
	\rr = \rho^\rho(1-\rho)^{1-\rho}, \qquad  \text{where $\rho= \frac{N}{L}$.}
	\end{equation}
	Fix a positive integer $m$, and let $(k_i,t_i)$, $1\le i\le m$, be $m$ distinct points in $\intZ\times[0,\infty)$ 
	satisfying $0< t_1\le \cdots\le t_m$. 
	Then, for arbitrary integers $a_1,\cdots,a_m$, 
	\begin{equation} \label{eq:multipoint_finite_time}
	\prob_Y\left( \bigcap_{\ell=1}^m \left\{ \bx_{k_\ell}(t_\ell) \ge a_\ell \right\}\right) 
	= \oint\cdots\oint \scrC_Y(\textbf{z}) \scrD_Y(\textbf{z}) \ddbar{z_1}{z_1}\cdots \ddbar{z_m}{z_m},
	\end{equation}
	where the contours are nested circles centered at the origin satisfying $0<|z_m|<\cdots<|z_1|<\rr$.
	Here we set $\textbf{z}=(z_1,\cdots,z_m)$.
	The function $\scrC_Y(\textbf{z})$ is defined in \eqref{eq:def_C0}. 
	The function $\scrD_Y(\textbf{z})$ is a Fredholm determinant 
	\beqq
	\scrD_Y(\textbf{z})= \det ( 1- \scrKY), 
	\eeqq 
	where $\scrKY$ is defined in \eqref{eq:def_D0}. 
\end{thm}

The product $\scrC_Y(\textbf{z}) \scrD_Y(\textbf{z})$ is an analytic function in the domain $0<|z_m|<\cdots<|z_1|<\rr$: see
Lemma \ref{lem:analyticityofCD} below.

Since the PTASEP satisfies the periodicity $\bx_{k+N}(t)=\bx_k(t)+L$, the finite-time formula should respect this property. 
In addition, we may also re-label the initial condition by specifying any $N$ values $\bx_k(0), \bx_{k+1}(0), \cdots, \bx_{k+N-1}(0)$. 
We discuss the invariance property of the formula under such changes in Subsection \ref{sec:invariancefinite}. 

\bigskip

The above result was obtained for the step initial condition in \cite{Baik-Liu17}. 
We use the following notations. 

\begin{defn}
	\label{def:CDstep}
	Let $\scrCstep(\textbf{z})$ and $\scrDstep(\textbf{z})= \det(1- \scrKstep)$ be 
	$\scrC_Y(\textbf{z})$ and $\scrD_Y(\textbf{z})=\det(1- \scrKY)$ when $Y=(-N+1, -N+2, \cdots, 0)$. 
\end{defn}

The formulas of $\scrCstep(\textbf{z})$ and $\scrDstep(\textbf{z})$ were computed in Theorem 4.6 of \cite{Baik-Liu17}; 
we present these formulas in Subsection \ref{sec:formulaofstepfinite} below.
In the next two subsections, we describe the formulas of $\scrC_Y(\textbf{z})$ and $\scrD_Y(\textbf{z})$ in terms of $\scrCstep(\textbf{z})$ and $\scrDstep(\textbf{z})$. 
We will see that the changes are small and explicit.

\subsection{Bethe roots and a symmetric polynomial} \label{sec:Betheroots}

The following polynomial and its roots, which appear in the finite time distribution formulas of PTASEP using the coordinate Bethe ansatz \cite{Bethe31,Baik-Liu16},  are important in the analysis of PTASEP.  

\begin{defn}[Bethe roots]
	Define the polynomial 
	\begin{equation*}
	q_z(w) := w^N(w+1)^{L-N} - z^L
	\end{equation*}
	for $z\in \C$. 
	We call this function the \emph{Bethe polynomial} corresponding to $z$. 
	Let $\roots_z$ be the set of roots of $q_z$, 
	\beq \label{eq:Betherootsset}
	\roots_z:= \{ w\in \C : q_z(w)=0\}. 
	\eeq
	The elements of $\roots_z$ are called the \emph{Bethe roots} corresponding to $z$.  
\end{defn}

The Bethe roots lie on the level curve $\{ w:  |w^\rho (w+1)^{1-\rho}| = |z|\}$. 
It is easy to check that (see Section 7 of \cite{Baik-Liu16}) the level curve  consists of a single contour when $|z|>\rr$ and of  two disjoint contours when $|z|<\rr$. Here $\rr$ is defined in~\eqref{eq:def_rr}. When $|z|= \rr$, the level curve has a self-intersection at the point $w=-\rho$. See Figure~\ref{fig:level_curves}. 

\begin{figure}
	\centering
	\includegraphics[scale=0.55]{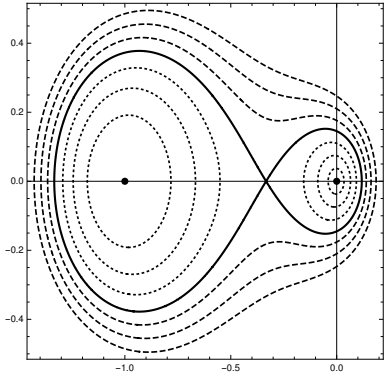}
	\caption{The level curves $|w^{\rho}(w+1)^{1-\rho}|=|z|$ when $\rho=1/3$ for $7$ different values of $|z|$. 
		The dashed, solid, and dotted curves correspond to the cases $|z|>\rr$, $|z|=\rr$ and $|z|<\rr$, respectively.
	}
	\label{fig:level_curves}
\end{figure}

\begin{defn}[Partition of Bethe roots]
	For $|z|<\rr$, define the sets 
	\begin{equation} \label{eq:def_rootsRL}
	\rootsL_z:=\{w\in\roots_z: \Re(w) < -\rho\} \quad \text{and} \quad  \rootsR_z:=\{w\in\roots_z: \Re(w) > -\rho\}.
	\end{equation}
	We call the elements of $\rootsL_z$ the left Bethe roots and the elements of $\rootsR_z$ the right Bethe roots.
	Define the polynomials 
	\begin{equation} \label{eq:def_q}
	q_{z,\LL}(w) := \prod_{u\in\rootsL_z} (w-u) \quad\text{and} \quad   q_{z,\RR}(w) := \prod_{v\in\rootsR_z} (w-v),
	\end{equation}
	which we call the left Bethe polynomials and the right Bethe polynomials, respectively.
\end{defn}

By definition, 
\beqq
\roots_z=\rootsL_z\cup \rootsR_z \quad\text{and} \quad q_z(w) = q_{z,\RR}(w) q_{z,\LL}(w). 
\eeqq
One can check that $\rootsL_z$ has $L-N$ elements and $\rootsR_z$ has $N$ elements for $0<|z|<\rr$. See \cite{Baik-Liu16} and \cite{Baik-Liu17} for more properties of these sets.

\begin{defn}[Symmetric polynomial]
	For
	\beqq
	\text{$\lambda=(\lambda_1,\cdots,\lambda_N)\in\intZ^N$ satisfying $\lambda_1\ge \lambda_2\ge\cdots\ge\lambda_N$,}
	\eeqq
	define the symmetric polynomial 
	\begin{equation}
	\label{eq:def_gftn}
	\gftn_{\lambda}(W) = \frac{\det\left[w_i^{N-j} (w_i+1)^{\lambda_j} \right]_{i,j=1}^N}{ \det\left[ w_i^{N-j}\right]_{i,j=1}^N},
	\end{equation}
	where $W=(w_1, \cdots, w_N)$.
\end{defn}

Since $\gftn_\lambda(W)$ is a symmetric polynomial of $(w_1, \cdots, w_N)$, we can regard $W$ as a set $W=\{w_1,\cdots,w_N\}$ instead of a vector $W=(w_1,\cdots,w_N)$. We interchange the meaning of $W$ freely in this paper. 

\begin{rmk}
	The above symmetric polynomial is related to the dual Grothendieck polynomial $\overline{G}_\lambda$ defined in \cite{Motegi-Sakai13} by the formula 
	\begin{equation*}
	\gftn_\lambda(W) = \left(\prod_{j=1}^N\frac{w_j+1}{w_j}\right)^{N-1} \overline{G}_\lambda(w_1+1,\cdots,w_N+1;-1)
	\end{equation*}
	when $\lambda_N\ge 0$. 
	Another related symmetric function is the inhomogeneous Schur polynomials, 
	\begin{equation*}
	\tilde F_\lambda^{(q=0)}(W) = \frac{\det\left[w_i^{N-j} (w_i-1)^{\lambda_j} \right]_{i,j=1}^N}{ \det\left[ w_i^{N-j}\right]_{i,j=1}^N},
	\end{equation*}
	introduced in \cite{Borodin17}. 
	The function $\tilde F_\lambda^{(q=0)}(W) = (-1)^{\sum_j \lambda_j} \gftn_{\lambda}(-W)$. 
	See \cite{Motegi-Sakai13}, \cite{Borodin17} for algebraic properties such as Cauchy type identities and orthogonality relations of the above symmetric functions.
\end{rmk}

We now introduce two quantities which encode the initial condition in the finite-time formula of the multi-point distribution.

\begin{defn}[Global energy function and characteristic function] \label{defn:energy_ich}
	For $Y\in\conf_N(L)$, let 
	\begin{equation} \label{eq:def_lambdaY}
	\lambda(Y)= (y_{N},y_{N-1}+1,\cdots,y_1+N-1) .
	\end{equation}
	For $|z|<\rr$, we define  the \emph{global energy function} associated to $Y$ by 
	\begin{equation} \label{eq:def_energy}
	\energy_Y (z) := \gftn_{\lambda(Y)} (\rootsR_z) .
	\end{equation}
	When $\energy_Y (z)\neq 0$, we define the \emph{characteristic function} by 
	\begin{equation} \label{eq:def_ich}
	\ich_Y (v,u;z) := \frac{\gftn_{\lambda(Y)}(\rootsR_z\cup \{u\} \setminus \{v\})}{\gftn_{\lambda(Y)} (\rootsR_z)}
	\qquad \text{for  $v\in\rootsR_z$ and $u\in\rootsL_z$}.
	\end{equation}
\end{defn}

Hence, $\energy_Y (z)$ is the symmetric polynomial above corresponding to $\lambda(Y)$ and evaluated at the right Bethe roots. 
The characteristic function involves removing one right Bethe root and replacing it by one left Bethe root. 

Since the Bethe roots are analytic functions of $z$, the function $\energy_Y (z)$ is an analytic function in  $|z|< \rr$. 
As $z\to 0$, all right Bethe roots converge to $0$. 
It is also easy to check that $\gftn_\lambda(W)\to 1$ as $W\to (0,\cdots, 0)$. 
Hence, $\energy_Y (z)$ is an analytic function in $|z|<\rr$ satisfying $\energy_Y (0)=1$. 
As a consequence, $\ich_Y (v,u;z)$ is defined for all but finitely many points of $z$ inside any compact subset of $|z|<\rr$. 
Furthermore, if we take $v$ and $u$ as specific left and right Bethe roots which are continuous in $z$, then $\ich_Y (v,u;z)$ is meromorphic function in $|z|<\rr$.

The above functions become trivial for the step initial condition. 

\begin{lm}
	For the step initial condition, $Y=(-N+1, \cdots, 0)$, we have $\energy_Y (z) =1$ and $\ich_Y (v,u;z) =1$. 
\end{lm}

\begin{proof}
	In this case, $\lambda(Y) = (0, \cdots, 0)$ and the formulas follow easily. 
\end{proof}

Formulas for flat and step-flat initial conditions are given in Section \ref{sec:flatandstepflat}.

\subsection{Definition of $\scrC_Y(\bz)$ and $\scrD_Y(\bz)$}

\begin{defn}[Definition of $\scrC_Y(\bz)$]\label{def:scrC}
	Define 
	\begin{equation} \label{eq:def_C0}
	\scrC_Y(\bz)= \energy_{Y}(z_1) \scrCstep (\bz). 
	\end{equation}
\end{defn}
Recall the definition of $\scrCstep(\bz)$ in~\eqref{def:CDstep}. Its explicit formula is given in Definition~\ref{def:scrCstep}.

The only change from the step initial condition is the explicit factor $\energy_{Y}(z_1)$.
Note that this factor involves only $z_1$, not $z_2, \cdots, z_m$.

\bigskip

For $m$ distinct complex numbers $z_i$ satisfying $|z_i|<\rr$, define the discrete sets 
\begin{equation} \label{eq:def_scrS1}
\scrS_1:= \rootsL_{z_1} \cup \rootsR_{z_2}\cup\rootsL_{z_3}\cup\cdots\cup 
\begin{cases} \rootsL_{z_m}, &\text{if $m$ is odd},\\
\rootsR_{z_m}, &\text{if $m$ is even}, \end{cases}
\end{equation}
and
\begin{equation} \label{eq:def_scrS2}
\scrS_2:= \rootsR_{z_1} \cup \rootsL_{z_2}\cup\rootsR_{z_3}\cup\cdots\cup 
\begin{cases} \rootsR_{z_m}, &\text{if $m$ is odd},\\
\rootsL_{z_m}, &\text{if $m$ is even.} \end{cases}
\end{equation}

\begin{defn}[Definition of $\scrD_Y(\bz)$]\label{def:scrD}
	Let $0<|z_m|<\cdots< |z_1|<\rr$. 
	Assume $\energy_Y(z_1)\ne 0$ so that $\ich_Y(v,u;z_1)$ is well defined.
	Define
	\begin{equation} \label{eq:def_D0}
	\scrD_Y(\boldsymbol{z}) = \det(I - \scrKY) \quad \text{with} \quad \scrKY= \scrKone\scrKtwo, 
	\end{equation}
	where $\scrKone : \ell^2(\scrS_2) \to \ell^2(\scrS_1)$ and $\scrKtwo:\ell^2(\scrS_1) \to \ell^2(\scrS_2)$ have kernels given by
	$\scrKone=\scrKstep_1$ and 
	\begin{equation} \label{eq:aux_2018_04_12_02}
	\scrKtwo(w,w')= \begin{cases}
	\ich_Y (w,w'; z_1) \scrKstep_2(w, w')  \quad & \text{for $w\in \rootsR_{z_1}$ and $w'\in \rootsL_{z_1}$},\\
	\scrKstep_2(w, w')  \quad & \text{otherwise.}	\end{cases}
	\end{equation}
\end{defn}

Recall that $\ich_\step(w,w';z_1)=1$. $\scrD_Y(\bz)$ becomes $\scrDstep(\bz)$ when $Y$ is the step initial condition as in Definition~\ref{def:CDstep}. The explicit formulas of $\scrKstep_1, \scrKstep_2$ and $\scrDstep(\bz)$ are given in Definition~\ref{def:scrDstep}.

Note that the only difference between $\scrD_Y(\bz)$ and $\scrDstep(\bz)$ is in the factor $\ich_Y (w,w'; z_1) $ which depends only on $z_1$, not on $z_2, \cdots, z_m$. 

\bigskip

The above functions satisfy the following analyticity properties. Its proof is given in Remark \ref{rmk:removable_poles} in Section \ref{sec:proofofal}.

\begin{lm}\label{lem:analyticityofCD}
	The function $\scrC_Y(\bz)$ is analytic and $\scrD_Y(\bz)$ is meromorphic in $0<|z_m|<\cdots<|z_1|<\rr$. 
	The product  $\scrC_Y(\bz)\scrD_Y(\bz)$ is analytic in the same domain. 
\end{lm}

\subsection{Formula of  $\scrCstep(\boldsymbol{z})$ and $\scrDstep(\boldsymbol{z})$} \label{sec:formulaofstepfinite}

For the completeness, we describe the formulas for the step initial conditions which were obtained in \cite{Baik-Liu17}. 
We start with some notational conventions. 

\begin{notations} \label{notations}
	We write 
	\begin{equation*} 
	f(W) = \prod_{i=1}^n f(w_i)
	\end{equation*}
	for a function $f$ and a finite set $W=\{w_1,\cdots,w_n\}$ or vector $W=(w_1,\cdots,w_n)$. If $n=0$, we set $f(W)=1$. 
	We write
	\begin{equation*}
	\Delta(W; W') = \prod_{i=1}^n \prod_{i'=1}^{n'} (w_i-w'_{i'})
	\end{equation*}
	for two finite sets $W=\{w_1,\cdots,w_n\}$ and $W'=\{w'_{1},\cdots,w'_{n'}\}$, or two finite vectors $W=(w_1,\cdots,w_n)$ and $W'=(w'_{1},\cdots,w'_{n'})$. 
\end{notations}

Recall the definition of the sets $\rootsL_z$ and $\rootsR_z$ in \eqref{eq:def_rootsRL}. 
Recall also the parameters $k_\ell, t_\ell$, and $a_\ell$ in Theorem \ref{thm:Fredholm}. 

\begin{defn}[Formula of  $\scrCstep(\boldsymbol{z})$] \label{def:scrCstep}
	For distinct points $z_j$, $1\le j\le N$, satisfying $0<|z_j|< \rr$, define \begin{equation*}
	\begin{split}
	\scrCstep (\bz)= & \left[ \prod_{\ell=1}^{m} \frac{\mathrm{E}_\ell(z_\ell)}{\mathrm{E}_{\ell-1}(z_\ell)}\right]
	\left[ \prod_{\ell=1}^{m} \frac{\prod_{u\in \rootsL_{z_\ell}} (-u)^N \prod_{v\in\rootsR_{z_\ell}} (v+1)^{L-N}}  {\Delta(\rootsR_{z_\ell}; \rootsL_{z_\ell})}\right]\\
	&\quad \times \left[  \prod_{\ell=2}^{m} \frac{z_{\ell-1}^L}{z_{\ell-1}^L -z_\ell^L}\right]
	\left[ \prod_{\ell=2}^{m} \frac{\Delta(\rootsR_{z_\ell}; \rootsL_{z_{\ell-1}})}
	{\prod_{u\in \rootsL_{z_{\ell-1}}} (-u)^N \prod_{v\in\rootsR_{z_\ell}} (v+1)^{L-N}} \right],
	\end{split}
	\end{equation*}
	where
	\begin{equation*}
	\mathrm{E}_\ell(z):= \prod_{u\in\rootsL_z} (-u)^{k_\ell-N-1} \prod_{v\in\rootsR_z} (v+1)^{-a_\ell+k_\ell-N} e^{t_\ell v}
	\end{equation*}
	for $\ell\ge 1$, and $\mathrm{E}_0(z)=1$.
\end{defn}
We remark that the above formula could be written in a more compact form by canceling some common factors $(-u)^N$ and $(v+1)^{L-N}$. The reason we write in this form is that, as shown in \cite{Baik-Liu17}, the four factors in the brackets converge in the relaxation time scale, and their limits correspond to the factors in the function $\limCstep(\limz)$ defined in~\eqref{def:limSstep}. Hence it is easy to see the limit of $\scrCstep (\bz)$ is $\limCstep(\limz)$. 

\bigskip

The formula of $\scrDstep(\boldsymbol{z})=\det(1-\scrKstep)$ involves several definitions. 
Recall that $\rho=N/L$.
Recall the left and right Bethe polynomials defined in \eqref{eq:def_q}.
Define, for $0< |z|<\rr$,
\begin{equation} \label{eq:def_H}
H_z(w) :=  \begin{cases}
\frac{q_{z,\LL}(w)}{ (w+1)^{L-N}} & \text{for } \mathrm{Re}(w) >-\rho,\\
\frac{q_{z,\RR}(w)}{w^N} & \text{for } \mathrm{Re}(w) <-\rho. \end{cases}
\end{equation}
It is easy to check that $H_z(w)\to 1$ as $z\to 0$ for any fixed $w\ne -1, 0$ since the Bethe roots in $\rootsL_z$ and $\rootsR_z$ converge to $-1$ and $0$ respectively as $z\to 0$, see Figure~\ref{fig:level_curves}. 
We set $H_z(w)=1$ if $z=0$.

Set
\begin{equation*}
J(w) := \frac{w(w+1)}{L(w+\rho)} . 
\end{equation*}
Note that $q_z'(w) = \frac{ w^N(w+1)^{L-N}} {J(w)}$.  
Since $q_z(w) = q_{z,\RR}(w)q_{z,\LL}(w)$ and $q_{z,\RR}(v)=0$ for $v\in\rootsR_{z}$, we find that 
\begin{equation*}
q'_{z,\RR}(v) =  \frac{v^N(v+1)^{L-N}}{J(v) q_{z,\LL}(v)} \quad \text{for $v\in\rootsR_{z}$} .
\end{equation*}

The next functions depend on the parameters $k_\ell, t_\ell, a_\ell$. 
Set 
\begin{equation*}
F_\ell(w) := w^{-k_\ell + N +1} (w+1)^{-a_\ell +k_\ell -N} e^{t_\ell w}
\end{equation*}
for $\ell=1,\cdots,m$, and set $F_0(w):=1$. Define
\begin{equation} \label{eq:def_f}
f_\ell(w):= \begin{dcases}
\frac { F_{\ell-1}(w)} { F_\ell(w) }& \text{for } \textrm{Re}(w) >-\rho,\\
\frac { F_\ell(w) } { F_{\ell-1}(w)}& \text{for }\textrm{Re}(w) <-\rho.
\end{dcases}
\end{equation}

Finally, we set 
\begin{equation*}
Q_1(j):= 1- \left(\frac{z_{j-(-1)^j}}{z_j}\right)^L\quad \text{and}\quad Q_2(j):= 1- \left( \frac{z_{j+(-1)^j}}{z_j}\right)^L
\end{equation*}
for $j=1,\cdots,m$, where we set $z_0=z_{m+1}=0$.

\begin{defn}[Formula of  $\scrDstep(\boldsymbol{z})$]\label{def:scrDstep}
	Recall the sets $\scrS_1$ and $\scrS_2$ defined in \eqref{eq:def_scrS1} and \eqref{eq:def_scrS2}. 
	Let $z_j$, $1\le j\le m$, be distinct points satisfying $0<|z_j|<\rr$. 
	Let 
	\begin{equation*}
	\scrKstep_1: \ell^2(\scrS_2) \to \ell^2(\scrS_1) \quad \text{and} \quad \scrKstep_2:\ell^2(\scrS_1) \to \ell^2(\scrS_2)
	\end{equation*}
	be the operators defined by the kernels
	\begin{equation*}
	\scrKstep_1(w, w') := \left( \delta_i(j) +\delta_i( j + (-1)^i )\right) 
	\frac{ J(w) f_i(w) (H_{z_i}(w))^2 }
	{ H_{z_{i-(-1)^i}}(w) H_{z_{j-(-1)^j}}(w') (w-w')} Q_1(j) 
	\end{equation*}
	and
	\begin{equation*}
	\scrKstep_2(w', w) := \left( \delta_j(i) +\delta_j( i - (-1)^j )\right)  \frac{ J(w') f_j(w') (H_{z_j}(w'))^2 }
	{ H_{z_{j + (-1)^j}}(w') H_{z_{i + (-1)^i}}(w) (w'-w)} Q_2(i)
	\end{equation*}
	for 
	\beqq
	w \in (\rootsL_{z_i} \cup \rootsR_{z_i}) \cap \scrS_1 \quad \text{and} \quad 
	w'\in (\rootsL_{z_j} \cup \rootsR_{z_j}) \cap \scrS_2
	\eeqq
	with $1\le i,j\le m$. Define 
	\begin{equation*} 
	\scrDstep (\boldsymbol{z}) 	= \det( 1- \scrKstep), \quad\text{where} \quad \scrKstep= \scrKstep_1\scrKstep_2.
	\end{equation*}
\end{defn}

\begin{rmk} 
	Note that
	\beqq 
	\scrKstep_1(w,w')=0 \quad \text{unless $(w,w')\in (\rootsL_{z_{2j-1}}\cup \rootsR_{z_{2j}})\times (\rootsL_{z_{2j}}\cup \rootsR_{z_{2j-1}})$}
	\eeqq
	for some $j\ge 1$, and
	\beqq
	\scrKstep_2(w',w)=0 \quad \text{unless $(w',w)\in  (\rootsL_{z_{2j}}\cup \rootsR_{z_{2j+1}}) \times (\rootsL_{z_{2j+1}}\cup \rootsR_{z_{2j}})$}
	\eeqq
	for some $j\ge 0$. Here we set $\rootsL_{z_0}=\rootsR_{z_0}=\emptyset$ and $\rootsL_{z_{m+j}}=\rootsR_{z_{m+j}}=\emptyset$ for all $j\ge 1$.
	The above property implies that the kernels of $\scrKstep_1$ and $\scrKstep_2$ can be written as matrix kernels with $2\times 2$ block structures with possible exceptions at the first/last rows and columns; 
	see Section 2.3 of \cite{Baik-Liu17} for details. 
	The kernels of $\scrKone$ and $\scrKtwo$ also have same structures. 
	The initial condition $Y$  appears only in the first block of $\scrKtwo$, which only involves $\rootsL_{z_1}$ and $\rootsR_{z_1}$.
\end{rmk}

\subsection{Invariance properties of the finite-time distribution formula} \label{sec:invariancefinite}

Note that the probability in the left-hand side of the equation \eqref{eq:multipoint_finite_time} is invariant if we translate the integer sites $\intZ$ by $\intZ+c$ for any integer $c$. More explicitly, the multi-point distribution $\prob_Y\left(\bigcap_{\ell=1}^m \{\bx_{k_\ell}(t)\ge a_\ell\}\right)$ is invariant under the following translation of the parameters:
\begin{enumerate}[(T1)]
	\item $y_i\mapsto y_i+c$ for all $1\le i\le N$ and $a_\ell \mapsto a_\ell + c$ for all $1\le \ell \le m$.
\end{enumerate}

Furthermore, due to the periodicity $\bx_{k+N}(t)= \bx_k(t)+L$, the probability is also unchanged under either of the following translations of the parameters: 

\begin{enumerate}[(T1)]\addtocounter{enumi}{1}
	\item $y_i\mapsto y_i+L$ for all $1\le i\le N$ and $k_\ell \mapsto k_\ell - N$ for all $1\le \ell \le m$.
	\item $a_\ell \mapsto a_\ell + L$ and $k_\ell \mapsto k_\ell +N $ for all $1\le \ell \le m$.
\end{enumerate}

Note the minus sign of the translation $k_\ell \mapsto k_\ell -N$ in (T2).

We can check directly that the finite-time distribution formula in the right-hand side of the equation \eqref{eq:multipoint_finite_time} satisfy these invariance properties. 
Indeed, each of the functions $\scrC_Y(\textbf{z})$ and $\scrD_Y(\textbf{z})$ is invariant under the translations. 
Recall that $\scrC_Y(\textbf{z})$ involves $\energy_Y (z_1)$ which depends on $Y$, and $\scrCstep (\bz)$ which involve $k_i, a_i$'s. Note that in $\scrCstep (\bz)$, the terms $\mathrm{E}_\ell(z)$ are the only ones that depend on the parameters $k_\ell, a_\ell$. From the formulas, we see that 
\beqq 
\energy_Y (z_1)\big|_{Y\mapsto Y+c} = \energy_Y (z_1) \prod_{v\in \rootsR_{z_1}} (v+1)^c,
\eeqq
\beqq
\scrCstep (\bz) \big|_{a_\ell\mapsto a_\ell+c} = \scrCstep (\bz)  \prod_{v\in \rootsR_{z_1}} (v+1)^{-c},
\eeqq
\beqq
\scrCstep (\bz) \big|_{k_\ell\mapsto k_\ell+ N} = \scrCstep (\bz) \prod_{u\in \rootsL_{z_1}} (-u)^{N} \prod_{v\in \rootsR_{z_1}} (v+1)^{N}.
\eeqq
Note that the extra factors depend only on the Bethe roots for $z_1$. The invariance of  $\scrC_Y(\textbf{z})$ under (T1) of is due to a simple cancellation. On the other hand, the invariance under (T2) and (T3) follows from the next simple lemma. 

\begin{lm}[(4.52) of \cite{Baik-Liu17}] \label{lem:Betherootssimpleiden}
	The Bethe roots corresponding to $z$ satisfy
	\beqq
	\prod_{u\in \rootsL_z} (-u)^N= \prod_{v\in \rootsR_z} (v+1)^{L-N}. 
	\eeqq
\end{lm}

\begin{proof}
	Since the Bethe roots are the solutions of the equation $w^N(w+1)^{L-N}-z^L=0$, 
	\beq \label{eq:Beteqrts}
	w^N(w+1)^{L-N}-z^L = \prod_{u\in \rootsL_z} (w-u)\prod_{v\in \rootsR_z} (w-v) .
	\eeq
	Setting $w=0$ and raising the power by $N$, 
	\beqq
	(-1)^N z^{NL} = \prod_{u\in \rootsL_z} (-u)^N\prod_{v\in \rootsR_z} (-v)^N.
	\eeqq
	Since the cardinality of $\rootsR_z$ is $N$, the sign $(-1)^N$ cancels out. 
	On the other hand, since every point $v\in \rootsR_z$ satisfies $z^L= v^N(v+1)^{L-N}$, taking the product over all $v$, 
	\beqq
	z^{NL} = \prod_{v\in \rootsR_z} v^N(v+1)^{L-N}.
	\eeqq
	The last two equations imply the lemma. 
\end{proof}

In $\scrD_Y(\textbf{z})= \det(I - \scrKY)$, the $Y$-dependent term is $\ich_Y (v,u;z_1)$. From the definition, it satisfies $\ich_{Y+c}(v,u;z_1)= \ich_Y(v,u;z_1) \left(\frac{u+1}{v+1}\right)^c$. 
For the $Y$-independent part of the kernel, $f_1(w)$ (see \eqref{eq:def_f}) is the only term which depends on the parameters of interest. 
Since $f_1(w)$ is defined as one way or its reciprocal depending on whether $\textrm{Re}(w) >-\rho$ or $\textrm{Re}(w) <-\rho$, 
(T1) results in a conjugation of the kernel which leaves the Fredholm determinant invariant. For (T3), since $v^N(v+1)^{L-N}=u^N(u+1)^{L-N}$ for all $v\in\rootsL_{z_1}$ and $u\in\rootsR_{z_1}$ , this implies the kernel $\scrKY$ and hence $\scrD_Y(\textbf{z})$ are both invariant. Finally, (T2) is a composition of (T1) with $c=L$ and (T3), thus $\scrD_Y(\textbf{z})$ is still invariant. 

\bigskip

In addition to the above translations, we also have the invariance that re-labeling the indices for the initial condition does not change the probability. 
For example, for a given $1\le n\le N$, if we set $\tilde y_j=y_{j+n}$ for $1\le j\le N-n$ and $\tilde y_j= y_{j-N+n}+L$ for $N-n+1\le j\le N$, and consider the PTASEP with particle locations $\tilde \bx_k(t)$ with the initial condition $\tilde \bx_k(0)= \tilde y_k$, then 
it should be related to the original PTASEP by $\tilde \bx_k(t)= \bx_{k+n}(t)$ for all $k$. 
This means the invariance under the following transformation: 

\begin{enumerate}[(T4)]
	\item Fix $0\le n\le N-1$, and change $(y_1, \cdots, y_N)\mapsto (y_{n+1}, \cdots, y_{N}, y_1+L, \cdots, y_n+L)$ and change $k_\ell\mapsto k_\ell+n$ for all $\ell$. 
\end{enumerate}

The fact that the finite-time formula satisfies this invariance can be checked as follows. 
Set $\tilde{Y}= (y_{n+1}, \cdots, y_{N}, y_1+L, \cdots, y_n+L)$ and consider $\energy_{\tilde Y} (z_1)$, which is given by a determinant. 
Moving the last $N-n$ columns to the front, using the equation $w_i^N(w+1)^{L-N}= z_1^L$ in the last $n$ columns of the new matrix, and taking out common row factors, we find that 
\beqq 
\energy_{\tilde Y} (z_1) = (-1)^{n(N-n)} \energy_{Y} (z_1) z_1^{nL} \prod_{v\in \rootsR_{z_1}} v^{-n}(v+1)^n.
\eeqq
On the other hand, it is straightforward to see that 
\beqq
\scrCstep (\bz) \big|_{k_\ell \mapsto k_\ell+n} = \scrCstep (\bz) \prod_{u\in \rootsL_{z_1}} (-u)^{n} \prod_{v\in \rootsR_{z_1}} (v+1)^{n}. 
\eeqq
The invariance (T4) of $\scrC_Y (\bz)$ follows from the identity $(-1)^{N-1} z^L$ $=$ $\prod_{u\in \rootsL_z} (-u)$ $\cdot\prod_{v\in \rootsR_z} v$, which is obtained by setting $w=0$ in \eqref{eq:Beteqrts}, and the fact that $(-1)^{n^2}=(-1)^n$.

The determinant $\scrD_Y(\textbf{z})= \det(I - \scrKY)$ is invariant under (T4) by the same reason as the invariance for (T1). 

\section{Toeplitz-like determinant and Fredholm determinant} \label{sec:proof_lemma}

In this section, we prove a determinant identity used in the proof of Theorem \ref{thm:Fredholm} in Section \ref{sec:proofofal}. 
As mentioned in the introduction,  a finite-time formula for general initial conditions was obtained in \cite{Baik-Liu17} in terms of a Toeplitz-like determinant; see Theorem \ref{thm:Toeplitzform} below for the statement. 
In  \cite{Baik-Liu17},  we showed that this Toeplitz-like determinant can be converted to a Fredholm determinant for the case of the step initial condition.  
In this section, we prove a general identity between a Toeplitz-like determinant and a Fredholm determinant which is applicable to arbitrary initial conditions. 
The main result is Proposition \ref{lm:key_lm}.

\subsection{A general determinant identity}
\label{sec:key_identity}

Define  the Vandermonde determinant 
\begin{equation*}
\Delta(V):=\prod_{1\le i<j\le n}(v_j-v_i)=\det\left[v_i^{j-1}\right]_{i,j=1}^n 
\end{equation*}
for a vector $V=(v_1,\cdots,v_n)$. 
Note that $\Delta(V)$ depends on the order of coordinates, but $\Delta(V)^2$ does not. 
Hence, 
\begin{equation*}
\Delta(W)^2:=\prod_{\substack{\{w_1,w_2\}\subset W\\ w_1\ne w_2}}(w_1-w_2)^2
\end{equation*}
is well-defined for any finite set $W$. We use $W$ for either a finite set or a vector.

To state the identity between determinants, we introduce the setup. 
The key element is that the underlying sets are discrete. 
These discrete sets do not need to be the roots of an algebraic equation which was the case for the Bethe roots.

Fix positive integers $N$ and $m$. We introduce the following objects. 

\begin{enumerate}[(a)]
	\item Let $\lemS_1,\cdots,\lemS_m$ be finite subsets of $\complexC$ with at least $N$ elements each. The sets are allowed to have different cardinalities.  Assume that $\lemS_i\cap \lemS_{i+1}=\emptyset$ for all $1\le i\le m-1$ if $m\ge 2$. 
	\item For each $1\le i\le m$, let $\lemSR_i$ be a subset of $\lemS_i$ such that $|\lemSR_i|=N$.
	\item Let $\lemSL_i=\lemS_i\setminus\lemSR_i$.
	\item Let $\lemf_1, \cdots, \lemf_N : \lemS_1 \to \C$ be functions on $\lemS_1$ and let $\lemg_1, \cdots, \lemg_N:\lemS_m \to \C$ be functions on $\lemS_m$. 
	\item For each $1\le i\le m$, let $\lemh_i: \lemS_i \to \C$ be a function on $\lemS_i$ such that $\lemh_i(w)\neq 0$ for all $w\in \lemSR_i$. 
\end{enumerate}

Recall the notational convention introduced in Notations~\ref{notations}.

\begin{prop}[Identity between Toeplitz-like determinant and Fredholm determinant] \label{lm:key_lm}
	Define $N\times N$ matrices $T=(T_{ij})_{i,j=1}^N$ and $M=(M_{ij})_{i,j=1}^N$ with entries
	\begin{equation*} 
	T_{ij} = \sum_{\substack{w_1\in \lemS_1\\ \cdots\\ w_m\in\lemS_m}}\frac{\lemf_i(w_1)\lemg_j(w_m)}{\prod_{\ell=2}^{m} (w_\ell-w_{\ell-1})} \prod_{\ell=1}^{m} \lemh_\ell(w_\ell),
	\end{equation*}
	and
	\begin{equation*}
	M_{ij}= \sum_{\substack{v_1\in \lemSR_1\\ \cdots\\ v_m\in\lemSR_m}}\frac{\lemf_i(v_1)\lemg_j(v_m)}{\prod_{\ell=2}^{m} (v_\ell-v_{\ell-1})} \prod_{\ell=1}^{m} \lemh_\ell(v_\ell) .
	\end{equation*}
	Assume that 
	\beqq
	\det\left[ \lemf_i(v_{j}^{(1)})\right]_{i,j=1}^N \det\left[ \lemg_i(v_{j}^{(m)})\right]_{i,j=1}^N\neq 0,
	\eeqq
	where $\lemSR_1=\{v_{1}^{(1)}, \cdots, v_{N}^{(1)}\}$ and $\lemSR_m=\{v_{1}^{(m)}, \cdots, v_{N}^{(m)}\}$.
	Then, 
	\begin{equation} \label{eq:key_identity}
	\begin{split}
	\det\left[ T\right]
	=\det\left[ M \right] \det\left(I -\lemK_1\lemK_2\right),
	\end{split}
	\end{equation}
	where  $K_1$ and $K_2$ are finite matrices defined in~\eqref{eq:def_lemK_1} or~\eqref{eq:def_lemK_2} of the next subsection. 
	Furthermore, the determinant of $M$ has the representation 
	\begin{equation} \label{eq:deofmfo}
	\begin{split}
	\det\left[ M \right]
	= &(-1)^{(m-1)N(N-1)/2}
	\frac{\det[\lemf_i(v_j^{(1)})]_{i,j=1}^N}{\Delta(v_1^{(1)},\cdots,v_N^{(1)})} \frac{\det[\lemg_i(v_j^{(m)})]_{i,j=1}^N}{\Delta(v_1^{(m)},\cdots,v_N^{(m)})}
	\cdot\frac{\prod_{\ell=1}^m \Delta(\lemSR_\ell)^2}{\prod_{\ell=2}^{m}\Delta(\lemSR_\ell;\lemSR_{\ell-1})}\prod_{\ell=1}^m\lemh_\ell(\lemSR_\ell) .
	\end{split}
	\end{equation}
\end{prop}

\begin{rmk}\label{rmk:Toepltiz}
	If $m=1$ and $p_i(x)= x^i$, $q_j(x)=x^{-j}$, then, writing $\lemS_1=\lemS$ and $\lemh_1=\lemh$, we have 
	\beqq
	T_{ij}= \sum_{w\in \lemS} w^{i-j} \lemh(w).
	\eeqq
	Hence, the matrix $T$ is a Toeplitz matrix. Similarly, $M$ is also a Toeplitz matrix in this case. 
\end{rmk}
\begin{rmk}
	For $m>1$, the Cauchy-type structure of the determinant plays an important role for this identity. Such structure allows us to factorize the expansion of the determinant and further regroup the factors in the form of block matrices in $\lemK_1$ and $\lemK_2$. Although it is quite direct to check the identity using Cauchy determinant formula, regrouping the factors properly and reformulating them as block matrices $\lemK_1$ and $\lemK_2$ is constructive.
\end{rmk}

In Remark~\ref{rmk:Toepltiz}, the symbol of the Toeplitz determinant is a discrete measure. 
For the usual case of a continuous symbol, assuming certain regularity of the symbol,  there is a general identity between the Toeplitz determinant and a Fredholm determinant known as the Borodin-Okounkov-Geronimo-Case (BOGC) identity \cite{Case-Geronimo79}, \cite{Borodin-Okounkov00}, \cite{Basor-Widom00}, \cite{Bottcher02}. 
The above identity is different from the BOGC identity; in the BOGC identity, the determinant of $M$ in the above is replaced by the limit of the Toeplitz determinant as the dimension $N\to \infty$ via the strong Szeg\"o limit theorem. 
In the discrete measure case, there are not enough moments of the symbol so that the limit $N\to \infty$ is not applicable. 
We also mention that a different identity between ``discrete Toeplitz determinants" and Fredholm determinants was used in a different situation; see \cite{Baik-Liu14}.

\subsection{Definition of matrices $\lemK_1$ and $\lemK_2$}

Define the discrete sets 
\begin{equation*}
\lemFS_1:=\lemSL_1\cup \lemSR_2\cup \lemSL_3\cup\cdots\cup\begin{dcases}
\lemSL_m,& \text{if $m$ is odd},\\
\lemSR_m,& \text{if $m$ is even},
\end{dcases}
\end{equation*}
and
\begin{equation*}
\lemFS_2:=\lemSR_1\cup \lemSL_2\cup \lemSR_3\cup\cdots\cup\begin{dcases}
\lemSR_m,& \text{if $m$ is odd},\\
\lemSL_m,& \text{if $m$ is even}.
\end{dcases}
\end{equation*}
The rows of the matrices $K_1$ and $K_2$ are indexed by sets $\lemFS_1$ and $\lemFS_2$, respectively, and the columns are indexed by sets  and $\lemFS_2$ and $\lemFS_1$, respectively. 
The matrices $ K_1$ and $K_2$ have the following block structures. For odd $m$, 
\begin{equation} \label{eq:def_lemK_1}
\lemK_1=\begin{bmatrix}
\lemBlk_1&&&\\
&\lemBlk_3&&\\
&&\ddots&\\
&&&\lemBlk_m
\end{bmatrix},\qquad \lemK_2=\begin{bmatrix}
\lemBlk_0&&&\\
&\lemBlk_2&&\\
&&\ddots&\\
&&&\lemBlk_{m-1}
\end{bmatrix},
\end{equation}
and for even $m$, 
\begin{equation}\label{eq:def_lemK_2}
\lemK_1=\begin{bmatrix}
\lemBlk_1&&&\\
&\lemBlk_3&&\\
&&\ddots&\\
&&&\lemBlk_{m-1}
\end{bmatrix},\qquad \lemK_2=\begin{bmatrix}
\lemBlk_0&&&\\
&\lemBlk_2&&\\
&&\ddots&\\
&&&\lemBlk_{m}
\end{bmatrix},
\end{equation}
where the matrix $\lemBlk_k$ is of size $|\lemSL_k\cup \lemSR_{k+1}|\times |\lemSR_k\cup \lemSL_{k+1}|$ with the convention that $\lemSL_0=\lemSR_0=\lemSL_{m+1}=\lemSR_{m+1}=\emptyset$. 

To define the matrix $\lemBlk_k$, we introduce the function
\begin{equation*}
\lemq_\ell (w):=\prod_{v\in\lemSR_\ell}(w-v).
\end{equation*}
We also set 
\begin{equation}
\label{eq:def_lemE}
\lemE(W; \lemF):=\frac{\det\left[\lemf_i(w_j)\right]_{1\le i,j\le N}}{\det\left[w_j^{i-1}\right]_{1\le i,j\le N}}\qquad \text{and}
\qquad 
\lemE(W; \lemG):=\frac{\det\left[\lemg_i(w_j)\right]_{1\le i,j\le N}}{\det\left[w_j^{i-1}\right]_{1\le i,j\le N}}
\end{equation}
for $W=\{w_1,\cdots,w_N\}\subset\complexC$ with distinct elements $w_i\ne w_j$, $(1\le i<j\le N)$, where $\lemF=(\lemf_1,\cdots,\lemf_N)$ and $\lemG=(\lemg_1,\cdots,\lemg_N)$. 
Note that by the assumption of Proposition~\ref{lm:key_lm}, we have $\lemE(\lemSR_1;\lemF)\neq 0$, $\lemE(\lemSR_m;\lemG)\neq 0$, and $\lemq_\ell'(v)\neq 0$ for $v\in \lemSR_\ell$. 
We order the elements of the sets $\lemSL_i$ and $\lemSR_i$, $1\le i\le m$, in an arbitrary way. 
The ordering of the elements does not change the determinant $\det(I-\lemK_1 \lemK_2)$.
The rows of the matrix $\lemBlk_k$ are indexed by $\lemSL_k\cup\lemSR_{k+1}$ and the columns are indexed by $\lemSR_k\cup \lemSL_{k+1}$: 
\begin{equation*}
\begin{split}
\lemBlk_k=& \left[
\begin{array}{ccc|ccc}
&\vdots&&&\vdots&\\
\cdots&\lemBlk_k(u_k,v_k)&\cdots&\cdots&\lemBlk_k(u_k,u_{k+1})&\cdots\\
&\vdots&&&\vdots&\\
\hline
&\vdots&&&\vdots&\\
\cdots&\lemBlk_k(v_{k+1},v_k)&\cdots&\cdots&\lemBlk_k(v_{k+1},u_{k+1})&\cdots\\
&\vdots&&&\vdots&
\end{array}
\right]
\begin{array}{cc}
&\\[2pt]
\leftarrow& \text{row }u_k\in\lemSL_k\\
&\\
&\\[13pt]
\leftarrow& \text{row }v_{k+1}\in\lemSR_{k+1}\\
&
\end{array}\\
&\begin{array}{cc}
\uparrow&  \uparrow\\
\qquad\ \text{column }v_k\in\lemSR_k\qquad\ &\  \text{column }u_{k+1}\in\lemSL_{k+1}\  
\end{array}
\end{split}
\end{equation*}
The entries are defined by, for $1\le k\le m-1$,
\begin{equation}
\label{eq:def_lemBlk2}
\begin{split}
\lemBlk_k(u_k,v_k)
&= 
\lemh_k(u_k)\cdot 
\frac{\lemq_k(u_k)\lemq_{k+1}(v_k)}
{\lemq'_k(v_k)\lemq_{k+1}(u_k)}
\cdot \frac{1}{u_k-v_k}, \\
\lemBlk_k(u_k,u_{k+1})
&= 
\lemh_k(u_k)\cdot 
\frac{\lemq_k(u_k)\lemq_{k+1}(u_{k+1})}
{\lemq_{k+1}(u_k)\lemq_{k}(u_{k+1})} \cdot \frac{1}{u_k-u_{k+1}}, \\
\lemBlk_k(v_{k+1},v_k)
&=  
\frac{1}{\lemh_{k+1}(v_{k+1})}\cdot 
\frac{\lemq_k(v_{k+1})\lemq_{k+1}(v_k)}
{\lemq'_{k+1}(v_{k+1})\lemq'_k(v_k)}
\cdot \frac{1}{v_{k+1}-v_k}, \\
\lemBlk_{k}(v_{k+1},u_{k+1})
&= 
\ds \frac{1}{\lemh_{k+1}(v_{k+1})}\cdot 
\frac{\lemq_{k+1}(u_{k+1})\lemq_{k}(v_{k+1})}
{\lemq'_{k+1}(v_{k+1})\lemq_k(u_{k+1})}
\cdot \frac{1}{v_{k+1}-u_{k+1}} .
\end{split}
\end{equation}
For $k=0$ and $k=m$, the entries of the matrices $\lemBlk_0$ and $\lemBlk_m$ are similar but they contain extra factors. 
The matrix $\lemBlk_0$ has rows indexed by $\lemSR_1$ and columns indexed by $\lemSL_1$, and the entries are defined by 
\begin{equation} \label{eq:def_lemBlk_0}
\lemBlk_0(v_1,u_1) = \frac{\lemE(\lemSR_1\cup \{u_1\} \setminus\{v_1\};\lemF)}{\lemE(\lemSR_1;\lemF)}\cdot  \frac{1}{\lemh_1(v_1)}\cdot \frac{\lemq_1(u_1)}{\lemq_1'(v_1)}\cdot  \frac{1}{v_1-u_1}
\end{equation}
for $v_1\in\lemSR_1$ and $u_1\in\lemSL_1$. Finally, the matrix $\lemBlk_m$ has rows indexed by $\lemSL_m$ and columns indexed by $\lemSR_m$, and the entries are defined by 
\begin{equation} \label{eq:def_lemBlk_m}
\lemBlk_m(u_m,v_m)=\frac{\lemE(\lemSR_m\cup \{u_m\} \setminus \{v_m\};\lemG)}{\lemE(\lemSR_m;\lemG)} \cdot \lemh_m(u_m)\cdot \frac{\lemq_m(u_m)}{\lemq_m'(v_m)}\cdot\frac{1}{u_m-v_m}
\end{equation}
for $u_m\in \lemSL_m$ and $v_m\in\lemSR_m$.

\subsection{Outline of the remainder of the section}

The remaining subsections are devoted to the proof of Proposition \ref{lm:key_lm}. 
We prove the identity by expanding the determinants as sums and then compare the two sides of the equation \eqref{eq:key_identity}. 

\subsection{The left hand side of~\eqref{eq:key_identity}} \label{sec:detLHS}

Applying the Cauchy-Binet formula $m$ times, $\det[T]$ is equal to  
\begin{equation*}
\begin{split}
\frac{1}{(N!)^{m}}
\sum_{\substack{W^{(b)} \in \lemS_b^N \\[3pt] 1\le b\le m}}
&\det\left[\lemf_i(w_j^{(1)})\right] \left( \prod_{\ell=1}^{m-1} \det\left[\frac{1}{w_i^{(\ell)}-w_{j}^{(\ell-1)}} \right]\right)
\cdot\det\left[\lemg_j(w_i^{(m)})\right] \prod_{\substack{1\le i\le N\\ 1\le j\le m}}\lemh_j(w_i^{(j)}).
\end{split}
\end{equation*}
Here, we use the notation $W^{(\ell)}=(w^{(\ell)}_1, \cdots, w^{(\ell)}_N)$ and all determinants are $N\times N$ with row indices denoted by $i$ and column indices denoted by $j$. 
Note that the summand is zero if there is an $\ell$ such that two of the coordinates of $W^{(\ell)}$ are equal.  
Thus, we may assume that for every $\ell$, the coordinates of $W^{(\ell)}$ are all distinct. 
We then note that for each $\ell$, the summand is invariant under any permutation of the coordinates $W^{(\ell)}$.
Hence, we may multiply the formula by $(N!)^m$ and take $W^{(\ell)}$ to be a subset of $\lemS_\ell^N$ of the cardinality $N$.
Evaluating the Cauchy determinants and using the notations \eqref{eq:def_lemE}, the above formula can be written as  
\begin{equation}
\label{eq:lhs_identity1}
\begin{split}
\det[T]=& (-1)^{(m-1)N(N-1)/2} \sum
\lemE(W^{(1)};\lemF)\lemE(W^{(m)};\lemG) \frac{\prod_{\ell=1}^m \Delta(W^{(\ell)})^2}{\prod_{\ell=2}^{m}\Delta(W^{(\ell)};W^{(\ell-1)})} \prod_{\ell=1}^m\lemh_\ell(W^{(\ell)}),
\end{split}
\end{equation}
where the sum is over all possible subsets  $W^{(\ell)}\subset \lemS_\ell$ for all $\ell=1,\cdots,m$.

\subsection{The determinant of $M$ on the right hand side of~\eqref{eq:key_identity}}
\label{sec:first_determinant_RHS_key_identity}

Since the matrix $M$ is same as the matrix $T$ except that the sums are over $\lemSR_\ell$ instead of $\lemS_\ell$, 
we obtain a similar formula as \eqref{eq:lhs_identity1}. 
This time, since $|\lemSR_\ell|=N$, there is only one subset $W_\ell$ of cardinality $N$. 
Hence, we find that $W_\ell= \lemSR_\ell$ and we find that $\det[M]$ is equal to \eqref{eq:deofmfo}. 

\subsection{The Fredholm determinant on the right hand side of~\eqref{eq:key_identity}} \label{sec:seconddetRHS}

We expand $\det(I-\lemK_1\lemK_2)$ using the series definition of Fredholm determinants. 
Recall that the matrices $\lemK_1$ and $\lemK_2$ have special block structures. 
In this case, the series has the following formula. 

\begin{lm}[Lemma 4.8 of \cite{Baik-Liu17}]
	For matrices $\lemK_1$ and $\lemK_2$ with the block structures~\eqref{eq:def_lemK_1} or~\eqref{eq:def_lemK_2}, we have 
	\begin{equation*}
	\begin{split}
	\det(I-\lemK_1\lemK_2)=&\sum_{\bn}\frac{(-1)^{|\bn|}}{(\bn!)^2}\sum_{\substack{W\in\lemSL_1^{n_1}\times\lemSR_2^{n_2}\times\cdots\\ W'\in\lemSR_1^{n_1}\times\lemSL_2^{n_2}\times\cdots}}
	\det\left[\lemK_1(w_i,w'_j)\right]_{i,j=1}^{|\bn|}\det\left[\lemK_2(w'_i,w_j)\right]_{i,j=1}^{|\bn|},
	\end{split}
	\end{equation*}
	where the first sum is over all possible $\bn=(n_1,n_2,\cdots,n_m)\in (\intZ_{\ge0})^m$. 
	Here, $|\bn|:=n_1+\cdots+n_m$ and $\bn!:=n_1!\cdots n_m!$.
\end{lm}

We now express the determinants in the sum explicitly. 
Fix $\bn=(n_1,n_2,\cdots,n_m)$. 
Since the matrices $\lemK_1$ and $\lemK_2$ have block structures, the sub-matrices $\left[\lemK_1(w_i,w'_j)\right]_{i,j=1}^{|\bn|}$ and $\left[\lemK_2(w'_i,w_j)\right]_{i,j=1}^{|\bn|}$ also have block structures. 
We express the components of $W$ and $W'$ in the set $\lemSL_j^{n_j}$  by $U^{(j)}=(u_1^{(j)},\cdots,u_{n_j}^{(j)})$ and the components in the set  $\lemSR_j^{n_j}$ by $V^{(j)}=(v_1^{(j)},\cdots,v_{n_j}^{(j)})$. 
We set $n_0=n_{m+1}=0$. 
Using these notations, for odd $m$, 
\begin{equation*}
\begin{split}
\det\left[\lemK_1(w_i,w'_j)\right]_{i,j=1}^{|\bn|}
&=\prod_{\ell=1}^{(m+ 1)/2}\det\left[\mathbf{B}_{2\ell-1} (U^{(2\ell-1)}, V^{(2\ell-1)}, U^{(2\ell)}, V^{(2\ell)}) \right] 
\end{split}
\end{equation*}
and
\begin{equation*}
\begin{split}	
\det\left[\lemK_2(w'_i,w_j)\right]_{i,j=1}^{|\bn|}
&=\prod_{\ell=0}^{(m-1)/2}
\det\left[\mathbf{B}_{2\ell} (U^{(2\ell)}, V^{(2\ell)}, U^{(2\ell+1)}, V^{(2\ell+1)}) \right],
\end{split}
\end{equation*}
where
\beqq
\mathbf{B}_k (U^{(k)}, V^{(k)}, U^{(k+1)}, V^{(k+1)}) = 
\left[
\begin{array}{c|c}
	\lemBlk_{k}(u_{i}^{(k)},v_{j}^{(k)})&\lemBlk_{k}(u_{i}^{(k)},u_{j'}^{(k+1)})\\[2pt]
	\hline\\[-9pt]
	\lemBlk_{k}(v_{i'}^{(k+1)},v_{j}^{(k)})&\lemBlk_{k}(v_{i'}^{(k+1)},u_{j'}^{(k+1)})
\end{array}
\right]_{\substack{1\le i,j\le n_{k}\\ 1\le i',j'\le n_{k+1}}}. 
\eeqq
The case of even $m$ is similar. In both cases, 
\begin{equation}
\label{eq:lem_aux}
\begin{split}
&\det\left[\lemK_1(w_i,w'_j)\right]_{i,j=1}^{|\bn|}\det\left[\lemK_2(w'_i,w_j)\right]_{i,j=1}^{|\bn|}
=\prod_{\ell=0}^{m}
\det\left[\mathbf{B}_{\ell} (U^{(\ell)}, V^{(\ell)}, U^{(\ell+1)}, V^{(\ell+1)}) \right]  . 
\end{split}
\end{equation}
From the formula~\eqref{eq:def_lemBlk2} of the entries of $\lemBlk_\ell$, each factor determinant in the product in~\eqref{eq:lem_aux} can be evaluated using the Cauchy determinant formula: for $1\le \ell\le m-1$, 
\beqq
\begin{split}
	&\det\left[\mathbf{B}_{\ell} (U^{(\ell)}, V^{(\ell)}, U^{(\ell+1)}, V^{(\ell+1)}) \right]  \\
	&=(-1)^{{n_\ell+n_{\ell+1}\choose 2}+n_{\ell+1}+n_\ell n_{\ell+1}} \frac{\lemh_\ell(U^{(\ell)})}{\lemh_{\ell+1}(V^{(\ell+1)})}
	\cdot \frac{\lemq_\ell(U^{(\ell)}) \lemq_{\ell+1}(U^{(\ell+1)})}
	{\lemq'_\ell(V^{(\ell)}) \lemq'_{\ell+1}(V^{(\ell+1)})}
	\cdot \frac{\lemq_{\ell+1}(V^{(\ell)}) \lemq_\ell(V^{(\ell+1)})}
	{\lemq_{\ell+1}(U^{(\ell)}) \lemq_{\ell}(U^{(\ell+1)})}
	\\
	&\quad\cdot \frac{\Delta(U^{\ell})\Delta(V^{(\ell)}) \Delta(U^{(\ell+1)}) \Delta(V^{(\ell+1)}) \Delta(U^{(\ell)};V^{(\ell+1)}) \Delta(V^{(\ell)};U^{(\ell+1)})}
	{\Delta(U^{(\ell)};V^{(\ell)}) \Delta(U^{(\ell+1)};V^{(\ell+1)}) \Delta(U^{(\ell)};U^{(\ell+1)}) \Delta(V^{(\ell)};V^{(\ell+1)})}.
\end{split}
\eeqq
This computation was also done in Lemma 4.9 of \cite{Baik-Liu17}. 
On the other hand, from the definitions~\eqref{eq:def_lemBlk_0} and~\eqref{eq:def_lemBlk_m} of $\lemBlk_0$ and $\lemBlk_m$, 
\beqq
\begin{split}
	&\det\left[\lemBlk_{0}(v_i^{(1)},u_j^{(1)})\right]_{i,j=1}^{n_1}
	=\det\left[\frac{\lemE(\lemSR_1\cup \{u_j^{(1)}\}\setminus \{v_i^{(1)}\};\lemF)}{\lemE(\lemSR_1;\lemF) (v_i^{(1)}-u_j^{(1)})} \right]_{i,j=1}^{n_1} 
	\frac{\lemq_1(U^{(1)})}{\lemq'_1(V^{(1)})}  \frac{1}{\lemh_1(V^{(1)})}
\end{split}
\eeqq
and
\beqq
\begin{split}
	&\det\left[\lemBlk_{m}(u_i^{(m)},v_j^{(m)})\right]_{i,j=1}^{n_m}
	=\det\left[\frac{\lemE(\lemSR_m\cup \{u_i^{(m)}\}\setminus \{v_j^{(m)}\};\lemG)}{\lemE(\lemSR_m;\lemG) (u_i^{(m)}-v_j^{(m)})} \right]_{i,j=1}^{n_m} 
	\frac{\lemq_m(U^{(m)})}{\lemq'_m(V^{(m)})}  \lemh_m(U^{(m)}).
\end{split}
\eeqq

Combining together, we find that 
\begin{equation*}
\begin{split}
&\det(I-\lemK_1\lemK_2)\\
&=\sum_{\bn}\frac{(-1)^{{n_1\choose 2}+{n_m\choose 2}}}{(\bn!)^2}\sum_{\substack{U^{(\ell)}\in\lemSL_\ell^{n_\ell}\\ V^{(\ell)}\in\lemSR_\ell^{n_\ell}\\ \ell=1,\cdots,m}}
\det\left[\frac{\lemE(\lemSR_1\cup \{u_j^{(1)}\}\setminus \{v_i^{(1)}\};\lemF)}{\lemE(\lemSR_1;\lemF) (u_j^{(1)}-v_i^{(1)})} \right]_{i,j=1}^{n_1} \\
&\quad\cdot
\frac{\Delta(U^{(1)};V^{(1)})}{\Delta(U^{(1)})\Delta(V^{(1)})}
\det\left[\frac{\lemE(\lemSR_m\cup \{u_i^{(m)}\}\setminus \{v_j^{(m)}\};\lemG)}{\lemE(\lemSR_m;\lemG) (u_i^{(m)}-v_j^{(m)})} \right]_{i,j=1}^{n_m}
\frac{\Delta(U^{(m)};V^{(m)})}{\Delta(U^{(m)})\Delta(V^{(m)})}\\
&\quad \cdot \prod_{\ell=1}^m \frac{\lemh_\ell(U^{(\ell)}) \left(\lemq_\ell(U^{(\ell)})^2\right)}
{\lemh_\ell(V^{(\ell)}) \left( \lemq'_\ell (V^{(\ell)})\right)^2}\frac{\Delta(U^{(\ell)})^2 \Delta(V^{(\ell)})^2}{\Delta(U^{(\ell)}; V^{(\ell)})^2}
\\
&\quad \cdot \prod_{\ell=1}^{m-1}\frac{\lemq_\ell(V^{(\ell+1)})\lemq_{\ell+1}(V^{(\ell)})}{\lemq_\ell(U^{(\ell+1)})\lemq_{\ell+1}(U^{(\ell)})}
\frac{\Delta(U^{(\ell)};V^{(\ell+1)}) \Delta(V^{(\ell)}; U^{(\ell+1)})}{ \Delta(U^{(\ell)}; U^{(\ell+1)}) \Delta(V^{(\ell)}; V^{(\ell+1)})}.
\end{split}
\end{equation*}
In the above equation, we changed the sign of the entries in the first determinant and this change contributed a $(-1)^{n_1}$ factor.
Here, the outside sum is taken over all $\bn=(n_1,n_2,\cdots,n_m)\in (\intZ_{\ge0})^m$. Since $|\lemSR_\ell|=N$ for all $\ell$, the sum is zero unless $n_\ell\le N$ for all $\ell$. 
We note that for given $\bn$, the summand is invariant under any permutation of the coordinates of $U^{(\ell)}$ or $V^{(\ell)}$ for each $\ell$ and it is zero if any of the components are equal. 
Hence, we may take $U^{(\ell)}$ as subsets of $\lemSL_\ell$ of cardinality $n_\ell$, take $V^{(\ell)}$ as subsets of $\lemSR_\ell$ of cardinality $n_\ell$, and multiply the summand by $(\bn!)^2$. 
We can also combine the sum over $\bn$ and the inside sum so that the sum is over all appropriate $U^{(\ell)}$ and $V^{(\ell)}$ of same cardinality. 

\subsection{Completion of the proof of Proposition \ref{lm:key_lm}}

By Subsection \ref{sec:detLHS}, the left-hand side of \eqref{eq:key_identity} is a sum over all subsets $W^{(\ell)}$ of $\lemS_\ell$ satisfying $|W^{(\ell)}|=N$ for $1\le \ell\le m$. 
By Subsection \ref{sec:seconddetRHS}, the right-hand side of \eqref{eq:key_identity} is a sum over all subsets $U^{(\ell)}$ of $\lemSL_\ell$ and subsets $V^{(\ell)}$ of $\lemSR_{\ell}$ satisfying $|U^{(\ell)}|= |V^{(\ell)}|$ for $1\le \ell\le m$.  
The identity~\eqref{eq:key_identity} is proved if we show the following lemma.

\begin{lm} \label{lem:twosicme}
	For each $1\le \ell\le m$, let $U^{(\ell)}\subset \lemSL_\ell$ and $V^{(\ell)}\subset \lemSR_\ell$ be subsets of equal cardinality satisfying $|U^{(\ell)}| = |V^{(\ell)}| \le N$. 
	Define the set $W^{(\ell)} = (\lemSR_\ell\setminus V^{(\ell)}) \cup U^{(\ell)}$. 
	Then, 
	\begin{equation}
	\label{eq:reduced_identity}
	\begin{split}
	&\lemE(W^{(1)};\lemF)\lemE(W^{(m)};\lemG) \frac{\prod_{\ell=1}^m \Delta(W^{\ell})^2}{\prod_{\ell=2}^{m}\Delta(W^{(\ell)};W^{(\ell-1)})} \prod_{\ell=1}^m\lemh_\ell(W^{(\ell)})\\
	&\qquad =\lemE(\lemSR_1;\lemF)\lemE(\lemSR_m;\lemG) \frac{\prod_{\ell=1}^m \Delta(\lemSR_\ell)^2}{\prod_{\ell=2}^{m}\Delta(\lemSR_\ell;\lemSR_{\ell-1})}	\prod_{\ell=1}^m\lemh_\ell(\lemSR_\ell)\\
	&\qquad\qquad \cdot (-1)^{{n_1\choose 2}+{n_m\choose 2}}\det\left[\frac{\lemE(\lemSR_1\cup \{u_j^{(1)}\}\setminus \{v_i^{(1)}\};\lemF)}{\lemE(\lemSR_1;\lemF) (u_j^{(1)}-v_i^{(1)}) } \right]_{i,j=1}^{n_1} 
	\frac{\Delta(U^{(1)};V^{(1)})}{\Delta(U^{(1)})\Delta(V^{(1)})}
	\\
	&\qquad\qquad \cdot  \det\left[\frac{\lemE(\lemSR_m\cup \{u_i^{(m)}\}\setminus \{v_j^{(m)}\};\lemG)}{\lemE(\lemSR_m;\lemG) (u_i^{(m)}-v_j^{(m)})} \right]_{i,j=1}^{n_m}
	\frac{\Delta(U^{(m)};V^{(m)})}{\Delta(U^{(m)})\Delta(V^{(m)})}\\
	&\qquad\qquad \cdot \prod_{\ell=1}^m \frac{\lemh_\ell(U^{(\ell)}) \left(\lemq_\ell(U^{(\ell)})^2\right)}
	{\lemh_\ell(V^{(\ell)}) \left( \lemq'_\ell (V^{(\ell)})\right)^2}\frac{\Delta(U^{(\ell)})^2 \Delta(V^{(\ell)})^2}{\Delta(U^{(\ell)}; V^{(\ell)})^2}
	\\
	&\qquad\qquad \cdot \prod_{\ell=1}^{m-1}\frac{\lemq_\ell(V^{(\ell+1)})\lemq_{\ell+1}(V^{(\ell)})}{\lemq_\ell(U^{(\ell+1)})\lemq_{\ell+1}(U^{(\ell)})}
	\frac{\Delta(U^{(\ell)};V^{(\ell+1)}) \Delta(V^{(\ell)}; U^{(\ell+1)})}{ \Delta(U^{(\ell)}; U^{(\ell+1)}) \Delta(V^{(\ell)}; V^{(\ell+1)})}.
	\end{split}
	\end{equation}
\end{lm}

We prove the identity by expressing each term on the the left-hand side (which is given in terms of the set $W^{(\ell)}$) using the sets $U^{(\ell)}$ and $V^{(\ell)}$. 
The factors involving $\lemh_\ell$ agree since $\lemh_\ell(W^{(\ell)}) \lemh_\ell(V^{(\ell)}) = \lemh_\ell(\lemSR_{\ell}) \lemh_\ell(U^{(\ell)})$, which follows from the facts that $W^{(\ell)} = (\lemSR_\ell\setminus V^{(\ell)}) \cup U^{(\ell)}$ and hence $W^{(\ell)} \cup V^{(\ell)} = \lemSR_{\ell} \cup U^{(\ell)}$. 
The factors $\Delta(W^{(\ell)})$ and $\Delta(W^{(\ell)}; W^{(\ell-1)})$ are computed in Lemma \ref{lem:firstV} below. 
We then compute the factors $\lemE(W^{(1)};\lemF)$ and $\lemE(W^{(m)};\lemG)$ in Lemma \ref{lem:secondG}.
These two lemmas prove Lemma \ref{lem:twosicme} above.  

Both sides of the identity \eqref{eq:reduced_identity} are symmetric functions of the variables.  
For the next lemma, we order the elements of $W^{(\ell)}$, $U^{(\ell)}$, and $V^{(\ell)}$  arbitrarily and regard the sets as vectors. 
The important aspect of the formula \eqref{eq:aux_0066} below is that the right-hand side depends on $A$ only in terms of $r(U)$ and $r'(V)$.

\begin{lm}\label{lem:firstV}
	\begin{enumerate}[(a)]
		\item 
		Let $U=(u_1, \cdots, u_n)$ and $V=(v_1, \cdots, v_n)$ be vectors of same dimension $n$. Let $A=(a_1, \cdots, a_{n'})$ be a vector of dimension $n'$ which is allowed to be different from $n$.  
		All components are complex numbers. 
		Define new vectors $R= (V, A)$ and $W=(U, A)$. 
		Then, 
		\begin{equation}\label{eq:aux_0066}
		\frac{\Delta(W)}{\Delta(R)} = (-1)^{\binom{n}{2}} \frac{\Delta(U) \Delta(V) r(U)}{\Delta(U;V) r'(V)},
		\end{equation}
		where $r(w)= \prod_{i=1}^N (w-r_i)$ using the notation $R=(r_1, \cdots, r_N)$ with $N=n+n'$. 
		\item 
		In addition, let $\widetilde U$ and $\widetilde V$ be vectors of same dimension $\tilde n$, and let $\widetilde A$ be a vector of possibly different dimension $\tilde n'$. 
		Define vectors $\widetilde R=( \widetilde V, \widetilde A)$ and $\widetilde W= (\widetilde U, \widetilde A)$. Then, 
		\begin{equation*} 
		\frac{\Delta(W; \widetilde W)}{\Delta(R; \widetilde R)} 
		= \frac{\Delta(U; \widetilde U) \Delta(V; \widetilde V) r(\widetilde U) \widetilde r(U)}{\Delta(U; \widetilde V) \Delta( V; \widetilde U) r(\widetilde V) \widetilde r(V)},
		\end{equation*}
		where $\widetilde r(w)= \prod_{i=1}^{\tilde N} (w- \tilde r_i)$ with $\widetilde R=(\tilde r_1, \cdots, \tilde r_{\tilde N})$ and $\tilde N = \tilde n + \tilde n'$.
	\end{enumerate}
\end{lm}

\begin{proof}
	These identities were stated in (7.48) and (7.50) of \cite{Baik-Liu16}, and also in (4.43) and (4.44) of \cite{Baik-Liu17} for the case that the elements of the vectors are solutions of an algebraic equation. However, the same proof goes through without assuming it. 
	We include the proof for the convenience of the reader.

	(a) It is a easy to see that 
	\begin{equation*}
	\frac{\Delta(W)}{\Delta(R)} =  \frac{\Delta(U) \Delta(U; A)}{\Delta(V) \Delta(V;A)} .
	\end{equation*}
	Since $r(w)=\prod_{j=1}^n (w-v_j) \prod_{k=1}^{n'} (w-a_k)$, we find after substituting $w=u_i$ and taking the product over $i$ that 
	\begin{equation*}
	\lemq(U)
	= \prod_{i=1}^n \lemq(u_i)
	=\Delta(U; V) \bigg[ \prod_{i,k}(u_i-a_k) \bigg]=  \Delta(U;V) \Delta(U;A) .
	\end{equation*}
	Similarly, substituting $w=v_i$ in the derivative $\lemq'(w)$ and taking the product over $i$, 
	\begin{equation*}
	\lemq'(V) =   \prod_{\substack{i\neq j\\ 1\le i,j\le n}} (v_i-v_j) \bigg[ \prod_{i=1}^n\prod_{k=1}^{n'} (v_i-a_k)\bigg] = (-1)^{n(n-1)/2} \Delta(V)^2 \Delta(V;A).
	\end{equation*}
	Taking the ratio of the above two formulas, we obtain an expression for the ratio of $\Delta(U;A)$ and $\Delta(V;A)$.
	Inserting it to the first equation, we obtain the result. 
	
	(b) The proof is similar to the case (a). As in (a), we have
	\beqq
	\frac{\Delta(W; \widetilde W)}{\Delta(\lemSR; \widetilde \lemSR)}=  \frac{\Delta(U; \widetilde U) \Delta(U; \widetilde A) \Delta(\widetilde U; A)}{\Delta(V; \widetilde V) \Delta(V; \widetilde A) \Delta(\widetilde V; A)} .
	\eeqq
	We then use the identities 
	\beqq
	\widetilde \lemq(U)=  \Delta(U; \widetilde A) \Delta(U; \widetilde V), 
	\qquad
	\widetilde \lemq(V)=  \Delta(V; \widetilde A) \Delta(V; \widetilde V)
	\eeqq
	and similar formulas for $\lemq(\widetilde U)$ and $\lemq(\widetilde V)$, which are proved by the same way as in the case (a). We obtain the result immediately.
\end{proof}

Next lemma finds an expression for $\lemE(W^{(1)};\lemF)$ and $\lemE(W^{(m)};\lemG)$. 

\begin{lm} \label{lem:secondG}
	Let $U, V, R, W$ be same as in Lemma \ref{lem:firstV} (a). 
	Note that $U$ and $V$ are of same dimension $n$, and $R$ and $W$ are vectors of same dimension $N$.
	Let $p_i: \C\to \C$ be a function for each $1\le i\le N$ and set 
	\beqq
	\lemE(\alpha_1, \cdots, \alpha_N)= \frac{\det[ p_i(\alpha_j)]_{i,j=1}^N}{\det [\alpha_j^{i-1} ]_{i,j=1}^N} .
	\eeqq
	Then,
	\begin{equation}
	\label{eq:reduced_identity1}
	\frac{\lemE (W)}{\lemE (R)}=(-1)^{n\choose 2}\det\left[\frac{\lemE (\lemSR^{i,j}) }{\lemE (\lemSR) (u_j -v_i)} \right]_{i,j=1}^{n} 
	\frac{\Delta(U; V)}{\Delta(U)\Delta(V)},
	\end{equation}
	where
	\beq \label{eq:Rijde}
	R^{i,j} = R \big|_{v_i\mapsto u_j} \quad \text{for $1\le i,j\le n$.} 
	\eeq
\end{lm}

\begin{proof} 
	From the definition and part (a) of the previous lemma, 
	\beq \label{eq:aux_005}
	\frac{\lemE(W)}{\lemE(\lemSR)} = \frac{\det[\lemf_i(w_j)]_{i,j=1}^N \Delta(R)}{\det[\lemf_i(r_j)]_{i,j=1}^N \Delta(W)} 
	=  (-1)^{\binom{n}{2}} \frac{\det[P] \Delta(U;V) \lemq'(V)}{\det[Q] \Delta(U)\Delta(V) \lemq(U)},
	\eeq
	where $P$ is the $N\times N$ matrix with entries $p_i(w_j)$ and $Q$ is the $N\times N$ matrix with entries $p_i(r_j)$. 
	Since the last $N-n$ columns of $P$ and $Q$ are same, the matrix $Q^{-1}P$ has the block structure
	\beqq
	Q^{-1}P= \left[ \begin{array}{c|c}
		* &0\\[2pt]
		\hline\\[-9pt]
		*&I 
	\end{array} \right],
	\eeqq 
	where the bottom-right block is the $(N-n)\times (N-n)$ identity matrix. 
	Thus, the determinant of $Q^{-1}P$ is equal to the determinant of the $n\times n$ top-left block: 
	\beqq
	\frac{\det[P]}{\det[Q]} = \det \left[ (Q^{-1}P)_{ij} \right]_{i,j=1}^n .
	\eeqq
	We now express the entries of $Q^{-1}P$ using the Cramer's rule; $(Q^{-1}P)_{ij}= \frac{\det [ Q^{(i,j)}] }{\det [ Q]}$ 
	where $Q^{(i,j)}$ which is same as $Q$ except that the $i$th column is replaced by the $j$th column of $P$. 
	From the definition of $\lemE$ and $Q$, 
	\beqq
	\det[Q]= \lemE(\lemSR) \Delta(R). 
	\eeqq
	On the other hand,  the matrix $Q^{(i,j)}= [p_s(\hat{r}_t)]_{s,t=1}^N$ where $\hat{r}_t=r_t$ for $t\neq i$ and $\hat{r}_i=w_j$. 
	Note that for $1\le j \le n$, $r_i=v_i$ and $w_j=u_j$. 
	Hence, by the definition of $\lemE$, 
	\beqq
	\det[Q^{(i,j)}]= \lemE( R^{i,j}) \Delta( R^{i,j}) \qquad \text{for $1\le i,j\le n$,} 
	\eeqq
	where  $R^{i,j}$ is defined in \eqref{eq:Rijde}.
	It is direct to check that 
	\beqq
	\frac{\Delta(R^{i,j})}{\Delta(R)} = \frac{r(u_j)}{(u_j-v_i) r'(v_i)}, 
	\eeqq
	and hence, 
	\beqq
	\frac{\det[P]}{\det[Q]} 
	= \det\left[ \frac{\lemE( R^{i,j}) \Delta( R^{i,j}) }{\lemE(\lemSR) \Delta(R)}\right]_{i,j=1}^n
	= \det\left[ \frac{\lemE( R^{i,j}) }{\lemE(\lemSR) (u_j-v_i)} \right]_{i,j=1}^n \frac{\lemq(U)}{\lemq'(V)}.
	\eeqq
	Together with \eqref{eq:aux_005}, this implies \eqref{eq:reduced_identity1}.
\end{proof}

The last two lemmas imply Lemma~\ref{lem:twosicme}, and we complete the proof of Proposition~\ref{lm:key_lm}.

\section{Proof of Theorem~\ref{thm:Fredholm}} \label{sec:proofofal}

We prove Theorem~\ref{thm:Fredholm} in this section. 

\subsection{Toeplitz-like formula}

The following Toeplitz-like formula for the multi-point distribution for general initial condition was obtained in \cite{Baik-Liu17}. 
As discussed at the beginning of the previous section, we prove Theorem~\ref{thm:Fredholm} by converting the  Toeplitz-like determinant to a Fredholm determinant formula using Proposition \ref{lm:key_lm}.

\begin{thm}[Toeplitz-like formula; Theorem 3.1 of \cite{Baik-Liu17}] \label{thm:Toeplitzform}
	Consider the process $\PTASEP(L,N, Y)$.  
	Fix a positive integer $m$, and let $(k_i,t_i)$, $1\le i\le m$, be $m$ distinct points in $\intZ\times \R_+$ 
	satisfying $0< t_1\le \cdots\le t_m$. 
	Then, for arbitrary integers $a_1,\cdots,a_m$, 
	\begin{equation*}
	\prob_Y \left( \bigcap_{\ell=1}^m \{ x_{k_\ell}(t_\ell) \ge a_\ell \} \right) 
	= \oint \cdots \oint 
	\caC(\bz) \caD_Y (\bz)    
	\ddbar{z_m}{z_m} \cdots \ddbar{z_1}{z_1},
	\end{equation*}
	where  the contours are nested circles satisfy $0<|z_m|<\cdots<|z_1|$. The functions in the integrand are 
	\begin{equation}
	\label{eq:def_caC}
	\caC(\bz) = (-1)^{(k_m-1)(N+1)} z_1^{(k_1-1)L} 
	\prod_{\ell=2}^{m} \left[ z_\ell^{(k_\ell - k_{\ell-1})L }  
	\left( 	\left( \frac{z_\ell}{z_{\ell-1}} \right)^L - 1 	\right)^{N-1} \right]
	\end{equation}
	and
	\begin{equation}
	\label{eq:aux_2018_04_06_01}
	\caD_Y(\bz)
	= \det \left[  \sum_{\substack{w_1\in\roots_{z_1}\\ \cdots \\ w_m\in\roots_{z_m}}}
	\frac{w_1^i(w_1+1)^{y_i-i}w_m^{-j}}
	{\prod_{\ell=2}^{m} (w_\ell -w_{\ell-1})}
	\prod_{\ell=1}^{m} \ggftn_\ell(w_\ell)      \right]_{i,j=1}^N,
	\end{equation}
	where for $1\le \ell \le m$,
	\begin{equation} \label{eq:aux_2018_04_08_04}
	\ggftn_\ell(w) := \frac{w(w+1)}{L(w+\rho)} \frac{w^{-k_\ell}(w+1)^{-a_\ell+k_\ell} e^{t_\ell w}}{ w^{-k_{\ell-1}} (w+1)^{a_{\ell-1}+k_{\ell-1}} e^{t_{\ell-1}w}} .
	\end{equation}
	Here, we set $t_0=k_0=a_0=0$. 
	Recall from \eqref{eq:Betherootsset} that $\roots_z= \{ w\in \C : q_z(w)=0\}$ is the set of Bethe roots corresponding to $z$.
\end{thm}

Unlike Theorem \ref{thm:Fredholm}, the complex numbers $z_i$ do not need to satisfy $|z_i|<\rr$ in the above formula. 
\bigskip

Theorem~\ref{thm:Fredholm} is proved if we show the following identity. 

\begin{prop} \label{prop:CDlocz}
	Let $\caC(\bz)$ and $\caD_Y(\bz)$ be the functions defined in \eqref{eq:def_caC} and \eqref{eq:aux_2018_04_06_01}.
	Let $\scrC_Y(\bz)$ and $\scrD_Y(\bz)$ be the functions defined in \eqref{eq:def_C0} and \eqref{eq:def_D0}. 
	Then, we have 
	\begin{equation} 
	\label{eq:identity_CD}
	\caC(\bz)\caD_Y(\bz) = \scrC_Y(\bz) \scrD_Y(\bz)
	\end{equation}
	for all $\bz=(z_1, \cdots, z_m)$ satisfying $0<|z_m|<\cdots<|z_1|<\rr$. 
\end{prop}

\begin{rmk}[Proof of Lemma \ref{lem:analyticityofCD}] \label{rmk:removable_poles}
	The Bethe roots are analytic functions of $z\neq 0$. 
	From the definitions, $\scrC_Y(\bz)$ is analytic in $0<|z_m|<\cdots<|z_1|<\rr$ and $\scrD_Y(\bz)$ is meromorphic in the same domain with possible poles when $\energy_Y(z_1)= 0$.
	On the other hand, from the definitions, the functions $\caC(\bz)$ and $\caD_Y(\bz)$ are analytic in the same domain. 
	(They are actually analytic in the larger domain $0<|z_m|<\cdots < |z_1|$; see the paragraph after the proof of Corollary 3.3 of \cite{Baik-Liu17}.) 
	The identity \eqref{eq:identity_CD} implies that the product $\scrC_Y(\bz) \scrD_Y(\bz)$ is analytic for $0<|z_m|<\cdots<|z_1|<\rr$.
\end{rmk}

\subsection{Proof of Proposition \ref{prop:CDlocz}}

It is enough to prove the result for $z_1$ satisfying $\energy_Y(z_1)\neq 0$ since both sides of the identity \eqref{eq:identity_CD} are meromorphic functions and the left-hand side is analytic. 

We convert $\caD_Y(\bz)$ to a Fredholm determinant using Proposition~\ref{lm:key_lm}. 
We set $\lemS_i=\roots_{z_i}$, $\lemSR_i=\rootsR_{z_i}$, and $\lemSL_i=\rootsL_{z_i}$ in the proposition. 
We also set 
\begin{equation}
\label{eq:aux_020}
\lemf_i(w)=w^{N-i}(w+1)^{y_{N+1-i}+i-1} \quad \text{and} \quad \lemg_i(w)=w^{i-1}
\end{equation} 
for $1\le i\le N$, and
\begin{equation}
\label{eq:def_lemh}
\lemh_i(w) = \begin{dcases}
\ggftn_i(w)w(w+1)^{-N}, \quad & i=1,\\
\ggftn_i(w),& 2\le i\le m-1,\\
\ggftn_i(w) w^{-N}, & i=m,
\end{dcases}
\end{equation}
where $\ggftn_i(w)$ is defined in \eqref{eq:aux_2018_04_08_04}. 
Proposition~\ref{lm:key_lm} is applicable if 
\beqq
\det\left[ \lemf_i(v_{j}^{(1)})\right]_{i,j=1}^N \det\left[ \lemg_i(v_{j}^{(m)})\right]_{i,j=1}^N\neq 0, 
\eeqq
where $\lemSR_1=\{v_{1}^{(1)}, \cdots, v_{N}^{(1)}\}$ and $\lemSR_m=\{v_{1}^{(m)}, \cdots, v_{N}^{(m)}\}$.
For our choice of $\lemf_i$ and $\lemg_i$, using the notation \eqref{eq:def_energy}, 
\beq \label{eq:fgdr}
\begin{split}
	\det\left[ \lemf_i(v_{j}^{(1)})\right]_{i,j=1}^N &=  (-1)^{N(N-1)/2}\energy_Y(z_1) \Delta(v_1^{(1)},\cdots,v_N^{(1)}),\\
	\det\left[ \lemg_i(v_{j}^{(m)})\right]_{i,j=1}^N&=  
	\Delta(v_1^{(m)},\cdots,v_N^{(m)}) .
\end{split}
\eeq
The Vandermonde terms are non-zero since all points in $\rootsR_{z_i}$ are distinct. 
Since we assumed that $\energy_Y(z_1)\neq 0$, Proposition~\ref{lm:key_lm} is applicable and we obtain
\begin{equation} \label{eq:aux_03_29_03}
\begin{split}
\caD_Y(\bz) &= 
\det\left[\sum_{\substack{w_1\in \roots_{z_1}\\ \cdots\\ w_m\in\roots_{z_m}}}\frac{\lemf_i(w_1)\lemg_j(w_m)}{\prod_{\ell=2}^{m} (w_\ell-w_{\ell-1})} \prod_{\ell=1}^{m} \lemh_\ell(w_\ell)\right]_{1\le i,j\le N}  \\
&=\mathcal{B}_Y(\bz)
\det\left(I -\lemK_1\lemK_2\right),
\end{split}
\end{equation}
where $\lemK_1$ and $\lemK_2$ are kernels defined by~\eqref{eq:def_lemK_1} and~\eqref{eq:def_lemK_2} with the special choices of $\lemf_i,\lemg_i,\lemh_i$ and $\lemS_i,\lemSR_i,\lemSL_i$ described above, and 
\beqq \label{eq:calBf}
\mathcal{B}_Y(\bz)= (-1)^{mN(N-1)/2} \energy_Y(z_1) 
\frac{\prod_{\ell=1}^m \Delta(\lemSR_\ell)^2}{\prod_{\ell=2}^{m}\Delta(\lemSR_\ell;\lemSR_{\ell-1})}\prod_{\ell=1}^m\lemh_\ell(\lemSR_\ell)
\eeqq
from \eqref{eq:deofmfo} using \eqref{eq:fgdr}.

It is thus sufficient to show that
\begin{equation} \label{eq:identity_1}
\scrC_Y(\bz) = \caC(\bz)  \mathcal{B}_Y(\bz)
\end{equation}
and
\begin{equation} \label{eq:identity_2}
\scrD_Y(\bz) = \det(I- \lemK_1\lemK_2).
\end{equation}

\begin{rmk}
	The equation \eqref{eq:aux_03_29_03} already gives us a formula for the multi-point distribution in terms of an integral involving a Fredholm determinant. 
	However, we make further changes using \eqref{eq:identity_1} and \eqref{eq:identity_2} in order to make the final formula suitable for asymptotic analysis. In the formula given in Theorem~\ref{thm:Fredholm} it is easy to take the large time limit in the relaxation time scale. 
	Most of the changes made in the identities \eqref{eq:identity_1} and \eqref{eq:identity_2} are cosmetic notational changes, but a few changes use the algebraic equation for the Bethe roots and hence utilize the fact that we consider a spatially periodic model. 
\end{rmk}

\subsubsection{Proof of~\eqref{eq:identity_1}}

Recall that $\energy_\step(z_1)=1$ for the step initial condition. 

\begin{lm}[\cite{Baik-Liu17}]
	The identity~\eqref{eq:identity_1} holds for the step initial condition $Y=(-N+1, \cdots, -1, 0)$. 
\end{lm}

\begin{proof}
	This lemma is same as the identity $C(\mathbf{z})= \mathcal{C} (\mathbf{z}, \mathbf{k})\mathcal{B}(\mathbf{z})$ proved in the paragraph below the equation (4.52) in \cite{Baik-Liu17}. 
	The proof uses the equation $w^N(w+1)^{L-N}= z^L$ for the Bethe roots which allows us to express certain expressions involving functions of $\rootsR_{z_i}$ in terms of functions of $\rootsL_{z_i}$: see (4.49)--(4.52) of \cite{Baik-Liu17}.
\end{proof}

Note that $\mathcal{B}_Y(\bz)$ depends on $Y$ only in the factor $\energy_Y(z_1)$. 
Thus we may write 
\beqq
\mathcal{B}_Y(\bz) = \energy_Y(z_1) \mathcal{B}_{\step}(\bz). 
\eeqq
Since $\scrC_Y(\bz)= \energy_Y(z_1) \scrCstep(\bz)$ by definition, the above lemma proves \eqref{eq:identity_1}.

\subsubsection{Proof of~\eqref{eq:identity_2}}

The equation~\eqref{eq:identity_2} is an identity of two Fredholm determinants. 
Translating the notations, we can show that $\scrKone=\lambda \lemK_1 \lambda'$ and $\scrKtwo= \mu' \lemK_2 \mu$ for multiplicative operators $\lambda, \mu, \lambda', \mu'$ which satisfy a certain property (see \eqref{eq:scalar_relation}.) 
These transformations of the operators are not the same as conjugations, but due to the block structure of the operators, the Fredholm determinant is still invariant under these transformation as we show in the next lemma.

\begin{lm} \label{lm:conjugation}
	Let $\Sigma_1,\cdots,\Sigma_m$ be disjoint sets in $\complexC$ and let $\mathcal{H}=L^2(\Sigma_1\cup\cdots\cup\Sigma_m,\mu)$ for a measure $\mu$. 
	Let $\Sigma'_1,\cdots,\Sigma'_m$ be another collection of disjoint sets in $\complexC$ and let $\mathcal{H'}=L^2(\Sigma'_1\cup\cdots\cup\Sigma'_m,\mu')$ for a measure $\mu'$. 
	Let $A:\mathcal{H}'\to \mathcal{H}$ and $B:\mathcal{H} \to\mathcal{H}'$ be operators defined by kernels.
	Assume the following block structures for the kernels:\begin{itemize}
		\item $A(w,w')=0$ unless there is an index $i$ such that $w\in\Sigma_{2i-1}\cup \Sigma_{2i}$ and $w'\in\Sigma'_{2i-1}\cup\Sigma'_{2i}$,
		\item $B(w',w)=0$ unless there is an index $i$ such that $w'\in\Sigma'_{2i}\cup \Sigma'_{2i+1}$ and $w\in\Sigma_{2i}\cup\Sigma_{2i+1}$.
	\end{itemize}
	Let $\lambda,\mu$ be two complex-valued functions on $\Sigma_1\cup\cdots \cup\Sigma_m$ and $\lambda',\mu'$ be two complex-valued functions on $\Sigma'_1\cup\cdots\cup \Sigma'_m$ satisfying 
	\begin{equation}
	\label{eq:scalar_relation}
	\lambda(w) \mu (w) \lambda'(w') \mu'(w')=1 \qquad \text{for every }(w,w')\in \Sigma_i\times\Sigma'_i \text{ with }1\le i\le m.
	\end{equation}
	Assume that the Fredholm determinant $\det(I-AB)$ is well-defined and is equal to the usual Fredholm determinant series expansion. 
	Then, $\det(I-(\lambda A \lambda') (\mu' B\mu))$ is also well defined by the usual Fredholm determinant series expansion and
	\begin{equation*}
	\det(I-AB) = \det(I -(\lambda A \lambda') (\mu' B\mu)).
	\end{equation*}
\end{lm}

\begin{proof}
	Under the given block structures, the series expansion of the Fredholm determinant becomes (see Lemma 4.8 of \cite{Baik-Liu17})
	\begin{equation}
	\label{eq:aux_011}
	\begin{split}
	\det(I-AB)&=\sum_{\bn\in(\intZ_{\ge 0})^m}
	\frac{(-1)^{|\bn|}}{(\bn!)^2} \int_{\Sigma_1^{n_1}\times\cdots\times\Sigma_m^{n_m}}\int_{(\Sigma'_1)^{n_1}\times\cdots\times(\Sigma'_m)^{n_m}}\\
	&\quad \det\left[A(w_i,w_j')\right]_{i,j=1}^{|\bn|} \det\left[B(w'_i,w_j)\right]_{i,j=1}^{|\bn|} \prod_{i=1}^{|\bn|}\dd\mu'(w_i')\prod_{i=1}^{|\bn|}\dd\mu(w_i).
	\end{split}
	\end{equation}
	Note that
	\beqq
	\left(\prod_{i=1}^{|\bn|}\lambda(w_i)\lambda'(w'_i)\right)\cdot \det\left[A(w_i,w_j')\right]_{i,j=1}^{|\bn|} =\det\left[(\lambda A \lambda')(w_i,w_j')\right]_{i,j=1}^{|\bn|},
	\eeqq
	and
	\beqq
	\left(\prod_{i=1}^{|\bn|}\mu'(w'_i)\mu(w_i)\right)\cdot \det\left[B(w'_i,w_j)\right]_{i,j=1}^{|\bn|} =\det\left[(\mu' B \mu)(w'_i,w_j)\right]_{i,j=1}^{|\bn|}.
	\eeqq
	If we multiply two terms, the product of $\lambda,\lambda',\mu,\mu'$ becomes $1$ due to  \eqref{eq:scalar_relation}. Hence, the right hand side of~\eqref{eq:aux_011} does not change if we replace $A,B$ by $\lambda A\lambda'$ and $\mu' B\mu$ respectively and this completes the proof.
\end{proof}

Now we prove~\eqref{eq:identity_2} by applying the above Lemma. 
Note that $\scrKone,\scrKtwo$ have the same block structure as $\lemK_1$ and $\lemK_2$, respectively. 
The kernels of $\lemK_1$ and $\lemK_2$ are given in terms of $\lemq_i(w)$ and $\lemh_i(w)$: see \eqref{eq:def_lemBlk2}--\eqref{eq:def_lemBlk_m}.
We re-write them in terms of $H_{z}(w)$, $f_i(w)$, $J(w)$, and $Q_{1,2}(j)$ in Section~\ref{sec:formulaofstepfinite} by which the kernels of $\scrKone$ and $\scrKtwo$ are expressed. 
We first express $\lemq_i(w)$ in terms of $H_{z}(w)$, $J(w)$, and $Q_{1}(j)$, $Q_2(j)$, and find a relationship between $\scrKone$, $\scrKtwo$ and $\lemK_1$, $\lemK_2$ while keeping the term $\lemh_i(w)$, which is trivially related to $J(w)$ and $f_i(w)$: see \eqref{eq:fJhrela} below.

We have
\begin{equation*}
\lemq_i(w)=q_{z_i,\RR}(w),\qquad \lemq_i'(w)=q'_{z_i,\RR}(w). 
\end{equation*}
Observe that if $w$ is a Bethe roots corresponding to $z_j$, then 
\beqq
q_{z_i, \RR}(w) = \frac{q_{z_i}(w)}{q_{z_i, \LL}(w)} = \frac{w^N(w+1)^{L-N}- z_i^L}{q_{z_i, \LL}(w)}
= \frac{z_j^L- z_i^L}{q_{z_i, \LL}(w)}. 
\eeqq
Using the above formula when $w\in \rootsR_{z_j}$, we find that (recall \eqref{eq:def_H}) 
\begin{equation*}
\lemq_i(w)= q_{z_i,\RR}(w)= \begin{dcases}
w^N H_{z_i}(w) \qquad & \text{for $w\in \rootsL_{z_j}$}, \\
\frac{w^N}{H_{z_i}(w)} \left( 1- \frac{z_i^L}{z_j^L}\right)  
\qquad & \text{for $w\in \rootsR_{z_j}$.} 
\end{dcases}
\end{equation*}
The term in the big parenthesis is $Q_1(j)$ if $i= j-(-1)^j$ and is $Q_2(j)$ if $i= j+(-1)^j$. 
Similarly, 
\begin{equation*}
\lemq_i'(w)= q_{z_i,\RR}'(w) =\frac{w^N}{J(w) H_{z_i}(w)} \qquad \text{for $w\in\rootsR_{z_i}$.}
\end{equation*}

Using the above identities, it is direct to check that 
\begin{equation*}
\scrKone=\lambda \lemK_1 \lambda',\qquad \scrKtwo= \mu' \lemK_2 \mu,
\end{equation*}
where $\lambda, \mu$ (and $\lambda',\mu'$ respectively) are multiplication operators on $\scrS_1$ (and $\scrS_2$ respectively). Their formulas are:
\beqq
\begin{split}
	&\lambda(w)= \begin{dcases}
		J(w)\frac{f_{2k+1}(w)}{\lemh_{2k+1}(w)}H_{z_{2k+1}}(w)  & \text{for $w\in\rootsL_{z_{2k+1}}$ if $1\le 2k+1\le m-1$,}\\
		J(w)\frac{f_m(w)}{w^{N}\lemh_{m}(w)} H_{z_{m}(w)} & \text{for $w\in\rootsL_{z_{m}}$ if  $m$ is odd,}\\
		f_{2k}(w)h_{2k}(w) H_{z_{2k}}(w) \frac{ z_{2k}^L}{z_{2k}^L-z_{2k-1}^L}  & \text{for $w\in\rootsR_{z_{2k}}$ if $2\le 2k\le m$,}
	\end{dcases}
\end{split}
\eeqq
\beqq
\lambda'(w)=
\begin{dcases}
	\frac{1}{J(w)H_{z_{2k+1}}(w)} \quad & \text{for $w\in\rootsR_{z_{2k+1}}$ if $1\le 2k+1\le m-1$,}\\
	\frac{w^N}{J(w)H_{z_m}(w)} &\text{for $w\in\rootsR_{z_m}$ if $m \text{ is odd}$,}\\
	\frac{1}{H_{z_{2k}}(w)} \frac{z_{2k}^L-z_{2k-1}^L}{z_{2k}^L} & \text{for $w\in\rootsL_{z_{2k}}$ if $2\le 2k\le m$,}
\end{dcases}
\eeqq
\beqq
\begin{split}
	&\mu'(w)= 
	\begin{dcases}
		f_{1}(w) h_1(w) w^N H_{z_1}(w) & \text{for $w\in\rootsR_{z_1}$,}\\
		f_{2k+1}(w) \lemh_{2k+1}(w) H_{z_{2k+1}}(w) \frac{z_{2k+1}^L}{z_{2k+1}^L-z_{2k}^L}  & \text{for $w\in\rootsR_{z_{2k+1}}$ if $2\le 2k\le m-1$,}\\
		J(w) \frac{f_{2k}(w)}{\lemh_{2k}(w)} H_{z_{2k}}(w) & \text{for $w\in\rootsL_{z_{2k}}$ if $2\le 2k\le m-1$,}\\
		J(w) \frac{f_m(w)}{w^N \lemh_m(w)} H_{z_m}(w) & \text{for $w\in\rootsL_{z_m}$ if $m$ is even,}
	\end{dcases}
\end{split}
\eeqq
and
\beqq
\mu(w)=
\begin{dcases}
	\frac{1}{w^N}\frac{1}{H_{z_1}(w)}& \text{for $w\in\rootsL_{z_1}$,}\\
	\frac{1}{H_{z_{2k+1}}(w)}\left(1-\frac{z_{2k}^L}{z_{2k+1}^L}\right) & \text{for $w\in\rootsL_{z_{2k+1}}$ if $2\le 2k\le m-1$,}\\
	\frac{1}{J(w)}\frac{1}{H_{z_{2k}}(w)} \quad & \text{for $w\in\rootsR_{z_{2k}}$ if $2\le 2k\le m-1$,}\\
	\frac{1}{J(w)}\frac{w^N}{H_{z_{m}}(w)} & \text{for $w\in\rootsR_{z_m}$ if $m$ is even.} 
\end{dcases}
\eeqq

From the definitions, 
\begin{equation} \label{eq:fJhrela}
\lemh_\ell(w) = \begin{dcases}
\frac{J(w) f_\ell(w)}{w^N}&  \text{for $w\in\rootsL_{z_\ell}$ if $\ell=1$ or $m$,}\\
J(w) f_\ell(w) & \text{for $w\in\rootsL_{z_\ell}$ if $2\le \ell \le m-1$,}\\
\frac{J(w)}{w^N f_\ell(w)} &  \text{for $w\in\rootsR_{z_\ell}$ if $\ell=1$ or $m$,}\\
\frac{J(w)}{f_\ell(w)}& \text{for $w\in\rootsR_{z_\ell}$ if $2\le \ell \le m-1$.}
\end{dcases}
\end{equation}
These relations imply that
\beqq
(\lambda\mu)(w)=\begin{dcases}
	1& \text{for $w\in\rootsL_{z_1}$,}\\
	\frac{z_{2k+1}^L-z_{2k}^L}{z_{2k+1}^L}& \text{for $w\in\rootsL_{z_{2k+1}}$,}\\
	\frac{z_{2k}^L}{z_{2k}^L-z_{2k-1}^L}& \text{for $w\in\rootsR_{z_{2k}}$,}
\end{dcases}
\eeqq
and
\beqq
(\lambda'\mu')(w)=\begin{dcases}
	1& \text{for $w\in\rootsR_{z_1}$,}\\
	\frac{z_{2k+1}^L}{z_{2k+1}^L-z_{2k}^L}& \text{for $w\in\rootsR_{z_{2k+1}}$,}\\
	\frac{z_{2k}^L-z_{2k-1}^L}{z_{2k}^L}& \text{for $w\in\rootsL_{z_{2k}}$,}
\end{dcases}
\eeqq
for $k\ge 1$.
Thus, \eqref{eq:scalar_relation} is satisfied, and Lemma~\ref{lm:conjugation} implies the equation~\eqref{eq:identity_2}.

\section{Large time limit of multi-point distributions}\label{eq:limittheorem}

We consider the sequence of $\PTASEP(L, N_L, Y_L)$ where $Y_L$ is the initial condition and $N=N_L$ is a sequence of integers indexed by $L$ satisfying $1\le N\le L$. 
We now consider the limit of the multi-point distribution as the period $L\to \infty$\footnote{\label{footnotelab}We allow the possibility that the sequence $N=N_L$ does not exist for some values of $L$. In that case, we take the limit $L\to \infty$ in the set $\{L: N_L \ \text{exists}\}$, which we assumed to be an infinite set. The flat initial condition (see the definition~\ref{def:special_IC}) is an example of such a case where we take $N= L/d$ for a positive integer $d$.} and time also tends to infinity in the relaxation time scale. 
The limit was obtained for the step initial condition case in  \cite{Baik-Liu17}.
In this section, we consider the limit for general initial conditions $Y_L$ which satisfy certain assumptions. 
We assume that the average density 
\begin{equation*}
\rho= \rho_L:= \frac{N}{L}
\end{equation*}
stays in a compact subset of $(0,1)$ for all $L$.

The finite-time formula of the multi-distribution function changes from the step initial condition to a general initial condition $Y$ in the factor $\energy_{Y}(z_1)$ of  the function $\scrC_Y(\bz)$ (see \eqref{eq:def_C0}) and the term $\ich_Y (w,w'; z_1)$ of the kernel $\scrKtwo$ (see \eqref{eq:aux_2018_04_12_02}.) 
We assume that the initial condition $Y_L$ satisfies certain asymptotic assumptions so that $\energy_{Y}(z_1)$ and $\ich_Y (w,w'; z_1)$ converge. 
Under these assumptions, we obtain a limit theorem for the joint distribution of the height function at multiple points in the space-time coordinates.
In later sections, we show that the flat and step-flat initial conditions satisfy the mentioned assumptions.

In the asymptotic analysis of the finite-time formula, it turns out that the integral parameter $z$ needs to be changed to a re-scaled parameter $\mathrm{z}$ as follows:
\begin{equation} \label{eq:aux_2018_04_12_01}
z^L=(-\rho)^N(1-\rho)^{L-N}\mathrm{z} =(-1)^N \rr^L \mathrm{z},
\qquad \text{where $\rr= \rho^\rho(1-\rho)^{1-\rho}$}.
\end{equation}
The condition $|z|<\rr$ is equivalent to the condition $|\mathrm{z}|<1$.
We will assume the above change and assume $|\mathrm{z}|<1$ throughout this section. 
For $z_\ell$, we use a similar change using $\mathrm{z}_\ell$.

\subsection{Convergence of Bethe roots}\label{sec:setupbr}

The finite-time formula, Theorem \ref{thm:Fredholm}, involves the Bethe roots, which are the solutions of the algebraic equation
\begin{equation} \label{eq:Betheequationagain}
w^N(w+1)^{L-N} =z^L. 
\end{equation}
If we scale $z$ as \eqref{eq:aux_2018_04_12_01}
and scale $w$ as 
\beqq
w= -\rho + \zeta \sqrt{\rho(1-\rho)} L^{-1/2}, 
\eeqq
then 
the equation \eqref{eq:Betheequationagain} becomes 
\begin{equation} \label{eq:zermz}
e^{-\zeta^2/2} = \mathrm{z}
\end{equation}
as $L\to \infty$. 
The solutions of the equation \eqref{eq:zermz} form a discrete set. See Figure~\ref{fig:limiting_nodes}.

\begin{figure} \centering
	\includegraphics[scale=0.6]{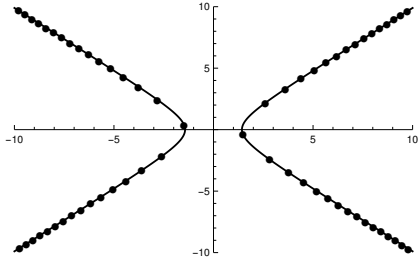}
	\caption{The points are the roots of the equation $e^{-\zeta^2/2} = \mathrm{z}$ for $\mathrm{z}=0.3+0.2\ii$. The solid line is the level curve of $|e^{-\zeta^2/2}|=|\mathrm{z}|$ for the same $z$.}
	\label{fig:limiting_nodes}
\end{figure}

\begin{defn}
	For $\mathrm{z}$ satisfying $0<|\mathrm z|<1$,  define the discrete sets 
	\begin{equation} \label{eq:LRzlimt}
	\inodesL_{\mathrm z} := \{ \xi:   e^{-\xi^2/2}  = \mathrm{z}, \ \Re (\xi) <0 \} \quad \text{and} \quad  
	\inodesR_{\mathrm z} := \{ \eta:  e^{-\eta^2/2} = \mathrm{z}, \ \Re (\eta) >0 \}.
	\end{equation}
\end{defn}

The convergence of the Bethe roots near the point $w=-\rho$ to the solutions of the above equation is precisely stated in the next lemma. 
Recall the sets (see \eqref{eq:def_rootsRL}) of the left Bethe roots and the right Bethe roots, 
\begin{equation*}
\rootsL_z:=\{w\in\roots_z: w^N(w+1)^{L-N} =z^L, \, \Re(w) < -\rho\} 
\end{equation*}
and
\begin{equation*}
\rootsR_z:=\{w\in\roots_z: w^N(w+1)^{L-N} =z^L , \, \Re(w) > -\rho\}.
\end{equation*}
We denote by $\mathbb{D}(a, r)=\{w\in\complexC: |w-a|\le r\}$ the disk of radius $r$ centered at $a$. The following lemma is from \cite{Baik-Liu16}.

\begin{lm}[Lemma 8.1 of \cite{Baik-Liu16}]  \label{lm:limiting_nodes}
	Fix $0<\epsilon<1/8$ and define   
	\beq  \label{eq:rstpepe}
	\begin{split}
		\rootsL_z^{(\epsilon)}&:= \rootsL_z \cap \mathbb{D} \left( -\rho, \sqrt{\rho(1-\rho)} L^{-1/2+\epsilon} \right),\\ 
		\rootsR_z^{(\epsilon)}&:= \rootsR_z \cap \mathbb{D} \left( -\rho, \sqrt{\rho(1-\rho)} L^{-1/2+\epsilon} \right) .
	\end{split}
	\eeq 
	Then, for $z$ and $\mathrm{z}$ related by~\eqref{eq:aux_2018_04_12_01} with $|\mathrm{z}|<1$, there are two injective maps 
	\beqq
	\mathcal{M}_{L,\mathrm{left}} : 	\rootsL_z^{(\epsilon)} \to \inodesL_{\mathrm{z}} 
	\quad \text{and} \quad
	\mathcal{M}_{L,\mathrm{right}}: \rootsR_z ^{(\epsilon)} \to \inodesR_{\mathrm{z}}
	\eeqq
	satisfying 
	\begin{equation*} 
	\left| \mathcal{M}_{L,\mathrm{left}}(u) - \frac{u+\rho}{\sqrt{\rho(1-\rho)}}L^{1/2}\right|  \le L^{ -1/2+ 3\epsilon} \log L
	\qquad  \text{for $u\in \rootsL_z^{(\epsilon)}$}
	\end{equation*}
	and 
	\begin{equation*} 
	\left| \mathcal{M}_{L,\mathrm{right}}(v) - \frac{v+\rho}{\sqrt{\rho(1-\rho)}}L^{1/2}\right| \le L^{ -1/2 +3\epsilon} \log L
	\qquad \text{for $u\in \rootsR_z^{(\epsilon)}$}
	\end{equation*}
	for all large enough $L$. 
	Furthermore, these maps satisfy 
	\beqq
	\inodesL_{\mathrm{z}} \cap \mathbb{D}(0, L^{\epsilon}-1) \subset \mathcal{M}_{L,\mathrm{left}} (\rootsL_z^{(\epsilon)}) \subset \inodesL_{\mathrm{z}} \cap \mathbb{D} (0, L^{\epsilon}+1)
	\eeqq
	and
	\beqq
	\inodesR_{\mathrm{z}} \cap \mathbb{D}(0, L^{\epsilon}-1)   \subset \mathcal{M}_{L,\mathrm{right}} (\rootsL_z^{(\epsilon)}) \subset \inodesR_{\mathrm{z}} \cap \mathbb{D}(0, L^{\epsilon}+1).
	\eeqq
\end{lm}

We remark that we made two modifications in statement of the above lemma compared to Lemma 8.1 of \cite{Baik-Liu16}. First, We redefined $\epsilon$ which does not affect the statement. 
Second, we replaced  $N^{-1/2+\epsilon}$, $N^{-1/2+3\epsilon}$ and $N^\epsilon\pm 1$ by $L^{-1/2+\epsilon}$, $L^{-1/2+3\epsilon}$ and $L^\epsilon\pm 1$ respectively since we use $L$ as the large parameter in this paper.  This modification are easily justified by tracking the error terms in the proof of Lemma 8.1 of \cite{Baik-Liu16}.

\subsection{Assumptions on the initial condition}

We now state the assumptions on the sequence of the initial conditions $Y_L$ under which we prove the limit theorem. 
The conditions are in terms of the global energy function and the characteristic function in Definition~\ref{defn:energy_ich}.

\begin{assum} \label{def:asympstab} 
	We assume that the sequence of the initial conditions $Y_L$ satisfies the following three conditions as $L\to \infty$.

	\begin{enumerate}[(A)]
		\item (Convergence of global energy) 
		There exist a constant $r\in (0,1)$ and a non-zero function $E_{\mathrm{ic}}(\mathrm{z})$ such that for every $0<\epsilon<1/2$, 
		\begin{equation*} 
		\energy_{Y_L} (z) = E_{\mathrm{ic}}(\mathrm{z}) \left( 1+ O(L^{\epsilon -1/2})\right)
		\end{equation*}
		uniformly for $ |\mathrm{z}|<  {r}$  as $L\to \infty$.  
		
		\item (Convergence of characteristic function) 
		There exist constants $0<r_1<r_2<1$ and a function $\chi_{\mathrm{ic}}(\xi,\eta;\mathrm{z})$ 
		such that for every $0<\epsilon<1/8$,  
		\begin{equation*} 
		\ich_{Y_L}(v,u; z) = \chi_{\mathrm{ic}}(\eta,\xi;\mathrm{z}) + O(L^{4\epsilon - 1/2})
		\end{equation*}
		uniformly for $ r_1<|\mathrm{z}|<  r_2$, $u\in \rootsL_z^{(\epsilon)}$ and $v \in \rootsR_z^{(\epsilon)}$ as $L\to \infty$
		where
		\begin{equation*}
		\xi = \mathcal{M}_{L,\mathrm{left}} (u)\in \inodesL_{\mathrm{z}} 
		\quad \text{and}\quad 
		\eta = \mathcal{M}_{L,\mathrm{right}}(v) \in \inodesR_{\mathrm{z}}
		\end{equation*}
		are the images under the maps defined in Lemma \ref{lm:limiting_nodes}.

		\item (Tail control of characteristic function) 
		Let $r_1$ and $r_2$ be same as in (B). 
		There are constants $0<\epsilon<1/8$ and $\epsilon', C>0$ such that for all sufficiently large $L$ and for all $\mathrm{z}$ in $r_1<|\mathrm{z}| <r_2$, 
		\begin{equation} \label{eq:aux_2018_04_12_07}
		|\ich_{Y_L}(v,u; z)| \le e^{C\max\{|\xi|,|\eta|\}^{3-\epsilon'}}
		\end{equation}
		for either $u\in \rootsL_z \setminus \rootsL_z^{(\epsilon)}$ or $v\in \rootsR_z\setminus \rootsR_z^{(\epsilon)}$ where 	\begin{equation*}
		\xi = \begin{dcases}
		\mathcal{M}_{L,\mathrm{left}} (u)& \text{for $u\in \rootsL_z^{(\epsilon)}$},\\
		\frac{L^{1/2}(u+\rho)}{\sqrt{\rho(1-\rho)}}& \text{for $u\in \rootsL_z \setminus \rootsL_z^{(\epsilon)}$},
		\end{dcases}
		\end{equation*}
		and
		\begin{equation*}\eta =\begin{dcases}
		\mathcal{M}_{L,\mathrm{right}} (v)& \text{for $v\in \rootsR_z^{(\epsilon)}$},\\
		\frac{L^{1/2}(v+\rho)}{\sqrt{\rho(1-\rho)}} & \text{for $u\in \rootsR_z \setminus \rootsR_z^{(\epsilon)}$.}
		\end{dcases}
		\end{equation*}
	\end{enumerate}
\end{assum}

A sufficient condition for part (C) to hold is:
there are constants $\epsilon'', C'>0$ such that 
\beq\label{eq:easier_tail_estimates}
|\ich_{Y_L}(v,u; z)| \le C' L^{\epsilon''}
\eeq
for all $(v,u)\in\rootsR_z\times\rootsL_z$ for all $r_1<|\mathrm{z}|<r_2$. 
This is because the right hand side of~\eqref{eq:aux_2018_04_12_07} is at least $e^{CL^{\epsilon}}$ by the assumption on $u, v$.

\medskip

We address a small issue in defining the characteristic function. It is defined under the condition that the global energy function is not zero. 
Note that since $E_{\mathrm{ic}}(\mathrm{z})$ is an uniform limit of a sequence of analytic functions $\energy_{Y_L} (z)$, it is analytic in $|z|<r$. 
Since we assumed that it is a non-zero function, $E_{\mathrm{ic}}(\mathrm{z})$ has only finitely many zeros in any compact subset of $|z|<r$. 
This implies that we can find $0<r_1'<r_2'<1$ such that $\energy_{Y_L} (z)$ does not vanish for $r_1'<|z|<r_2'$ for all large enough $L$, and $\ich_{Y_L}(v,u; z)$ is well-defined. 
We also mention that since the Bethe roots depend on $z$ analytically, $\chi_{\mathrm{ic}}(\eta,\xi;\mathrm{z})$ is an analytic function of $\mathrm{z}$ in the region $r_1<|\mathrm{z}|<r_2$ if we view $\xi\in \inodesL_{\mathrm{z}} ,\eta\in\inodesR_{\mathrm{z}}$ as functions of $\mathrm{z}$.

\subsection{Limit theorem}

We state the limit theorem in terms of the height function. 
Recall the definition of $\height(\ell,t)$ in~\eqref{eq:height}.
Due to the periodicity of the model, it is enough to consider a region of width $L$ in the spatial direction. 
We use the following moving frame. 
Define the vectors $\mb e_1=(1,0)$ and $\mb e_c = (1-2\rho, 1)$ and set the region 
\begin{equation} \label{eq:thesetRL} \begin{split}
\mathbf{R}_L := & \{ (\ell, t) \in \intZ \times \realR_{\ge 0} : 0 \le \ell -(1-2\rho) t \le L\} \\
= & \{s\mb e_1 + t \mb e_c \in \intZ \times \realR_{\ge 0}: 0\le s\le L\}.
\end{split} \end{equation}
This choice is made for the following reason. 
For the TASEP on the whole line, the step initial condition does not produce shocks since the local density of the particle is a decreasing function as a function of $\ell$ initially. 
However, for the PTASEP with the step initial condition, the local density function is not a decreasing function and hence there are shocks: see \cite{Baik-Liu16} for details. 
The direction of shocks is parallel to the line $\ell= (1-2\rho)t$.
Hence, the region \eqref{eq:thesetRL} is a natural choice to describe the limit theorem for the step initial condition. 
Now we state the finite-time formula for the multi-point distribution for general initial conditions using the formula for the step initial condition. 
Hence, we choose to state the limit theorem for general initial condition using the same region.

To emphasize the initial condition in the limit theorem, we add the subscript or superscript ``$\mathrm{ic}$'' for the terms in the limit which depends on the initial condition. 
As discussed in Section \ref{sec:invariancefinite}, when we describe the particles of the PTASEP, there is a freedom of labelling of the particles.
For the next theorem, we make the following choice of the labelling: we assume that $x_N(0)\le 0 < x_{N+1}(0)$. 
This is equivalent to assume that the initial condition satisfies $y_N\le 0< y_1+L$.
This assumption is a simple labelling convention. 

\begin{thm}[Limit theorem] \label{thm:main}
	Consider a sequence  $\PTASEP(L, N_L, Y_L)$ where $\rho=\rho_L=N/L$ stays in a compact subset of $(0,1)$ and $y_N^{(L)}\le 0<y_1^{(L)}+L$.
	Suppose that the sequence of initial conditions $Y_L$ satisfies Assumption \ref{def:asympstab}.  
	Fix a positive integer $m$ and let $\mathrm{p}_j = (\gamma_j,\tau_j)$ be $m$ points in the region
	\begin{equation*}
	\mathrm{R}:= [0,1] \times \realR_{>0}
	\end{equation*}
	satisfying
	\begin{equation*}
	0< \tau_1<\tau_2<\cdots<\tau_m.
	\end{equation*}
	Then, for $\mb{p}_j=s_j \mb e_1 + t_j \mb e_c$ in $\mathbf{R}_L$ given by 
	\begin{equation} \label{eq:parameters}
	s_j = \gamma_j L, \qquad t_j = \tau_j \frac{L^{3/2}}{\sqrt{\rho(1-\rho)}}
	\end{equation}
	and for every fixed $\mathrm{x}_1,\cdots,\mathrm{x}_m\in\realR$, 
	we have
	\begin{equation} \label{eq:main_theorem}
	\begin{split}
	&\lim_{L\to\infty} \prob_L \left( \bigcap_{j=1}^m \left\{ \frac{\height(\mb p_j)  -(1-2\rho) s_j -(1-2\rho+2\rho^2) t_j}{-2\rho^{1/2}(1-\rho)^{1/2} L^{1/2}} \le \mathrm{x}_j	\right\}\right) \\
	&= \Fic (\mathrm{x}_1,\cdots, \mathrm{x}_m; \mathrm{p}_1,\cdots, \mathrm{p}_m),
	\end{split}
	\end{equation}
	where $\prob_L$ denotes the probability associated to $\PTASEP(L, N_L, Y_L)$. 
	The function $\Fic$ is defined in~\eqref{eq:def_Fic} below. 
	The convergence is locally uniform in $\mathrm{x}_j, \tau_j$, and $\gamma_j$. 
	If $\tau_i=\tau_{i+1}$ for some $i\ge 2$, then the result still holds if we assume that $\mathrm{x}_i < \mathrm{x}_{i+1}$. 
	If $\tau_1=\tau_2$, the result holds if $\mathrm{x}_1 < \mathrm{x}_{2}$ and if we assume a stronger assumption on the tail control of the characteristic function under which the right hand side of~\eqref{eq:aux_2018_04_12_07} is replaced by $e^{C\max\{|\xi|,|\eta|\}^{1-\epsilon'}}$.
\end{thm}

\subsection{Formula of the limiting function}

For the step initial condition, the limiting function $\FS$ was obtained in \cite{Baik-Liu17}. 
The formula involves $\mathrm{C_\step}(\mathbf{z})$ which is a limit of $\scrC_\step(\textbf{z})$
and operator $\mathrm{K}^\step=\mathrm{K}^\step_1\mathrm{K}^\step_2$ which is a limit of $\scrKstep= \scrKstep_1\scrKstep_2$. 
These quantities are described in Subsection \ref{sec:Cforsteplimit} and \ref{sec:Dforsteplimit}.
The operators $\mathrm{K}^\step_1$ and $\mathrm{K}^\step_2$ are defined on the sets 
\begin{equation} \label{eq:limSo}
\mathrm{S}_1:= \inodesL_{\mathrm{z}_1}\cup \inodesR_{\mathrm{z}_2} \cup \inodesL_{\mathrm{z}_3} \cup \cdots 
\cup
\begin{dcases}
\inodesR_{\mathrm{z}_m},& \text{if $m$ is even},\\
\inodesL_{\mathrm{z}_m}, & \text{if $m$ is odd},
\end{dcases} 
\end{equation}
and
\begin{equation} \label{eq:limSt}
\mathrm{S}_2:=\inodesR_{\mathrm{z}_1}\cup \inodesL_{\mathrm{z}_2} \cup \inodesR_{\mathrm{z}_3} \cup \cdots 
\cup
\begin{dcases}
\inodesL_{\mathrm{z}_m},& \text{if $m$ is even},\\
\inodesR_{\mathrm{z}_m}, & \text{if $m$ is odd},
\end{dcases}
\end{equation}
where $\inodesL_{\mathrm{z}}$ and $\inodesR_{\mathrm{z}}$  are the sets  defined in \eqref{eq:LRzlimt}. See Figure~\ref{fig:nodes_space}. 
We express $\Fic$ in terms of the above terms. 

\begin{figure} \centering
	\includegraphics[scale=0.8]{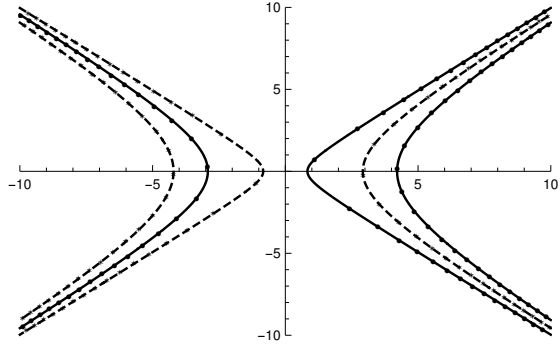}
	\caption{These the sets $\inodesL_{z_i}$ and $\inodesR_{z_i}$ for $i=1,2,3$ corresponding to sets  $\mathrm{S}_1$ and $\mathrm{S}_2$ when $m=3$. The points in $\mathrm{S}_1$ are represented by stars and the points in $\mathrm{S}_2$ are represented by dots. 
		We also include the level curves of $e^{-\zeta^2/2}=|z_i|$ for a better visual: dashed lines for $\mathrm{S}_1$ and solid lines for $\mathrm{S}_2$. 
	}
	\label{fig:nodes_space}
\end{figure}

\begin{defn}[Limiting function] \label{defn:Fic}
	Let $\mathbf{x} = (\mathrm{x}_1,\cdots,\mathrm{x_m})$, $\boldsymbol{\gamma} = (\gamma_1,\cdots,\gamma_m)$, and $\boldsymbol{\tau} = (\tau_1,\cdots,\tau_m)$ be points in $\R^m$ such that $\mathrm{p}_j = (\gamma_j,\tau_j) \in [0,1]\times \realR_{>0}$. 
	Assume that 
	\begin{equation*}
	0< \tau_1\le \cdots\le \tau_m
	\end{equation*}
	and that $\mathrm{x}_i<\mathrm{x}_{i+1}$ when $\tau_i=\tau_{i+1}$ for $i=1,\cdots,m-1$.
	Define 
	\begin{equation}
	\label{eq:def_Fic}
	\Fic(\mathrm{x}_1,\cdots, \mathrm{x}_m; \mathrm{p}_1,\cdots,\mathrm{p}_m) = \oint\cdots\oint \mathrm{C}_{\mathrm{ic}}(\mathbf{z}) \mathrm{D}_{\mathrm{ic}} (\mathbf{z}) \ddbar{\mathrm{z}_m}{\mathrm{z}_{m}} \cdots \ddbar{\mathrm{z}_1}{\mathrm{z}_{1}}, 
	\end{equation}
	where $\mathbf{z} = (\mathrm{z}_1,\cdots,\mathrm{z}_m)$ and the contours are nested circles satisfying $0<|\mathrm{z}_m|<\cdots<|\mathrm{z}_1|<1$ and also, $r_1<|\mathrm{z}_1|<r_2$ with $r_1,r_2$ being the constants in Assumption~\ref{def:asympstab} (B).
	The first function in the integrand is given by 
	\begin{equation}\label{eq:def_rmC}
	\mathrm{C}_{\mathrm{ic}} (\mathbf{z}) = E_{\mathrm{ic}} (\mathrm{z}_1) \mathrm{C_\step} (\mathbf{z}).
	\end{equation}
	The second function is 
	\begin{equation} \label{eq:def_rmD}
	\mathrm{D}_{\mathrm{ic}} (\mathbf{z}) = \det (I- \mathrm{K}^{\mathrm{ic}}), 
	\qquad \mathrm{K}^{\mathrm{ic}}= \mathrm{K}^{\mathrm{ic}}_1 \mathrm{K}^{\mathrm{ic}}_2,
	\end{equation}
	where $\mathrm{K}^{\mathrm{ic}}_1: \ell^2(\mathrm{S}_2) \to \ell^2(\mathrm{S}_1)$ and 
	$\mathrm{K}^{\mathrm{ic}}_2: \ell^2(\mathrm{S}_1) \to \ell^2(\mathrm{S}_2)$ are given by $\mathrm{K}^{\mathrm{ic}}_1= \mathrm{K}_1^\step$ and 
	\begin{equation*}
	\mathrm{K}^{\mathrm{ic}}_2 (\zeta,\zeta'):= 
	\begin{dcases}
	\chi_{\mathrm{ic}} (\zeta,\zeta';\mathrm{z}_1) \mathrm{K}_2^\step (\zeta,\zeta'), \quad  & \text{if $\zeta\in\inodesR_{\mathrm{z}_1} \text{ and } \zeta'\in\inodesL_{\mathrm{z}_1}$,}\\
	\mathrm{K}_2^\step (\zeta,\zeta'), &\text{otherwise.}
	\end{dcases}
	\end{equation*}
\end{defn}

The re-scaled position parameter $\gamma_i$ appears in the above formula only in $\mathrm{K}_1^\step$ and $\mathrm{K}_2^\step$. 
It was shown in \cite{Baik-Liu17} that the product $\mathrm{K}_1^\step\mathrm{K}_2^\step$ is unchanged if we replace $\gamma_i$ by $\gamma_i+1$ for some $i$. 
Using the fact that $\chi_{\mathrm{ic}}$ does not depend on $\gamma_i$, it is easy to see that  $\mathrm{K}_1^\step\mathrm{K}_2^{\mathrm{ic}}$ is also invariant under the same change. 
Hence, we find that for each $i$, the function $\Fic$ can be extended to a continuous periodic function, with period $1$, in the parameter $\gamma_i$. 

For the completeness, we describe $\limCstep (\limz)$, and $\mathrm{K}_1^\step$ and  $\mathrm{K}_2^\step$ explicitly in the next two subsections.
These formulas are from \cite{Baik-Liu17}. 

\subsection{The factor $\limCstep (\limz)$} \label{sec:Cforsteplimit}


Let $\polylog_s(\mathrm{z})$ be the polylogarithm function which is defined by
\begin{equation*}
\polylog_s(\mathrm{z}) = \sum_{k=1}^\infty \frac{\mathrm{z}^k}{ k^s}
\end{equation*}
for $|\mathrm{z}|<1$ and $s\in\complexC$. 
Set
\begin{equation*}
A_1(\mathrm{z})= -\frac{1}{\sqrt{2\pi}} \polylog_{3/2}(\mathrm{z}) \quad \text{and} \quad  A_2(z)=-\frac{1}{\sqrt{2\pi}} \polylog_{5/2}(\mathrm{z}).
\end{equation*}
Let $\log \mathrm{z}$ denote the principal branch of the logarithm function with cut $\realR_{\le 0}$. 
Set
\begin{equation} \label{eq:deofBz}
B(\mathrm{z},\mathrm{z}') = \frac{\mathrm{z} \mathrm{z}'}{2} 
\int\int \frac{\eta\xi \log(-\xi+\eta)}{(e^{-\xi^2/2} - \mathrm{z}) (e^{-\eta^2/2} - \mathrm{z}')}
\ddbar{\xi}{} \ddbar{\eta}{}
= \frac{1}{4\pi} \sum_{k,k'\ge 1} \frac{\mathrm{z}^k (\mathrm{z}')^{k'}}{(k+k')\sqrt{kk'}}
\end{equation}
for $0<|\mathrm{z}|, |\mathrm{z}'|<1$
where the integral contours are the vertical lines $\Re(\xi) = \mathrm{a}$ and $\Re (\eta) = \mathrm{b}$ with  constants $\mathrm{a}$ and $\mathrm{b}$ satisfying
$-\sqrt{-\log|\mathrm{z}|} <\mathrm{a} <0 <\mathrm{b} <\sqrt{-\log|\mathrm{z}'|}$. 
The equality of the integral formula and the series formula is easy to check (see (9.27)-(9.30) of \cite{Baik-Liu16}
for the proof when $\mathrm{z}=\mathrm{z}'$.)
We also set $B(\mathrm{z}):= B(\mathrm{z},\mathrm{z})$. 
One can check that 
\begin{equation}\label{eq:def_Bz}
B(\mathrm{z})= B(\mathrm{z},\mathrm{z}) =\frac{1}{4\pi} \int_0^{\mathrm{z}} \frac{(\polylog_{1/2}(\mathrm{y}))^2}{\mathrm{y}} \dd \mathrm{y}.
\end{equation}

\begin{defn}
	\label{def:limSstep}
	For $\mathbf{z} = (\mathrm{z}_1,\cdots,\mathrm{z}_m)$ satisfying $0<|\mathrm{z}_j|<1$ and $\mathrm{z}_j \ne \mathrm{z}_{j+1}$ for all $j$, we define
	\begin{equation*}	
	\limCstep(\limz) 
	= \left[ \prod_{\ell=1}^{m} \frac{\mathrm{z}_\ell}{\mathrm{z}_\ell -\mathrm{z}_{\ell+1}} \right]
	\left[ \prod_{\ell=1}^{m} \frac{e^{\mathrm{x}_\ell A_1(\mathrm{z}_\ell) +\tau_\ell A_2(\mathrm{z}_\ell)}} {e^{\mathrm{x}_\ell A_1(\mathrm{z}_{\ell+1}) +\tau_\ell A_2(\mathrm{z}_{\ell+1})}}
	e^{2B(\mathrm{z}_\ell) - 2 B(\mathrm{z}_{\ell+1}, \mathrm{z}_\ell)}
	\right],
	\end{equation*}
	where we set $\mathrm{z}_{m+1}=0$. 
\end{defn}

Note that $\limCstep$, and hence $\limC(\limz)$, depend on $\mathrm{x}_i$ and $\tau_i$, but they do not depend on the spatial parameters $\gamma_i$.

\subsection{The operators $\limKstepo$ and $\limKstept$} \label{sec:Dforsteplimit}

Set 
\begin{equation} \label{eq:def_h}
\mathrm{h} (\zeta,\mathrm{z}) = \frac{\mathrm{z}}{2\pi \ii} 
\int_{\ii \R} \frac{\mathrm{w} \log (\mathrm{w}-\zeta)}{e^{-\mathrm{w}^2/2}-\mathrm{z}} \dd \mathrm{w}
\qquad \text{for $\Re(\zeta)<0$ and $|\mathrm{z}|<1$} 
\end{equation}
and define
\beq
\label{eq:2019_10_18_02}
\mathrm{h} (\zeta,\mathrm{z}) = \mathrm{h} (-\zeta,\mathrm{z}) \qquad \text{for $\Re(\zeta)>0$ and $|\mathrm{z}|<1$.} 
\eeq
For each $i$, define 
\begin{equation*}
\mathrm{f}_i(\zeta) = \begin{dcases}
e^{-\frac{1}{3}(\tau_i-\tau_{i-1})\zeta^3 +\frac{1}{2} (\gamma_i-\gamma_{i-1}) \zeta^2 +(\mathrm{x}_i -\mathrm{x}_{i-1})\zeta}& \text{for $\Re(\zeta)<0$,}\\
e^{\frac{1}{3}(\tau_i-\tau_{i-1})\zeta^3 -\frac{1}{2} (\gamma_i-\gamma_{i-1}) \zeta^2 -(\mathrm{x}_i -\mathrm{x}_{i-1})\zeta}& \text{for $\Re(\zeta)>0$,}
\end{dcases}
\end{equation*}
where we set $\tau_0=\gamma_0=\mathrm{x}_0=0$. 
We also define 
\begin{equation*}
\mathrm{Q}_1(j) = 1 -\frac{\mathrm{z}_{j-(-1)^j}}{\mathrm{z}_j} \quad \text{and} \quad
\mathrm{Q}_2(j) = 1 -\frac{\mathrm{z}_{j+(-1)^j}}{\mathrm{z}_j},
\end{equation*}
where we set $\mathrm{z_0} =\mathrm{z}_{m+1}=0$.

\begin{defn} \label{def:mathrmK}
	Let $\inodesL_{\mathrm{z}}$ and $\inodesR_{\mathrm{z}}$ be the discrete sets in \eqref{eq:LRzlimt}, and let $\limSo$ and $\limSt$ be the discrete sets in \eqref{eq:limSo} and \eqref{eq:limSt}. 
	Let 
	\begin{equation*}
	\limKstepo : \ell^2(\mathrm{S}_2) \to \ell^2(\mathrm{S}_1)
	\quad \text{and} \quad  
	\limKstept : \ell^2(\mathrm{S}_1) \to \ell^2(\mathrm{S}_2)
	\end{equation*}
	be the operators defined by their kernels 
	\begin{equation*}
	\limKstepo (\zeta,\zeta') 
	= (\delta_i(j) + \delta_i(j + (-1)^i)) 	\frac{\mathrm{f}_i(\zeta) e^{2\mathrm{h}(\zeta,\mathrm{z}_i)
			-\mathrm{h}(\zeta,\mathrm{z}_{i-(-1)^i}) - \mathrm{h}(\zeta',z_{j-(-1)^j})}}{\zeta(\zeta-\zeta')} \mathrm{Q}_1(j)
	\end{equation*}
	and
	\begin{equation*}
	\limKstept (\zeta',\zeta) 
	= (\delta_j(i) + \delta_j(i - (-1)^j)) 
	\frac{\mathrm{f}_j(\zeta') e^{2\mathrm{h}(\zeta',\mathrm{z}_j) -\mathrm{h}(\zeta',\mathrm{z}_{j+(-1)^j}) - \mathrm{h}(\zeta,z_{i+(-1)^i}) }}{\zeta'(\zeta'-\zeta)}
	\mathrm{Q}_2(i)
	\end{equation*}
	for
	\begin{equation*}
	\zeta \in (\inodesL_{\mathrm{z}_i}\cup \inodesR_{\mathrm{z}_i}) \cap \mathrm{S}_1 \quad \text{and} \quad
	\zeta' \in (\inodesL_{\mathrm{z}_j}\cup \inodesR_{\mathrm{z}_j}) \cap \mathrm{S}_2
	\end{equation*}
	with $1\le i,j\le m$. Define 
	\begin{equation*} 
	\limDstep (\limz) 	= \det( 1- \limKstep), \quad \text{where} \quad \limKstep= \limKstepo\limKstept.
	\end{equation*}
\end{defn}

\bigskip

We will need a few properties of the function $\mathrm{h} (\zeta,\mathrm{z})$ in this paper. We record them here.

\begin{lm}[Properties of $\mathrm{h}(\zeta, \mathrm{z})$] \label{lem:hpropperty}
	The function $\mathrm{h} (\zeta,\mathrm{z})$ satisfies the following properties for each fixed $|z|<1$. 
	
	\begin{enumerate}[(a)]
		\item As a function of $\zeta$, it is  analytic in $\complexC\setminus \ii\R$.
		
		\item For $\zeta\in \ii \R$, the limits of $\mathrm{h}(\eta, \mathrm{z})$ as $\eta\to \zeta$ from the right, $\Re(\eta)>0$, and from the left, $\Re(\eta)<0$, exist. We denote them by $\mathrm{h}_+(\zeta, \mathrm{z})$ and $\mathrm{h}_-(\zeta, \mathrm{z})$, respectively. 
		
		\item There is an alternative formula
		\begin{equation*} 
		\mathrm{h} (\zeta,\mathrm{z}) = -\frac{1}{\sqrt{2\pi}}
		\int_{-\infty + \ii \Im(\zeta)}^{\zeta} \polylog_{1/2} (\mathrm{z} e^{(\zeta^2-\mathrm{y}^2)/2})\dd \mathrm{y} \qquad \text{for $\Re(\zeta)<0$.}
		\end{equation*}
		
		\item It can also be written as 
		\beqq
		\mathrm{h}(\zeta,\mathrm{z}) 
		= \int_{\ii\realR}\frac{\log(1-\mathrm{z}e^{w^2/2})}{w-\zeta}\ddbar{w}{} \qquad \text{for $\Re(\zeta)<0$.}
		\eeqq
		
		\item For $\zeta \in\ii\realR$, we have  $\mathrm{h}_+(\zeta,\mathrm{z}) = \mathrm{h}_-(-\zeta,\mathrm{z})$, and 
		\beqq
		\begin{split}
			\mathrm{h}_+ (\zeta,\mathrm{z}) +\mathrm{h}_- (\zeta,\mathrm{z}) &= \log (1-\mathrm{z}e^{\zeta^2/2}).
		\end{split}
		\eeqq
		In particular, $\mathrm{h}(\zeta,\mathrm{z})$ is continuous at $\zeta=0$. We have 
		$\mathrm{h}(0,\mathrm{z}) = \frac12 \log(1-\mathrm{z})$. 
		
		\item For every positive constant $c$, we have $\mathrm{h}(\zeta,\mathrm{z})= O(\zeta^{-1})$ as  $\zeta\to\infty$ in $\{\zeta\in\complexC: |\Re(\zeta)|>c\}$. 
	\end{enumerate}
\end{lm}

\begin{proof}
	(a) is clear from the definition of the function. 
	(c) is proved in (9.21) of \cite{Baik-Liu16}. 
	(d) can be checked from (c) by using the power series of the polylogarithm function and using the identity  $\frac1{\sqrt{2\pi}} \int_{-\infty}^u e^{-w^2/2} \dd w = \int_{\ii\R} \frac{e^{(-u^2+w^2)/2}}{w-u} \ddbar{w}{}$. 
	(b) and (e) follow from the Cauchy transform formula (d) and properties of Cauchy transform. 
	(f) also follows from (d). 
\end{proof}

\subsection{Proof of Theorem~\ref{thm:main}} \label{sec:pflimitthm}

Theorem~\ref{thm:main} is about the height fluctuation, while Theorem~\ref{eq:multipoint_finite_time} is about the particle locations. Hence we first apply~\eqref{eq:relation_height_particle_locations} and translate the height function in terms of particle locations. After the translation, the probability in the left-hand side of \eqref{eq:main_theorem} (note that $K=N$ in the formula by our assumption that $y_N^{(L)}\le 0<y_1^{(L)}+L$) is equal to 
\begin{equation*}	
\prob_{Y_L}\bigg( \bigcap_{j=1}^m \left\{\bx_{k_j}(t_j) \ge a_j \right\}\bigg),
\end{equation*}
where 
\begin{equation*}
a_j = s_j + (1-2\rho)t_j +O(1) 
\end{equation*}
and
\begin{equation}
\label{eq:aux_028}
k_j = N +\rho s_j - \rho^2 t_j + \mathrm{x}_j \rho^{1/2}(1-\rho)^{1/2} L^{1/2} +O(1).
\end{equation}
Recall \eqref{eq:parameters} that $s_j$ and $t_j$ are scaled as 
\begin{equation*} 
s_j = \gamma_j L \quad \text{and} \quad t_j = \tau_j \frac{L^{3/2}}{\sqrt{\rho(1-\rho)}}.
\end{equation*}

By Theorem~\ref{thm:Fredholm}, we need to prove that 
\begin{equation} \label{eq:aux_023}
\lim_{L\to\infty}\oint\cdots\oint \scrC_{Y_L}(\bz ) \scrD_{Y_L}(\bz ) \ddbar{z_1}{z_1}\cdots \ddbar{z_m}{z_m} = \Fic (\mathrm{x}_1,\cdots, \mathrm{x}_m; \mathrm{p}_1,\cdots, \mathrm{p}_m)
\end{equation}
with the above parameters. 
We change the variables 
\begin{equation*} 
z_j^L =(-1)^N\rr^L \mathrm{z}_j=  (-\rho)^N (1-\rho)^{L-N} \mathrm{z}_j,
\end{equation*}
where $\mathbf{z}= (\mathrm{z}_1,\cdots,\mathrm{z}_m)$ satisfies
\begin{equation*}
r_1<|\mathrm{z}_m| <\cdots< |\mathrm{z}_1| <r_2
\end{equation*}
with $r_1,r_2$ given in Assumption \ref{def:asympstab}.

In Section 6 of \cite{Baik-Liu17}, the limit was evaluated for the step initial condition. 
It was shown that under the same scales, $\scrCstep(\bz )$ converges to $\limCstep(\limz)$
and after simple conjugations the operators $\scrK^\step_1$ and $\scrK^\step_2$ converges to $\limKstepo$ and $\limKstept$ respectively. 
The tail estimates for the operators were also obtained so that the Fredholm determinant converges (see Lemma 6.6 and Section 6.3.4 of \cite{Baik-Liu17}.) 
Hence, to prove \eqref{eq:aux_023}, it is enough to (a) prove the convergence of the global energy function $\energy_{Y_L}(z_1)$, (b) prove the convergence of the characteristic function $\ich_{Y_L}(v,u;z)$, and (c) prove a tail estimate of $\ich_{Y_L}(v,u;z)$. 
But these three properties are exactly the conditions in Assumption \ref{def:asympstab}.
Hence, we obtain Theorem \ref{thm:main}.

\section{Flat and step-flat initial conditions}\label{sec:flatandstepflat}

We consider two particular initial conditions, flat and step-flat initial conditions defined in Definition \ref{def:special_IC}. 
Recall that 
\beqq
L=dN
\eeqq 
for the flat initial condition case, and 
\beqq 
L=dN+L_s
\eeqq
for the step-flat initial condition case where $d\ge 2$ and $0<L_s<L$ are integers. 
The flat initial condition is a special case of the step-flat initial condition, but we state and prove the results separately. 
The initial height functions in one period are given by (See Figure~\ref{fig:step_flat_IC}): 
\begin{enumerate}[(i)]
	\item (flat) $\height_{\mathrm{flat}} (\ell , 0) = |x| +k(d-2)$ for $\ell = x+kd \in\intZ$ with $x,k\in\intZ$ satisfying $-1 \le x \le d-2$ and $-N+1 \le  k\le 0$,  
	\item (step-flat) \begin{equation*} 
	\begin{split}
	\height_{\stof}(\ell, 0)
	= \begin{dcases}
	|x| +k(d-2)& \text{for $\ell = x+kd \in \intZ$ with $x,k\in\intZ$ satisfying}\\ &\text{$-1 \le x \le d-2$ and $-N+1\le k\le 0$}, \\
	\ell& \text{for $\ell\in\intZ$ satisfying $d-1\le \ell \le L_s+d-2$}.
	\end{dcases}
	\end{split}
	\end{equation*}
\end{enumerate}	
They are extended by $\height_{\mathrm{ic}} (\ell +L, 0) = \height_{\mathrm{ic}}(\ell, 0) + L - 2N$ using the periodicity. Note that by definition, the space parameter $\ell\in\intZ$.

\bigskip

We now consider the large $L$ limit. 
In the step-flat initial condition, there is a parameter $L_s$. This parameter denotes the length of the step part in a period of the step-flat initial condition, while $dN=L-L_s$ denotes the length of the flat part.
If $L_s=O(1)$, then the limit is same as the flat initial condition.  
On the other hand, if $L_s=O(N)=O(L)$, then it turns out that the limit is same as the step initial condition; see Appendix \ref{sec:prob_step_flat} for an argument using using a directed last passage percolation interpretation of the PTASEP. 
The interpolating regime turns out to be $L_s=O(\sqrt{L})$. 
We will focus on this case. 

The following result shows that both initial conditions satisfy Assumption \ref{def:asympstab}. 
Hence, the limit theorem \ref{thm:main} is applicable and we obtain the limit of the multi-point distributions for these two initial conditions. 
The proof of the following theorem is given in the next two sections.

\begin{thm}\label{thm:special_IC}
	We have the following: 
	\begin{enumerate}[(i)]
		\item (flat) For the flat initial condition, Assumption~\ref{def:asympstab} holds with 
		\begin{equation*} 
		E_{\mathrm{flat}}(\mathrm{z}) = (1-\mathrm{z})^{-1/4}  e^{-B(\mathrm{z})}
		\quad \text{and}\quad
		\chi_{\mathrm{flat}} (\eta,\xi;\mathrm{z}) = \delta_\xi(-\eta) e^{-\mathrm{h}(\xi,\mathrm{z})-\mathrm{h}(\eta,\mathrm{z})}\eta(\eta-\xi)
		\end{equation*}
		for $0<|\mathrm{z}|<1$, where $B(\mathrm{z})$ is defined in~\eqref{eq:def_Bz}, $\mathrm{h}(\zeta,\mathrm{z})$ is defined in~\eqref{eq:def_h} and~\eqref{eq:2019_10_18_02}, and $\delta_\xi(-\eta)=1$ if $\eta=-\xi$, and is equal to $0$ otherwise.
		
		\item (step-flat) For the step-flat initial condition, if  
		\begin{equation*}
		L_s= \mu \sqrt{d-1}  L^{1/2} +O(1)
		\end{equation*} 
		for a fixed constant $\mu>0$, then Assumption~\ref{def:asympstab} holds with 
		\begin{equation} \label{eq:sf2e}
		E_{\stof} (\mathrm{z}) = \exp\left(-\frac{1}{2}\mathrm{h}(\mu,\mathrm{z}) +\frac{\mathrm{z}^2}{2}\int_{\ii\realR}\int_{\ii\realR} 
		\frac{ \eta \eta' \log (\eta +\eta' +2\mu) }{(e^{-\eta^2/2}-\mathrm{z}) (e^{-(\eta')^2/2}-\mathrm{z})}\ddbar{\eta'}{} \ddbar{\eta}{}\right)
		\end{equation}
		and $\chi_{\stof}(\eta,\xi;\mathrm{z})$ is given by 
		\begin{equation} \label{eq:sf2chi}
		\begin{split}
		\chi_{\stof}(\eta,\xi;\mathrm{z}) 
		&= \frac{2(\eta +\mu)}{\xi+\eta +2\mu} 
		e^{\mathrm{h}(-\xi-2\mu,\mathrm{z})-\mathrm{h}(-\eta-2\mu,\mathrm{z})} \qquad \text{for $\Re(\xi+2\mu)> 0$}
		\end{split}
		\end{equation}
		and analytically continued for $\xi \in \C$. 
		The above result holds for $0<|\mathrm{z}|< 1$, 
		where the function $\mathrm{h}(\zeta,\mathrm{z})$ is defined in~\eqref{eq:def_h}. 
	\end{enumerate}
\end{thm}

In the formula \eqref{eq:sf2chi}, $\eta$ is in $\inodesR_{\mathrm{z}}$. 
This implies that $\Re(\eta)>0$, and hence $\Re(\eta+2\mu)>0$.
Thus, $\mathrm{h}(-\eta-2\mu,\mathrm{z})$ is well-defined for all $\eta\in \inodesR_{\mathrm{z}}$. 
On the other hand, $\xi$ is in $\inodesL_{\mathrm{z}}$, and hence $\Re(\xi)<0$.
This implies that $\Re(\xi+2\mu)$ can be positive or negative or zero. 
When $\Re(\xi+2\mu)>0$, we use the formula \eqref{eq:sf2chi}. 
If $\Re(\xi+2\mu)\le 0$, then we use the analytic continuation. 
The analytic continuation for $\xi \in \C$ is possible by 
Lemma~\ref{lem:hpropperty} (e). 
The explicit formula is given by 
\begin{equation} \label{eq:sf2chi2}
\begin{split}
\chi_{\stof}(\eta,\xi;\mathrm{z}) 
&= \frac{2(\eta +\mu)}{\xi+\eta +2\mu}(1-\mathrm{z}e^{(\xi+2\mu)^2/2})
e^{-\mathrm{h}(\xi+2\mu,\mathrm{z})-\mathrm{h}(-\eta-2\mu,\mathrm{z})} 
\end{split}
\end{equation}
for $\Re(\xi+2\mu)< 0$.

\bigskip

The step-flat initial condition interpolates the flat and the step initial conditions. 
We can check this property for the limits of the global energy function and the characteristic function directly.

\begin{prop} \label{prop:crossover}
	The functions $E_{\stof} (\mathrm{z})$ and $\chi_{\stof}(\eta,\xi;\mathrm{z})$ in \eqref{eq:sf2e} and \eqref{eq:sf2chi}, \eqref{eq:sf2chi2} satisfy the following properties. 
	
	\begin{enumerate}[(i)]
		\item As $\mu\to 0$, 
		\begin{equation*}
		E_{\stof} (\mathrm{z})\to E_{\flat}(\mathrm{z}) \qquad \text{ and } \qquad
		\chi_{\stof}(\eta,\xi;\mathrm{z}) \to \chi_{\flat}(\eta,\xi;\mathrm{z}). 
		\end{equation*}
		
		\item 
		As $\mu\to\infty$,
		\begin{equation*}
		E_{\stof} (\mathrm{z})\to E_{\step}(\mathrm{z})=1 \qquad \text{ and } \qquad
		\chi_{\stof}(\eta,\xi;\mathrm{z}) \to \chi_{\step}(\eta,\xi;\mathrm{z})=1. 
		\end{equation*}
	\end{enumerate}
\end{prop}

\begin{proof}
	(i) Consider $E_{\stof}(\mathrm{z})$ as $\mu\to 0$. 
	By Lemma \ref{lem:hpropperty} (c), $e^{-\frac12 \mathrm{h}(0, z)}= (1-z)^{-1/4}$. 
	For the double integral, we first deform the contours to $c+\ii\realR$ for any constant $c>0$ satisfying $e^{-c^2/2}> |\mathrm{z}|$. 
	The limit as $\mu\to 0$ is equal to 
	\begin{equation*}
	\frac{\mathrm{z}^2}{2}\int_{c+\ii\realR}\int_{c+\ii\realR} 
	\frac{ \eta \eta' \log (\eta +\eta' ) }{(e^{-\eta^2/2}-\mathrm{z}) (e^{-(\eta')^2/2}-\mathrm{z})}\ddbar{\eta'}{} \ddbar{\eta}{}.
	\end{equation*}
	Changing $\eta$ to $-\eta$, this is equal to $-B(\mathrm{z})$ in \eqref{eq:deofBz}. 
	Thus we obtain $E_{\stof} (\mathrm{z})\to (1-z)^{-1/4} e^{-B(\mathrm{z})}= E_{\flat}(\mathrm{z})$. 
	
	Consider $\chi_{\stof}(\eta,\xi;\mathrm{z})$ as $\mu\to 0$. 
	We only need to consider $\xi$ satisfying $\xi \in  \inodesL_{\mathrm{z}}$. 
	This implies that $\Re(\xi)<0$, and hence $\Re(\xi+2\mu)<0$ for all small enough $\mu$. 
	Thus, by \eqref{eq:sf2chi2}, the formula in this case is 
	\beqq
	\chi_{\stof}(\eta,\xi;\mathrm{z}) 
	= \frac{2(\eta +\mu)}{\xi+\eta +2\mu}(1-\mathrm{z}e^{(\xi+2\mu)^2/2})
	e^{-\mathrm{h}(\xi+2\mu,\mathrm{z})-\mathrm{h}(-\eta-2\mu,\mathrm{z})}.		
	\eeqq
	Since $\xi\in \inodesL_{\mathrm{z}}$, we have $e^{-\xi^2/2}=\mathrm{z}$. 
	Hence, the limit of $1-\mathrm{z}e^{(\xi+2\mu)^2/2}$ as $\mu\to 0$ is zero. 
	This implies that if $\xi+\eta\neq 0$, then $\chi_{\stof}(\eta,\xi;\mathrm{z})$ converges to zero as $\mu\to 0$. 
	On the other hand, if $\xi+\eta=0$, then, using $\mathrm{z}= e^{-\xi^2/2}$, 
	\beqq
	\frac{1-\mathrm{z}e^{(\xi+2\mu)^2/2}}{\xi+\eta+2\mu} = \frac{1-e^{2\xi \mu + 2\mu^2}}{2\mu} \to -\xi.
	\eeqq
	Therefore, in this case, $\chi_{\stof}(\eta,\xi;\mathrm{z})\to 2\xi^2 e^{-2\mathrm{h}(\xi, \mathrm{z})}$. 
	Hence, $\chi_{\stof}(\eta,\xi;\mathrm{z}) \to \chi_{\flat}(\eta,\xi;\mathrm{z})$ as $\mu\to 0$. 
	
	(ii) Consider $E_{\stof}(\mathrm{z})\to 1$ as $\mu\to\infty$.
	Lemma \ref{lem:hpropperty} (f) implies that $\mathrm{h}(\mu, \mathrm{z})\to 0$ as $\mu\to\infty$. 
	For the double integral, note that 
	\beqq
	\int_{\ii\realR}\int_{\ii\realR} \frac{ \eta \eta'  }{(e^{-\eta^2/2}-\mathrm{z}) (e^{-(\eta')^2/2}-\mathrm{z})}\ddbar{\eta'}{} \ddbar{\eta}{}
	= \left( \int_{\ii\realR} \frac{ \eta }{ e^{-\eta^2/2}-\mathrm{z} } \ddbar{\eta}{} \right)^2=0 
	\eeqq
	by the oddness of the integrand. 
	Hence, we may replace the term $\log(\eta+\eta'+2\mu)$ by $\log(\frac{\eta+\eta'+2\mu}{2\mu})$ in the double integral without changing the value. 
	Since $\log(\frac{\eta+\eta'+2\mu}{2\mu}) \to 0$, we find that 
	$E_{\stof}(\mathrm{z})\to 1$ as $\mu\to\infty$.
	
	For $\chi_{\stof}(\eta,\xi;\mathrm{z})$, we note that $\Re(\xi+2\mu)>0$ for all large enough $\mu$. 
	Hence, Lemma \ref{lem:hpropperty} (f) implies that $\chi_{\stof}(\eta,\xi;\mathrm{z})$ converges to $1$ as $\mu\to \infty$.
\end{proof}

\section{Asymptotic evaluations of various products}
\label{sec:asymptotis_products}

The proof of Theorem \ref{thm:special_IC} involves the limits of various products about Bethe roots. 
In this section, we summarize some of the limits obtained from \cite{Baik-Liu16} and \cite{Baik-Liu17}. 
The results in this section do not depend on the initial conditions.

Fix positive integers $L>N$ and fix a complex number $|z|<\rr$. 
The trajectory $|w^N(w+1)^{L-N}| = |z^L|$ consists of two disjoint closed curves. 
Let $\Lambda_\RR$ be the closed curve satisfying $\Re(w)>-\rho$ and let $\Lambda_\LL$ be the closed curve satisfying $\Re(w)<-\rho$:
\beq \label{eq:lammbdaad1}
\begin{split}
	\Lambda_\RR &:= \{ w\in \complexC : |w^N(w+1)^{L-N}| = |z^L|, \quad \Re(w)>-\rho\}, \\
	\Lambda_\LL &:= \{ w\in \complexC : |w^N(w+1)^{L-N}| = |z^L|, \quad \Re(w)<-\rho\}.
\end{split}
\eeq

Taking the logarithm, the products are changed to sums involving functions of the Bethe roots. 
The next simple lemma transforms the sums into integrals using the residue theorem. 

\begin{lm}[Lemma 9.1 of \cite{Baik-Liu16}] \label{lem:sumtointres}
	
	Recall the Bethe polynomial 
	\beqq
	q_z(w)= w^N(w+1)^{L-N}-z^L. 
	\eeqq
	
	\begin{enumerate}[(a)]
		\item Let $p(w)$ be a function which is analytic in the interior and also in a neighborhood of $\Lambda_\RR$. Then, 
		\beqq
		\sum_{v\in\rootsR_z} p(v) = Np(0) + \frac{Lz^L}{2\pi \ii} \oint \frac{p(w)( w+\rho)}{ w(w+1) q_z(w)} \dd w,
		\eeqq
		where the integral is over an arbitrary simple closed contour which lies inside the half plane $\Re(w)\ge -\rho$ and  contains the curve $\Lambda_\RR$ inside.  
		\item Let $p(w)$ be a function which is analytic in the interior and also in a neighborhood of $\Lambda_\LL$. Then, 
		\beqq
		\sum_{u\in\rootsL_z} p(u) = (L-N)p(-1) + \frac{Lz^L}{2\pi \ii} \oint \frac{p(w)( w+\rho)}{ w(w+1) q_z(w) } \dd w,
		\eeqq
		where the integral is over an arbitrary simple closed contour which lies inside the half plane $\Re(w)\le -\rho$  and  contains the curve $\Lambda_\LL$ inside. 
	\end{enumerate}
\end{lm}

\begin{proof}
	The result follows directly from the residue theorem. 
\end{proof}

Using the above lemma and applying the method of steepest-descent, we can show the following. 

\begin{lm} \label{lem:asymofprod}
	If $z^L= (-1)^N\rr^L \mathrm{z}$ with $|\mathrm{z}|<1$ and 
	$\rho=N/L$ staying in a compact subset of $(0,1)$, then for every $0<\epsilon<1/2$ the following holds for all large enough $L$. 
	\begin{enumerate}[(i)]
		\item (Lemma 8.2 (a) of \cite{Baik-Liu16}) For $w$ which is a finite distance away from the trajectory $\Lambda_\LL\cup\Lambda_\RR$, 
		\beqq \begin{split}
			\prod_{u\in \rootsL_z} \sqrt{w-u} &= (\sqrt{w+1})^{L-N} \left( 1+ O(L^{\epsilon-1/2})\right) \quad \text{if $\Re(w)>-\rho$}, \\
			\prod_{v\in \rootsR_z} \sqrt{v-w} &= (\sqrt{-w})^N \left( 1+ O(L^{\epsilon-1/2})\right) \qquad \text{if $\Re(w)<-\rho$.}
		\end{split} \eeqq
		
		\item (Lemma 8.2 (a) and (b) of \cite{Baik-Liu16}) 
		For $w= -\rho +\zeta  \sqrt{\rho(1-\rho)} L^{-1/2}$ with $|\zeta|\le L^{\epsilon/4}$, 
		\beqq
		\begin{split}
			\prod_{u\in \rootsL_z} \sqrt{w-u} &= (\sqrt{w+1})^{L-N} e^{\frac12 \mathrm{h} (\zeta,\mathrm{z}) } \left( 1+ O(L^{\epsilon-1/2}\log L)\right) \quad \text{if $\Re(\zeta) \ge 0$},\\
			\prod_{v\in \rootsR_z} \sqrt{v-w} &= (\sqrt{-w})^N e^{\frac12 \mathrm{h} (\zeta,\mathrm{z}) } \left( 1+ O(L^{\epsilon-1/2}\log L)\right) \qquad \text{if $\Re(\zeta) \le 0$},
		\end{split} \eeqq
		where $\mathrm{h} (\zeta,\mathrm{z})$ is the function defined in \eqref{eq:def_h} and~\eqref{eq:2019_10_18_02}. Here, when $\Re(\zeta)=0$, $\mathrm{h}(\zeta, \mathrm{z})$ is the limit of $\mathrm{h}(\eta, \mathrm{z})$ as $\eta\to \zeta$ from $\Re(\eta)>0$ for the first case and from $\Re(\eta)<0$ for the second case.

		\item (Lemma 6.7 of \cite{Baik-Liu17})
		There is a constant $C>0$ such that for every $w$ satisfying $|w+\rho|\ge L^{\epsilon-1/2}$, 
		\begin{equation*}
		e^{-CL^{-\epsilon}}\le \left|\frac{q_{z,\LL}(w)}{(w +1)^{L-N}} \right| \le e^{CL^{-\epsilon}} 
		\quad \text{if $\Re(w)>-\rho$},
		\end{equation*}
		and
		\begin{equation*}
		e^{-CL^{-\epsilon}}\le \left|\frac{q_{z,\RR}(w)}{w^N} \right| \le e^{CL^{-\epsilon}}
		\qquad \text{if $\Re(w)<-\rho$.}
		\end{equation*}

		\item (Lemma 8.2 (c) of \cite{Baik-Liu16}) 
		We have 
		\beqq
		\frac{\prod_{v\in\rootsR_z} \prod_{u\in\rootsL_z} \sqrt{v-u}}
		{\prod_{u\in\rootsL_z} \left(\sqrt{-u}\right)^N \prod_{v\in\rootsR_z} \left( \sqrt{v+1} \right)^{L-N}}
		= 
		e^{-B(\mathrm{z})} \left(1+O(L^{\epsilon-1/2})\right),
		\eeqq
		where $B(\mathrm{z})$ is the function defined in \eqref{eq:def_Bz}. 
	\end{enumerate}
\end{lm} 

\begin{proof} The proofs are in the papers \cite{Baik-Liu16} and \cite{Baik-Liu17}, except for the uniform error bound in (ii). 
	In Lemma 8.2 (a) and (b) of \cite{Baik-Liu16}, the error is proven to be uniform assuming that $\Re(\zeta)\ge c$ for a constant $c>0$ for the first case and assuming that $\Re(\zeta)\le -c$ for the second case.
	However, a quick inspection of the proof shows that these assumptions are not necessary. The proof proceeds by expressing $\sum_{u\in\rootsL_z} \log (w-u)$ as an integral using Lemma \ref{lem:sumtointres} and then applying the method of steepest-descent. 
\end{proof}

\section{Proof of Theorem~\ref{thm:special_IC} (ii)}\label{sec:stepflatassumtions}

The step-flat initial condition is $Y_\stof= (-(N-1)d, \cdots, -2d, -d, 0)$.
To show that Assumption~\ref{def:asympstab} holds, we evaluate the limit of the global energy function and the characteristic function. 
These functions are given in terms of the symmetric function~\eqref{eq:def_gftn}. 
For step-flat initial condition, the equation \eqref{eq:def_lambdaY} becomes 
\beqq
\lambda(Y_\stof) = (0, -(d-1), -2(d-1), \cdots, -(N-1)(d-1))
\eeqq
and the numerator in the formula of the symmetric function~\eqref{eq:def_gftn} becomes a Vandermonde determinant. 
Evaluating it, the symmetric function becomes 
\begin{equation} \label{eq:aux_053}
\gftn_{\lambda (Y_\stof)}(W)
=  \prod_{i=1}^N (w_i+1)^{-(d-1)(N-1)}
\prod_{1\le i<j\le N}
g(w_i, w_j),
\end{equation}
where
\begin{equation}
\label{eq:def_g}
g(w,w') := \frac{w(w+1)^{d-1} -w'(w'+1)^{d-1}}{w-w'}.
\end{equation}
Note that when $w=w'$, 
\beqq
g(w,w)= (dw+1)(w+1)^{d-2}. 
\eeqq

This section is structured as follows. In the subsection~\ref{sec:analyzing_g_function}, we discuss a few properties of the roots of the function $g(w,w')$. The next three subsections prove the conditions (A), (B), (C) of Assumption \ref{def:asympstab}, respectively, thus proving Theorem~\ref{thm:special_IC} (ii). 
Some lemmas stated in these subsections are proved in the last subsection, Subsection \ref{sec:proof_lemmas_stepflat}.

The parameters in Theorem~\ref{thm:special_IC} satisfy $L= dN +L_s$ with $L_s = \mu \sqrt{d-1} L^{1/2} +O(1)$ for a fixed constant $\mu\ge 0$. 
The special case when $L_s=0$ is the flat initial condition case. 
The proof becomes simpler in this case, and we give a shorter proof in Section \ref{sec:proof_special_IC} for the flat initial condition. 
Hence, throughout this section, we assume that $\mu>0$. 

The average density $\rho$ satisfies
\begin{equation*}
\rho =\frac{ N }{ L } = \frac{L - L_s}{dL} =d^{-1} - \mu \frac{\sqrt{d-1}}{d} L^{-1/2} + O(L^{-1}).
\end{equation*}
This implies that
\begin{equation} \label{eq:dintermsofrho}
- d^{-1} = - \rho - \mu \sqrt{\rho(1-\rho)} L^{-1/2} + O(L^{-1}).
\end{equation}
Note that, in particular, $-d^{-1} < -\rho$ for all large enough $L$.

Throughout this section, we assume that 
\begin{equation*}
z^L = (-1)^N \rr^L \mathrm{z},
\quad \text{where} \quad 
\rr= \rho^\rho (1-\rho)^{1-\rho}
\end{equation*}
for a complex number $\mathrm{z}$ satisfying $0<|\mathrm{z}|<1$. 
The analysis of this section also works if $\mathrm{z}$ depends on $L$ but $|\mathrm{z}|$ stays in a compact subset of the interval $(0,1)$. 
However, to make the presentation simple, we assume that $\mathrm{z}$ does not depend on $L$.

\subsection{Roots of the function $g(w,w')$}
\label{sec:analyzing_g_function}

We discuss a few properties of the roots of the function $g(w,w')$ defined in \eqref{eq:def_g}. 
Note that $g(w,w')$ is a symmetric polynomial. 
For each complex number $w$, it is a polynomial of degree $d-1$ in variable $w'$. 
Let 
\begin{equation}
\label{eq:def_U_w}
U(w')=\{w\in\complexC: g(w,w')=0\}
\end{equation}
be the set of the roots for a given $w'$, where the roots are counted with multiplicities. 
Hence, $U(w)$ has $d-1$ elements and we have
\beqq
g(w, w')= \prod_{u\in U(w')} (w-u). 
\eeqq
In particular, $U(0)$ consists of $d-1$ copies of $-1$. 
We can check the following property. Its proof is in Subsection~\ref{sec:proof_U_v}.

\begin{lm}
	\label{lm:roots_stof}
	For $v\in\rootsR_z$ with $0<|z|< \rr$, the elements of $U(v)$ are all distinct. The real part of each element is less than $-d^{-1}$.
\end{lm}

\bigskip

The previous lemma assumes that $v\in \rootsR_z$, in particular that $\Re(v)>-\rho$. 
The next lemma is about the case when $w$ is close to the point $-\rho$ including the case that $\Re(w)<-\rho$. 
In the proof of Theorem~\ref{thm:special_IC} (ii), we will use the result when $w$ is in the discrete set $\rootsL_z$ and is close to the point $-\rho$. 

Fix $0<\epsilon<1/8$ and denote the disk
\begin{equation} \label{eq:diskOmgc}
\Omegac= \{w: |w+\rho| \le \sqrt{\rho(1-\rho)} L^{-1/2+\epsilon}\}. 
\end{equation}
By \eqref{eq:dintermsofrho}, $w=-d^{-1}$ is in $\Omegac$ for sufficiently large $L$.
Consider the set $U(-d^{-1})$. The solutions of the equation $g(-d^{-1}, w')=0$ are necessarily the roots of the polynomial  $p(w')= (-d^{-1})(-d^{-1}+1)^{d-1} - w'(w'+1)^{d-1}$. 
It is straightforward to check by considering the critical points that there is a double root at $w'=-d^{-1}$ and the remaining $d-2$ roots are distinct. The $d-2$ roots have real part less than $-d^{-1}$.   
Since $g(-d^{-1}, w')$ is $p(w')$ divided by $-d^{-1}-w'$, we conclude that all $d-1$ elements of $U(-d^{-1})$ are distinct, one of them is $w'=-d^{-1}$, and the rest $d-2$ has real part less than $-d^{-1}$. 
We denote the elements as 
\begin{equation*}
U(-d^{-1})=\{-d^{-1},c_1,\cdots,c_{d-2}\},
\end{equation*}
where $\Re(c_i)<-\rho$. 
For other points $w\in \Omegac$, we compare the elements of $U(w)$ with the elements of $U(-d^{-1})$. 

\begin{lm} \label{lm:roots_stof_estimate}
	Fix $0<\epsilon<1/8$ and let $\Omegac$ be the disk in \eqref{eq:diskOmgc}. 
	For $w=-\rho+\zeta\sqrt{\rho(1-\rho)}L^{-1/2} \in \Omegac$ with $|\zeta|<L^{\epsilon}$, all $d-1$ elements of $U(w)$ are distinct. 
	One of them is given by 
	\begin{equation*} 
	\hat w= -\rho + (-2\mu -\zeta) \sqrt{\rho(1-\rho)} L^{-1/2} + O(L^{2\epsilon-1}).
	\end{equation*}
	The rest $d-2$ elements are of the form $c_i(w)=c_i +O(L^{2\epsilon-1})$ for $1\le i\le d-2$. Note that $\Re(c_i(w))<-\rho$ for all large enough $L$. \end{lm}

We prove this lemma in  Subsection~\ref{sec:proof_stof_roots}.

Observe that for $w=-d^{-1}$, we have $\zeta=-\mu$ from \eqref{eq:dintermsofrho}.
Hence, in this case $\hat w = -\rho-\mu \sqrt{\rho(1-\rho)} L^{-1/2}+ O(L^{2\epsilon-1})$, which is consistent with \eqref{eq:dintermsofrho}.
We also note that we have  $c_i(-d^{-1}) = c_i$ by definition.

An example of the points $\hat w$ and $c_i(w)$ are shown in Figure~\ref{fig:set_roots_estimate}. 
Note that the lemma shows that the distance between $c_i(w)$ to $c_i$ is much smaller than the distance between $\hat w$ and $-\rho$; one is  $O(L^{2\epsilon-1})$ and the other is $O(L^{-1/2})$. This feature is clearly visible in the picture.

\begin{figure}
	\centering
	\begin{multicols}{2}
		\includegraphics[scale=0.3]{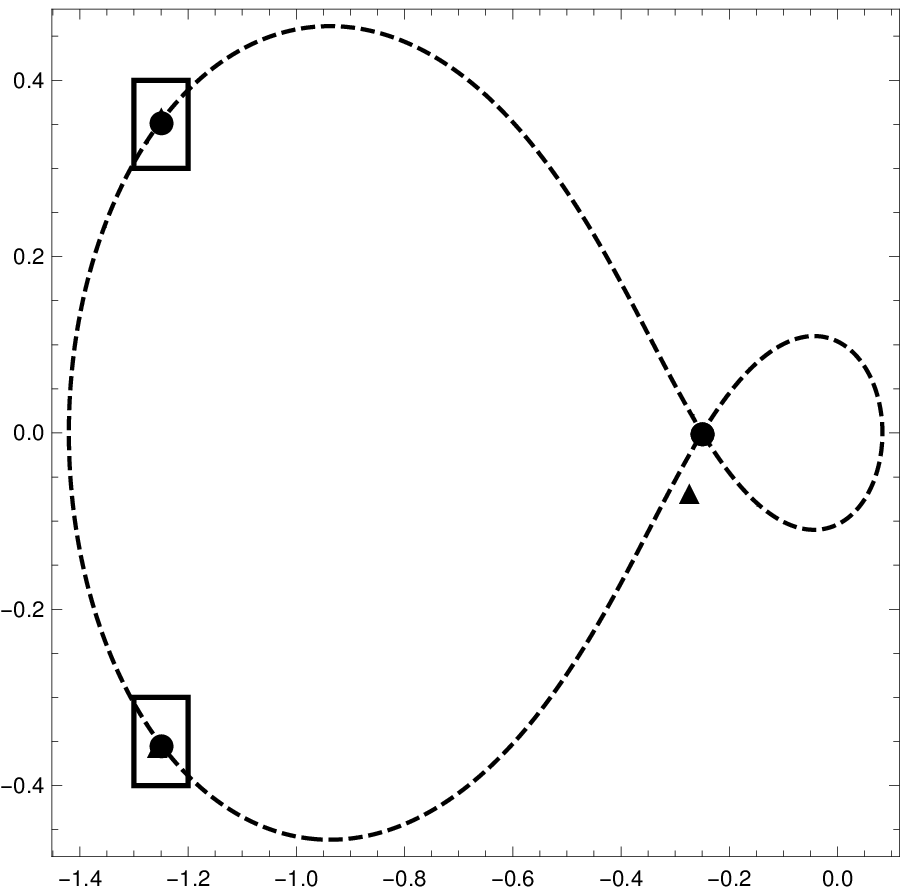}
		
		\includegraphics[scale=0.3]{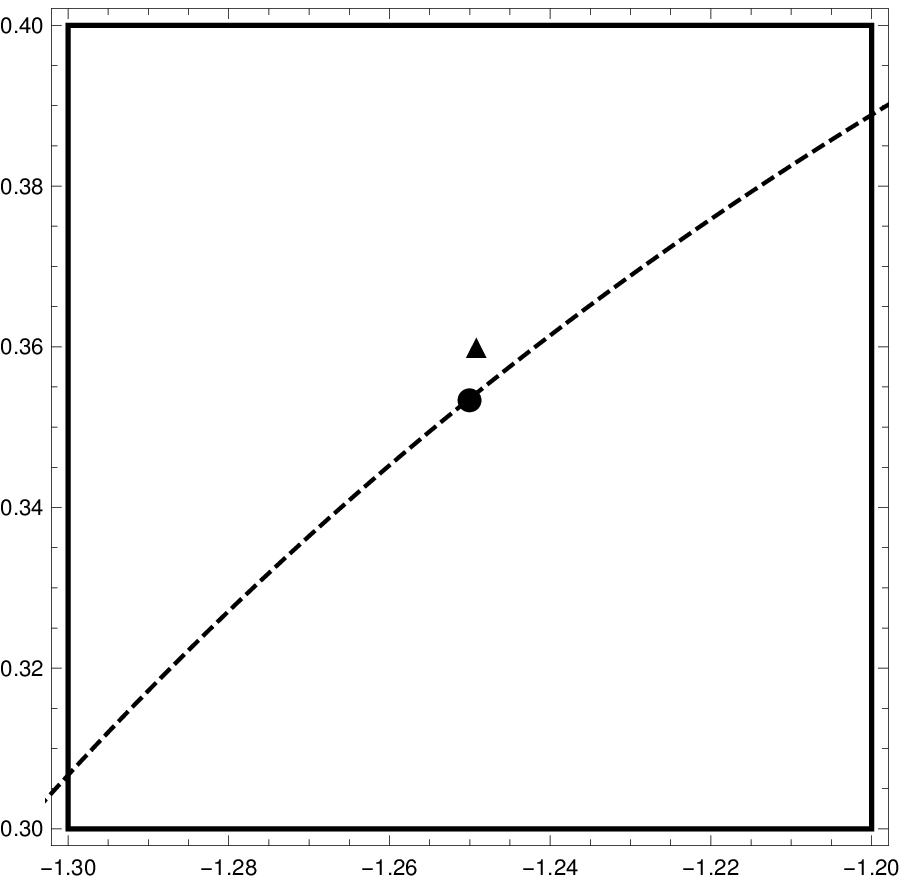}
		\includegraphics[scale=0.3]{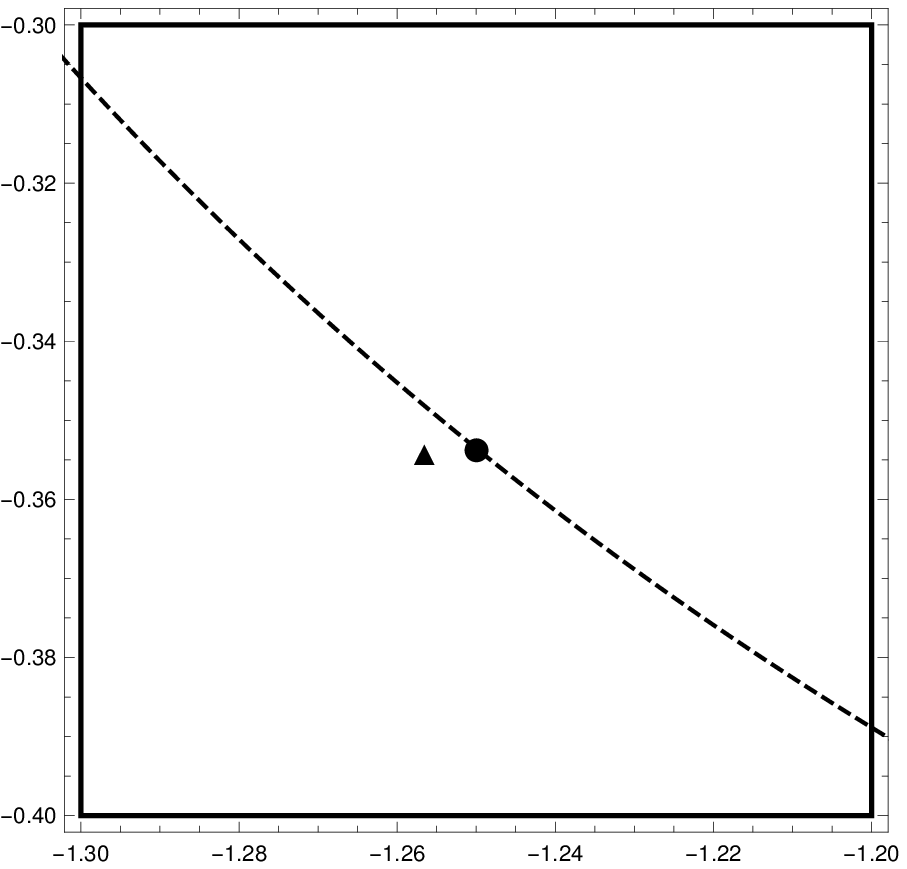}
	\end{multicols}
	\caption{Left: The curve is the trajectory  $|w'(w'+1)^{d-1}|=|-d^{-1}(-d^{-1}+1)^{d-1}|$ when $d=4$. 
		The $\bullet$ points are $U(-d^{-1})=\{ -d^{-1}, c_1. c_2\}$ and the $\blacktriangle$ points are $U(w)$ when $w$ is close to $-d^{-1}$. Two of the points, $c_1(w)$ and $c_2(w)$, are very close to the $\bullet$ points. Right: These are the close-up of the points $c_i(w)$ and $c_i$.}
	\label{fig:set_roots_estimate}
\end{figure}

\subsection{Global energy function and the condition (A)}\label{sec:stepflat_energy}

The global energy function is $\energy_{\mathrm{\stof}}(z)= \gftn_{\lambda (Y_\stof)}(\rootsR_z)$ which is obtained by 
inserting $W=\rootsR_z$, the right Bethe roots, in the formula \eqref{eq:aux_053} of the symmetric function. 
We re-express the formula which makes the asymptotic analysis easier. 
The result is in Lemma \ref{lem:energy_stof} below.

Recall the definition of $U(w)$ given  in~\eqref{eq:def_U_w}. 
Define 
\begin{equation*}
\SU_z := \bigcup_{v\in\rootsR_z} U(v).
\end{equation*}
This set has $(d-1)N$ elements. 
By Lemma \ref{lm:roots_stof}, each element has real part less than $-d^{-1}$. 
See Figure~\ref{fig:set_SU} for an example of the set $\SU_z$.
We note that for the flat initial condition case, $\SU_z=\rootsL_z$: See the discussions after Lemma~\ref{lem:energyflatfor} in Section~\ref{sec:proof_special_IC}. 
For general step-flat initial condition, the set $\SU_z$ is not equal to $\rootsL_z$, but we will think it as a proxy to $\rootsL_z$. 

\begin{figure}
	\centering
	\includegraphics[scale=0.5]{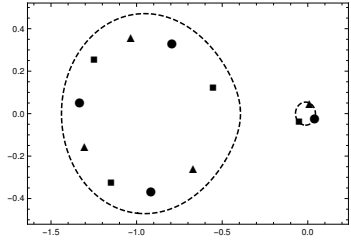}
	\caption{The curve is the trajectory of $|w^N(w+1)^{L-N}|=|z|^L$ when $L=15$, $N=3$, and $z=(0.01+ 0.05\ii)^{1/5}(1.01+0.05\ii)^{4/5}$. 
		Three marked points on the right closed curve are the elements of $\rootsR_z$. 
		When $d=4$, 
		for each $v\in \rootsR_z$, $U(v)$ has $d-1=3$ points. 
		They are indicated by the same marks (disk, triangle, or square) inside the left closed curve. 
		The set of all marked points inside the left close curve is the set $\SU_z$.
	}
	\label{fig:set_SU}
\end{figure}

\begin{lm}[Global energy function for step-flat initial condition] \label{lem:energy_stof}
	We have 
	\begin{equation} \label{eq:energy_stof}
	\begin{split}
	\energy_{\mathrm{\stof}}(z) 
	&= \frac{\prod_{v\in\rootsR_z} \left(\sqrt{v+1}\right)^{d}}{\prod_{v\in\rootsR_z} \sqrt{dv+1}} 
	\left[\frac{\prod_{v\in\rootsR_z} \prod_{u\in\SU_z} \sqrt{v-u} }{\prod_{v\in\rootsR_z} \left(\sqrt{v+1}\right)^{(d-1)N}  \prod_{u\in\SU_z} \left(\sqrt{-u}\right)^{N}}\right] .
	\end{split}
	\end{equation}
	Here and throughout this section, the notation $\sqrt{w}$ denotes the principal branch of the square root function with the cut $\realR_{\le 0}$. 
\end{lm}

\begin{proof}
	By~\eqref{eq:def_energy} and~\eqref{eq:aux_053}, 
	\begin{equation}
	\label{eq:aux_2019_08_16_03}
	\energy_{\mathrm{\stof}}(z)  = \gftn_{\lambda (Y_\stof)}(\rootsR_z)
	= \prod_{v\in\rootsR_z} (v+1)^{-(d-1)(N-1)}   \prod_{1\le i<j\le N} g(v_i, v_j),
	\end{equation}
	where $v_1,\cdots,v_N$ are the elements in the set $\rootsR_z$.
	From the definition of the set $U(w)$ and the symmetry of $g$, we have 
	\begin{equation} \label{eq:gfacpro}
	g(v,v')=\prod_{u'\in U(v')} (v-u') =\prod_{u\in U(v)} (v'-u)
	\end{equation}
	for  $v,v'\in\rootsR_z$.
	Multiplying the two products in the above formula and taking the square-root, 
	\begin{equation*}
	g(v,v')= \prod_{u'\in U(v')} \sqrt{v-u'} \prod_{u\in U(v)} \sqrt{v'-u}
	\end{equation*}
	for $v, v'\in \rootsR_z$. 
	From this we find that 
	\begin{equation}
	\label{eq:aux_2019_08_16_04}
	\prod_{1\le i<j\le N} g(v_i,v_j)=\prod_{v\in\rootsR_z}\prod_{u\in \SU_z\setminus U(v)}\sqrt{v-u}
	= \frac{\prod_{v\in\rootsR_z}\prod_{u\in \SU_z}\sqrt{v-u}}{\prod_{v\in\rootsR_z}\prod_{u\in U(v)}\sqrt{v-u}} .
	\end{equation}
	We now change the denominator of the last expression.
	Note that $g(w,w)=(dw+1)(w+1)^{d-2}$. Hence, setting $v=v'$ in \eqref{eq:gfacpro} and taking the square-root, we find that 
	\beqq
	\prod_{u\in U(v)} \sqrt{v-u} = \sqrt{dv+1}(\sqrt{v+1})^{d-2} . 
	\eeqq
	Taking a product, 
	\begin{equation}
	\label{eq:aux_2019_08_16_05}
	\prod_{v\in\rootsR_z}\prod_{u\in U(v)} \sqrt{v-u} = \prod_{v\in\rootsR_z}\sqrt{dv+1}(\sqrt{v+1})^{d-2} . 
	\end{equation}
	Now, for given $v$, the set of all roots of the polynomial $p(w):=w(w+1)^{d-1}-v(v+1)^{d-1}$ of the variable $w$ is $U(v)\cup \{v\}$. 
	Hence, the product of the solutions satisfies $v \prod_{u\in U(v)} u = (-1)^{d-1} v(1+v)^{d-1}$. This implies that $\prod_{u\in U(v)} \sqrt{-u} = (\sqrt{v+1})^{d-1}$. Thus,  
	\begin{equation}
	\label{eq:aux_2019_08_16_06}
	\prod_{u\in \SU_z} \sqrt{-u} = \prod_{v\in\rootsR_z} (\sqrt{v+1})^{d-1}.
	\end{equation}
	Inserting the the equations~\eqref{eq:aux_2019_08_16_04},~\eqref{eq:aux_2019_08_16_05} and~\eqref{eq:aux_2019_08_16_06} into~\eqref{eq:aux_2019_08_16_03}, we obtain the result. 
\end{proof}

We now show that the condition (A) of Assumption \ref{def:asympstab} is satisfied. 
We take the limit of the formula \eqref{eq:energy_stof} as $L, N\to \infty$. 
Fix $0<\epsilon<1/2$. 
The second equations in Lemma~\ref{lem:asymofprod} (i) when $w=-1$ implies that 
\beqq
\prod_{v\in\rootsR_z}(\sqrt{v+1})^d=1+O(L^{\epsilon-1/2}). 
\eeqq
On the other hand, the second equation of Lemma~\ref{lem:asymofprod} (ii) with $w=-d^{-1}$ implies that 
\begin{equation*}
\prod_{v\in\rootsR_z}\sqrt{dv+1} = e^{\frac{1}{2}\mathrm{h}(-\mu,\mathrm{z})}(1+O(L^{\epsilon-1/2}\log L)),
\end{equation*}
where we used the formula \eqref{eq:dintermsofrho}. 
It remains to find the limit of the term in the big bracket in the formula \eqref{eq:energy_stof}. 
This is given in the following lemma whose proof is given in Subsection~\ref{sec:proof_assumptionA}. 
From this and using the symmetry of $\mathrm{h}(-\mu,\mathrm{z})=\mathrm{h}(\mu,\mathrm{z})$ as in~\eqref{eq:2019_10_18_02}, we see that the condition (A) of Assumption \ref{def:asympstab} is satisfied with 
$E_{\stof} (\mathrm{z})$ defined in \eqref{eq:sf2e}. 
The following lemma is a generalization of Lemma~\ref{lem:asymofprod} (iv) from $\mu=0$ to $\mu>0$ in which $\rootsL_z$ is changed to $\SU_z$.

\begin{lm}
	\label{lm:energy_stof_asymptotics2}
	Under the conditions of Theorem~\ref{thm:special_IC} (ii), we have, as $L, N\to \infty$, 
	\begin{equation} \label{eq:enegasprod}
	\begin{split}
	&\frac{\prod_{v\in\rootsR_z} \prod_{u\in\SU_z} \sqrt{v-u}}{\prod_{v\in\rootsR_z}(\sqrt{v+1})^{(d-1)N}\prod_{u\in\SU_z}(\sqrt{-u})^{N}} \\
	&= \exp\left(\frac{\mathrm{z}^2}{2}\int_{\ii\realR}\int_{\ii\realR} 
	\frac{ \eta \eta' \log (\eta +\eta' +2\mu) }{(e^{-\eta^2/2}-\mathrm{z}) (e^{-(\eta')^2/2}-\mathrm{z})}\ddbar{\eta'}{} \ddbar{\eta}{}\right)  (1+O(L^{\epsilon-1/2})).
	\end{split}
	\end{equation}
\end{lm}

\subsection{Proof of the condition (B)}

By the definition \eqref{eq:def_ich} of characteristic functions and using the formula \eqref{eq:aux_053}, 
\beq \label{eq:aux_2019_08_21_07}
\ich_{\mathrm{\stof}}(v,u;z)
= \frac{(v+1)^{(d-1)(N-1)} g(v,v)}{(u+1)^{(d-1)(N-1)} g(u,v)}
\prod_{v'\in \rootsR_z} \frac{g(u,v')}{g(v,v')}
\eeq
for $v\in \rootsR_z$ and $u\in \rootsL_z$.

Fix $0<\epsilon<1/8$. 
To check the condition (B) of Assumption \ref{def:asympstab}, it is enough to consider $v\in \rootsR_z^{(\epsilon)}$ and $u\in \rootsL_z^{(\epsilon)}$. 
(See \eqref{eq:rstpepe} for the notations.)  
For such $u$ and $v$, we use Lemma \ref{lm:roots_stof_estimate} to express the characteristic function in a different way. The result is given in Lemma \ref{lm:ich_stof} below whose proof uses the following result.

\begin{lm}
	\label{lm:product_stof}
	Recall the notations of Lemma~\ref{lm:roots_stof_estimate}. 
	For $w\in \Omegac$, 
	\begin{equation*}
	\frac{1}{(w+1)^{(d-1)N} }\prod_{v'\in\rootsR_z} g(w,v') 
	=  \frac{q_{z, \RR}(\hat w)}{\hat w^N} 
	\prod_{i=1}^{d-2} \frac{q_{z, \RR}( c_i(w))}{c_i( w)^N} ,
	\end{equation*}
	where we recall the right Bethe polynomial $q_{z,\RR}(w') =\prod_{v'\in\rootsR_z} (w' -v')$.
\end{lm}

\begin{proof}
	By the definition of $\hat w$ and $c_i(w)$ in Lemma~\ref{lm:roots_stof_estimate}, 
	\begin{equation}
	\label{eq:aux_2019_08_16_08}
	g(w,w') =(w'-\hat w) \prod_{i=1}^{d-2} (w'- c_i(w)).
	\end{equation}
	Taking a product,  
	\begin{equation}
	\label{eq:aux_2019_08_16_09}
	\begin{split}
	\prod_{v'\in\rootsR_z} g(w,v') 
	&= \prod_{v'\in\rootsR_z} \left[ (v'-\hat w) \prod_{i=1}^{d-2} (v'- c_i(w)) \right] = (-1)^{(d-1)N} q_{z,\RR}(\hat w)  \prod_{i=1}^{d-2} q_{z,\RR}(c_i(w)).
	\end{split}
	\end{equation}
	Setting $w'=0$ in \eqref{eq:aux_2019_08_16_08} implies that 
	\beqq
	(w+1)^{d-1}= (-1)^{d-1} \hat w  \prod_{i=1}^{d-2} c_i(w).
	\eeqq
	Dividing \eqref{eq:aux_2019_08_16_09} by the above equation raised by power $N$, we obtain the result. 
\end{proof}

Recall the notation~\eqref{eq:def_H}, 
\begin{equation}
\label{eq:aux_2019_08_20_05}
H_{z}(w)= \begin{cases}
\frac{q_{z,\RR}( w)}{ w^N}
\quad &\text{for $\Re(w)<-\rho$},\\
\frac{q_{z,\LL}(w)}{ (w+1)^{L-N}}
\quad &\text{for $\Re(w)>-\rho$.}\\
\end{cases}
\end{equation}
In the next lemma, we express the characteristic function for $v\in \rootsR_z^{(\epsilon)}$ and $u\in \rootsL_z^{(\epsilon)}$ using $H_z$. 
For such $v$ and $u$, we may set 
\begin{equation}
\label{eq:aux_2019_08_23_01}
u=-\rho +\xi\sqrt{\rho(1-\rho)}L^{-1/2}
\in \rootsL_z^{(\epsilon)}
\quad \text{and} \quad 
v=-\rho +\eta\sqrt{\rho(1-\rho)}L^{-1/2}
\in \rootsR_z^{(\epsilon)}
\end{equation}
with  $|\xi|,|\eta|<L^{\epsilon}$.
Then, in terms of the notations of Lemma \ref{lm:roots_stof_estimate}, 
\beqq
\hat u=-\rho+ (-2\mu -\xi) \sqrt{\rho(1-\rho)} L^{-1/2} + O(L^{2\epsilon-1})
\eeqq
and $c_i(u)=c_i+O(L^{2\epsilon-1})$ for $1\le i\le d-2$. 
Similarly,
\beqq 
\hat v=-\rho+ (-2\mu -\eta) \sqrt{\rho(1-\rho)} L^{-1/2} + O(L^{2\epsilon-1})
\eeqq
and $c_i(v)=c_i+O(L^{2\epsilon-1})$  for $1\le i\le d-2$.

\begin{lm} \label{lm:ich_stof}
	Fix $0<\epsilon<1/8$. 
	Let $v\in \rootsR_z^{(\epsilon)}$ and $u\in \rootsL_z^{(\epsilon)}$ given by \eqref{eq:aux_2019_08_23_01} above. 
	Then, 
	\beq
	\label{eq:aux_2019_08_16_11}
	\ich_{\mathrm{\stof}}(v,u;z)
	= \frac{(u+1)^{d-1} g(v,v) H_z(\hat u)}{(v+1)^{d-1} g(u,v) H_z(\hat v)} \prod_{i=1}^{d-2} \frac{ H_z(c_i(u))}{H_z (c_i(v))}
	\eeq
	if $\Re(\hat u) \le -\rho$,
	and
	\beq
	\label{eq:aux_2019_08_21_01}
	\ich_{\mathrm{\stof}}(v,u;z)  =
	\frac{(u+1)^{d-1} g(v,v) (1- z^L \hat u^{-N} (\hat u +1)^{-L+N}) }{(v+1)^{d-1} g(u,v) H_z(\hat u)  H_z(\hat v)}\prod_{i=1}^{d-2} \frac{H_{z}(c_i(u))}{H_z(c_i(v))}
	\eeq
	if $\Re(\hat u) \ge -\rho$.
	Here, when $\Re(\hat u)=-\rho$, the formula is defined by the limit of either of the formulas of \eqref{eq:aux_2019_08_20_05}. 
\end{lm}

\begin{proof}
	By \eqref{eq:aux_2019_08_21_07} and  Lemma~\ref{lm:product_stof}, 
	\beqq
	\ich_{\mathrm{\stof}}(v,u;z)
	= \frac{(u+1)^{d-1} g(v,v) \hat v^N q_{z, \RR}(\hat u)}{(v+1)^{d-1} g(u,v) \hat u^N q_{z, \RR}(\hat v)} \prod_{i=1}^{d-2} \frac{ c_i(v)^N q_{z, \RR}(c_i(u))}{c_i(u)^N q_{z, \RR} (c_i(v))}.
	\eeqq
	By Lemma \ref{lm:roots_stof_estimate}, $\Re(c_i(u))<-\rho $, and hence we may replace $q_{z, \RR}(c_i( u))$ by $c_i(u)^N H_z(c_i(u))$. 
	The same applies to $q_{z, \RR}(c_i(v))$. 
	Consider $q_{z, \RR}(\hat v)$ and $q_{z, \RR}(\hat u)$.
	Since $v\in \rootsR_z$, we have $\Re(\eta)>0$ in \eqref{eq:aux_2019_08_23_01}. 
	Thus, we find that $\Re(\hat v)<-\rho$ and we may replace $q_{z, \RR}(\hat v)$ by $\hat v^N H_z(\hat v)$. 
	On the other hand, for $u\in \rootsL_z$, and hence $\Re(\xi)<0$, the real part of $-2\mu-\xi$ can be either positive or negative. 
	When it is negative and thus $\Re(\hat u)<-\rho$, we replace $q_{z, \RR}(\hat u)$ by $\hat u^N H_z(\hat u)$, and we obtain \eqref{eq:aux_2019_08_16_11}.
	However, if $\Re(\hat u)>-\rho$, then we first use the identity $q_{z,\LL}(w)q_{z, \RR}(w) = w^N(w+1)^{L-N}- z^L$ 
	to write $q_{z, \RR}(\hat u)$ in terms of $q_{z, \LL}(\hat u)$, and use the definition of $H_z(w)$ for $\Re(w)>-\rho$ to obtain \eqref{eq:aux_2019_08_21_01}. 
	
	When $\Re(\hat u)=-\rho$, the limits of two formulas \eqref{eq:aux_2019_08_16_11} and \eqref{eq:aux_2019_08_21_01} are equal using $q_{z,\LL}(w)q_{z, \RR}(w) = w^N(w+1)^{L-N}- z^L$ with $w=\hat u$.  
\end{proof}

\bigskip

We now prove the condition (B) of Assumption \ref{def:asympstab}  using Lemma \ref{lm:ich_stof}. 
The proof follows from the next asymptotic results most of which are in Section \ref{sec:asymptotis_products}.

\begin{enumerate}[(a)]
	\item Since $c_i$ are $O(1)$ away from the point $w=-\rho$ and satisfy $\Re(c_i)<-\rho$, the second equation of Lemma~\ref{lem:asymofprod} (i) implies that 
	\beqq   
	H_z(c_i(u)) = 1 + O(L^{\epsilon-1/2}) \quad \text{and} \quad  H_z(c_i(v)) = 1 + O(L^{\epsilon-1/2}).
	\eeqq 
	\item The second equation of Lemma~\ref{lem:asymofprod} (ii) implies that 
	\begin{equation*}
	H_z(\hat v) = e^{\mathrm{h} (-\eta-2\mu,\mathrm{z})} \big(1 + O(L^{\epsilon-1/2})\big). 
	\end{equation*}
	\item The second equation of Lemma~\ref{lem:asymofprod} (ii) implies that 
	\beqq 
	H_z(\hat u) =
	e^{\mathrm{h}(-\xi-2\mu,\mathrm{z})}\left(1+O(L^{\epsilon-1/2})\right)  \qquad \text{if $\Re(\xi+2\mu)\ge 0$},
	\eeqq
	where when $\Re(\xi+2\mu)=0$, we interpret $\mathrm{h}(-\xi-2\mu,\mathrm{z})$ as the limit of $\mathrm{h}(-\xi-2\mu- \delta,\mathrm{z})$ as $\delta \downarrow 0$. 
	The first equation of Lemma~\ref{lem:asymofprod} (ii) and the formula~\eqref{eq:2019_10_18_02} imply that
	\beqq
	H_z(\hat u) = 
	e^{\mathrm{h}(-\xi-2\mu,\mathrm{z})}\big(1+O(L^{\epsilon-1/2})\big)
	=e^{\mathrm{h}(\xi+2\mu,\mathrm{z})}\big(1+O(L^{\epsilon-1/2})\big)
	\eeqq
	if $\Re(\xi+2\mu)<0$.
	\item  Note that for $w=-\rho+\zeta \sqrt{\rho(1-\rho)}L^{-1/2}$ with $|\zeta|\le L^{\epsilon}$, we have $ w(w+1)^{d-1} = (-\rho)(1-\rho)^{d-1}\big(1-\frac{\zeta^2+2\mu\zeta}{2\rho} L^{-1}+O(L^{3\epsilon-3/2})\big)$. 
	Hence, it follows from the definition of $g$ that 
	\beqq
	\frac{(u+1)^{d-1} g(v,v)}{(v+1)^{d-1} g(u,v)}=  \frac{2(\mu+\eta)}{\xi+\eta+2\mu}\left(1+O(L^{\epsilon-1/2})\right).
	\eeqq
	We will show in Lemma \ref{lm:lower_bound_den} that $\xi+\eta+2\mu$ is non-zero. 
	\item It is straightforward to check using the definition that 
	\beqq   
	1- z^L \hat u^{-N} (\hat u +1)^{-L+N} = 1 - \mathrm{z} e^{\frac12 (\xi+2\mu)^2 }  \big(1+O(L^{3\epsilon-1/2})\big) . 
	\eeqq
\end{enumerate}

{All of the above limits are uniform for $\mathrm{z}$ in a compact subset of the unit disk.} 
From the above results, we find that the condition (B) of Assumption \ref{def:asympstab} holds {with arbitrary two real numbers $0<r_1<r_2<1$ and} with 
\begin{equation*}
\chi_\stof(\eta,\xi;\mathrm{z}) 
= \frac{2(\eta +\mu)}{\xi+\eta +2\mu}
e^{\mathrm{h}(-\xi-2\mu,\mathrm{z})-\mathrm{h}(-\eta-2\mu,\mathrm{z})}
\end{equation*}
if $\Re(\xi+2\mu)\ge 0$, and 
\begin{equation*}
\chi_\stof(\eta;\xi;\mathrm{z}) = \frac{2(\eta +\mu)}{\xi+\eta +2\mu}(1-\mathrm{z}e^{(\xi+2\mu)^2/2})
e^{-\mathrm{h}(-\xi-2\mu,\mathrm{z})-\mathrm{h}(-\eta-2\mu,\mathrm{z})}
\end{equation*}
if $\Re(\xi+2\mu)<0$.
Here, we note that $\Re(-\eta-2\mu)<0$. 
When $\Re(\xi+2\mu)=0$, the formula $\mathrm{h}(-\xi-2\mu,\mathrm{z})$ should be interpreted as the limit of $\mathrm{h}(-\xi-2\mu-\delta,\mathrm{z})$
as the limit of the positive $\delta\downarrow 0$.
As we discussed in Lemma~\ref{lem:hpropperty} (c), the formula of $\chi_\stof(\eta,\xi;\mathrm{z})$ when $\Re(\xi+2\mu)<0$ is an analytic continuation of the formula when $\Re(\xi+2\mu)>0$.

\bigskip

It remains to show that $\xi+\eta+2\mu$ in (d) above is not zero. 

\begin{lm} \label{lm:lower_bound_den}
	For $u$ and $v$ in \eqref{eq:aux_2019_08_23_01}, we have (recall that $\mu>0$)
	\begin{equation*}
	|\xi + \eta +2\mu| \ge \mu
	\end{equation*}
	for all sufficiently large $L$. 
\end{lm}

\begin{proof}
	By Lemma~\ref{lm:limiting_nodes},   the discrete sets $\rootsL_z$ and $\rootsR_z$ in a $L^{-1/2+\epsilon}$-neighborhood of the point $w=-\rho$ converge to the discrete sets $\inodesL_{\mathrm{z}}$ and $\inodesR_{\mathrm{z}}$, respectively, as $L\to \infty$. 
	Hence, the lemma follows if we show that $|\xi+\eta+2\mu|\ge \sqrt{2}\mu$ for all $\xi\in\inodesL_{\mathrm{z}}$ and $\eta\in\inodesR_{\mathrm{z}}$. 
	The points $\xi\in\inodesL_{\mathrm{z}}$ and $\eta\in\inodesR_{\mathrm{z}}$ satisfy the equation 
	\beqq
	e^{-\xi^2/2} =e^{-\eta^2/2} = \mathrm{z} \quad \text{and} \quad \Re(\xi)<0<\Re(\eta).
	\eeqq
	All such points are of the form 
	\beqq
	\xi = \sqrt{-2\log |\mathrm{z}|} \left(-\sec\alpha + \ii\tan\alpha\right) \quad \text{and} \quad \eta=\sqrt{-2\log |\mathrm{z}|} \left(\sec\beta + \ii\tan\beta\right)
	\eeqq
	for some $-\pi/2< \alpha,\beta <\pi/2$. 
	Since $\cos(\alpha+\beta)+\cos(\alpha-\beta)=2 \cos\alpha\cos\beta>0$, we have 
	\begin{equation}
	\label{eq:aux_2019_08_23_02}
	\cos^2\big( \frac{\alpha+\beta}{2} \big) = \frac{1+\cos(\alpha+\beta)}{2} > \frac{1-\cos(\alpha-\beta)}{2} =\sin^2\big(\frac{\alpha-\beta}{2}\big).
	\end{equation}
	A direct calculation shows that
	\beqq
	|\Im(\xi+\eta)| =2\sqrt{-2\log|\mathrm{z}|}\left|\sec\alpha\sec\beta\sin\big(\frac{\alpha+\beta}{2}\big)\cos\big(\frac{\alpha+\beta}{2}\big)\right|
	\eeqq
	and
	\beqq
	|\Re(\xi+\eta)| =2\sqrt{-2\log|\mathrm{z}|}\left|\sec\alpha\sec\beta\sin\big(\frac{\alpha+\beta}{2}\big)\sin\big(\frac{\alpha-\beta}{2}\big)\right|.
	\eeqq
	Thus, by ~\eqref{eq:aux_2019_08_23_02}, we have $|\Im(\xi+\eta)|\ge |\Re(\xi+\eta)|$. 
	It is easy to see from the geometry that $|x+y\ii + 2a|\ge \sqrt{2}a$ if $x,y,a\in\realR$ satisfy $|y|\ge |x|$ and $a>0$. Hence, 
	we find that $|\xi+\eta+2\mu|\ge \sqrt{2}\mu>\mu$ and this completes the proof.
\end{proof}

\subsection{Proof of the condition (C)}

By \eqref{eq:easier_tail_estimates}, the condition (C) of Assumption \ref{def:asympstab} is proved if we show that there 
is a constant $C>0$ such that
$|\ich_{\mathrm{\stof}}(v,u;z)|\le C L$  for all $(u,v)\in \rootsR_z\times \rootsL_z$ and $r_1<|\mathrm{z}|<r_2$ 
where $r_1$ and $r_2$ are from the condition (B). 
Since we showed that the condition (B) holds for arbitrary $0<r_1<r_2<1$, it is enough to fix them arbitrarily in this subsection. 

The characteristic function is given by the formula \eqref{eq:aux_2019_08_21_07}.
To prove the upper bound of the characteristic function, we need the following two lemmas which will be proved in Section~\ref{sec:proof_assumptionC}.

\begin{lm}
	\label{lm:assumptionC_03}
	There is a positive constant $C$ such that 
	\begin{equation*}
	\left| \frac{\prod_{v'\in \rootsR_z} g(v,v')}{(v+1)^{(d-1)N}} \right|
	\ge C 
	\end{equation*}
	for all $v\in\rootsR_z$ and $r_1<|\mathrm{z}|<r_2$.
\end{lm}

\begin{lm}
	\label{lm:assumptionC_02}
	There is a positive constant $C$ such that 
	\begin{equation*}
	\left| \frac{\prod_{v'\in \rootsR_z} g(u,v')}{(u+1)^{(d-1)N}} \right| 
	\le C 
	\end{equation*}
	for all $u\in\rootsL_z$ and $r_1<|\mathrm{z}|<r_2$.
\end{lm}

Using these lemmas and the equation~\eqref{eq:aux_2019_08_21_07}, the condition (C) is proved if we show that the absolute value of 
\beqq
\frac{(v+1)^{-(d-1)} g(v,v) }{(u+1)^{-(d-1)} g(u,v)} = \frac{(u+1)^{d-1}(dv+1)(u-v)}{(v+1) \big( u(u+1)^{d-1}-v(v+1)^{d-1} \big)}
\eeqq
is bounded by $C'L$ for a constant $C'>0$ uniformly for $u,v$ and $\mathrm{z}$. 
Since the sets $\rootsL_z$ and $\rootsR_z$ remain in a bounded region for bounded $\mathrm{z}$, the following lemma shows the desired bound. 
This completes the proof of the condition (C).

\begin{lm}
	\label{lm:assumptionC_01}
	There is a positive constant $C$ such that 
	\begin{equation*}
	|u(u+1)^{d-1} -v(v+1)^{d-1}| \ge CL^{-1}
	\end{equation*}
	for all $u\in\rootsL_z$ and $v\in\rootsR_z$ and for $r_1<|\mathrm{z}|<r_2$.
\end{lm}

\begin{proof}
	Since $u$ and $v$ are Bethe roots, they satisfy $u^N(u+1)^{L-N}= v^N (v+1)^{L-N}$.
	This implies that $|u(u+1)^{\rho^{-1}-1}|=|v(v+1)^{\rho^{-1}-1}|$ since $\rho=N/L$. 
	Thus, we find, using \eqref{eq:dintermsofrho}, that 
	\beqq
	\left|\frac{u(u+1)^{d-1}}{v(v+1)^{d-1}}\right|= \left|\frac{v+1}{u+1}\right|^{\rho^{-1}-d}\ge \left|\frac{v+1}{u+1}\right|^{cL^{-1/2}} 
	\eeqq
	for a positive constant $c$.
	The set $\rootsL_z$ is a subset of the trajectory $\Lambda_\LL$ given by $|w^N(w+1)^{L-N}|=|z|^L$ satisfying $\Re(w)<-\rho$, and $\rootsR_z$ is a subset of the trajectory $\Lambda_\RR$ given by the same equation but satisfying $\Re(w)>-\rho$. 
	Consider the point $u_0$ of $\Lambda_\LL$  which is farthest from the point $w=-1$. This point is the point which is closest to the point $w=-\rho$; this can be seen by noting that on the circle $|w+1|=r$, the value of $|w^N(w+1)^{L-N}|$ increases as $w$ moves away from the point $w=-1+r$. 
	Similarly, the point $v_0$ of $\Lambda_\RR$ which is nearest to the point $w=-1$ is the point which is closest to the point $w=-\rho$. 
	Since $z^L= (-1)^N\rr^L \mathrm{z}$, the distance between the points $u_0$ and $v_0$ is of order $L^{-1/2}$. Hence, 
	\beqq
	\left|\frac{u(u+1)^{d-1}}{v(v+1)^{d-1}}\right|\ge \left|\frac{v_0+1}{u_0+1}\right|^{cL^{-1/2}} 
	\ge (1+c'L^{-1/2})^{c L^{-1/2}}\ge 1+ c''L^{-1}
	\eeqq
	for some positive constants $c', c''$. 
	Thus, 
	\begin{equation*}
	|u(u+1)^{d-1} -v(v+1)^{d-1}| 
	\ge c''L^{-1}|v(v+1)^{d-1}|.
	\end{equation*}
	Since $|\mathrm{z}|$ stays in a compact subset of the interval $(0,1)$, $|v|$ and $|v+1|$ are greater than a postive constant. This completes the proof.
\end{proof}

\subsection{Proof of lemmas}
\label{sec:proof_lemmas_stepflat}

In this section, 
we prove five lemmas we used in the previous sections. 
We first prove Lemma~\ref{lm:roots_stof_estimate} in Subsection~\ref{sec:proof_stof_roots} whose proof is independent of other lemmas. 
The proof of other lemmas involves certain properties of the map $w\mapsto w(w+1)^{d-1}$. 
We discuss such properties in Subsection~\ref{sec:aux_lemmas}, and then prove Lemma~\ref{lm:roots_stof}, Lemma~\ref{lm:energy_stof_asymptotics2}, Lemma~\ref{lm:assumptionC_03}, and Lemma~\ref{lm:assumptionC_02} in the next three subsections.

\subsubsection{Proof of Lemma~\ref{lm:roots_stof_estimate}}
\label{sec:proof_stof_roots}

Consider the Taylor expansions of $g(w, w')$. 
Note that 
\beqq
g(w, w') = \frac{p(w)-p(w')}{w-w'}, \quad \text{where} \quad p(w)= w(w+1)^{d-1}.
\eeqq
Some of the derivatives are 
\beqq
\begin{split}
	g_w &= \frac{p'(w)(w-w') - (p(w)-p(w'))}{(w-w')^2}, \\
	g_{w'} &= \frac{-p'(w')(w-w') + (p(w)-p(w'))}{(w-w')^2}, 
\end{split}        
\eeqq
and
\beqq
g_{ww}= \frac{p''(w)(w-w')^2-2p'(w)(w-w')+2(p(w)-p(w'))}{(w-w')^3}.
\eeqq
Recall that one of the solutions of the equation $g(-d^{-1}, w')=0$ is $w'=-d^{-1}$ and the other $d-2$ solutions are denoted by $w'=c_i$ for $1\le i\le d-2$. 
Since the only critical points of the function $p(w)$ are $w=-d^{-1}$ and $w=-1$, we have 
\beqq
p'(c_i)\neq 0 \quad \text{and} \quad p'(-d^{-1})=0. 
\eeqq
Furthermore, 
\beqq
p''(-d^{-1})=(d-1)^{d-2}d^{3-d} \neq 0. 
\eeqq
From these formulas and the fact that $p(-d^{-1})= p(c_i)$, we find that 
\beqq
g_w(-d^{-1}, c_i)=0, \qquad   g_{w'}(-d^{-1}, c_i)\neq 0, \qquad 
g_{ww}(-d^{-1}, c_i)\neq 0,
\eeqq
and  
\beqq
g_w(-d^{-1}, -d^{-1}) = g_{w'} (-d^{-1}, -d^{-1}) = \frac12 p''(-d^{-1})\neq 0. 
\eeqq
Hence, the Taylor expansion of $g(w, w')$ near the point $(w, w')=(-d^{-1}, -d^{-1})$ is
\beq \label{eq:TE1}
g(w, w')= A (w+w'+2d^{-1}) + O(|w+d^{-1}|^2)+ O(|w'+d^{-1}|^2)
\eeq
for a non-zero constant $A$, 
and the Taylor expansion near the point $(w, w')=(-d^{-1}, c_i)$ implies that 
\beq \label{eq:TE2}
g(w, w') = B_i (w+d^{-1})^2 + D_i (w'-c_i) + O\big( |w+d^{-1}|(|w+d^{-1}|^2+|w'-c_i|) \big) 
\eeq
for non-zero constant $B_i$ and $D_i$. 

Now, Lemma~\ref{lm:roots_stof_estimate} assumes that $w=-\rho+\zeta\sqrt{\rho(1-\rho)}L^{-1/2}$ with $|\zeta|\le L^\epsilon$. 
By \eqref{eq:dintermsofrho}, such a $w$ can be written as 
\beqq \begin{split}
	w=-d^{-1} + (\zeta+\mu) \sqrt{\rho(1-\rho)} L^{-1/2} + O(L^{-1+\epsilon}). 
\end{split} \eeqq
Thus the solutions $w'$ of the equation $g(w, w')=0$ for given $w$ above can be found from the equations \eqref{eq:TE1} and \eqref{eq:TE2}. Solving them, we obtain Lemma~\ref{lm:roots_stof_estimate}.

\subsubsection{Auxiliary lemmas} 
\label{sec:aux_lemmas}

We discuss two auxiliary lemmas which are used in the proof of the remaining four lemmas.

The first lemma is about the map $w\mapsto w(w+1)^{d-1}$.

\begin{lm}
	\label{lm:Omega_plus}
	For every $d\ge 2$, 
	there exists a simply connected domain $\Omega_+$ in $\complexC$ satisfying the following properties, where $\Gamma^\perp$ denotes the boundary of $\Omega_+$, 
	\beqq
	\Gamma^\perp =\partial \Omega_+ . 
	\eeqq
	\begin{enumerate}[(i)]
		\item The region $\Omega_+$ contains the part $\{x\in \realR : x>-d^{-1}\}$ of the real line. 
		\item We have $\Gamma^\perp= \Gamma^\perp_{\upp}\cup \Gamma^\perp_{\down}$ where the curve $\Gamma^\perp_\upp$ extends from the point $-d^{-1}$ to infinity such that it lies in the upper half plane except for the endpoint $-d^{-1}$. 
		The curve $\Gamma^\perp_{\down}=\{\bar w: w\in \Gamma_{\upp}^\perp\}$. 
		\item The map $w\mapsto w(w+1)^{d-1}$ is a bijection from $\Omega_+$ to  $\complexC\setminus (-\infty, -(\sfrr)^d]$ where
		\begin{equation*}
		\sfrr:= d^{-1/d}(1-d^{-1})^{1-1/d}.
		\end{equation*}
		\item The map $w\mapsto w(w+1)^{d-1}$ is a bijection from $\Gamma^\perp_{\upp}$ to $(-\infty, -(\sfrr)^d]$.
		\item If $w\in\Omega_+\cup \Gamma^\perp$, then $w+c\in\Omega_+$ for every constant $c>0$.
	\end{enumerate}
\end{lm}

\begin{figure}
	\centering
	\includegraphics[scale=0.4]{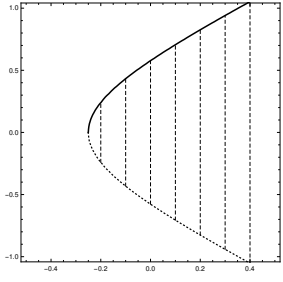}
	\caption{The solid curve is $\Gamma^\perp_{\upp}$ and the dotted curve is ${\Gamma^\perp_{\down}}$. The region $\Omega_+$ is indicated by the dashed lines. Here, we assumed $d=4$.}
	\label{fig:Omega_plus}
\end{figure}

See Figure \ref{fig:Omega_plus}.  
To prove the above lemma, we define a few contours and regions.
Define the regions  
\begin{equation*} 
\begin{split}
\Omega_{\LL} &= \{w\in \complexC: |w(w+1)^{d-1}| < (\sfrr)^d ,  \quad  \Re(w) < -d^{-1}\}, \\
\Omega_{\RR} &= \{w\in \complexC: |w(w+1)^{d-1}| < (\sfrr)^d , \quad  \Re(w) > -d^{-1}\},
\end{split} \end{equation*}
and set 
\begin{equation}
\label{eq:def_Gamma}\begin{split}
\Gamma_{\LL} = \partial \Omega_{\LL}  \quad \text{and} \quad \Gamma_{\RR} = \partial \Omega_{\RR}. 
\end{split} \end{equation}
The contours $\Gamma_\LL$ and $\Gamma_\RR$  are simple and closed. The contour $\Gamma_\LL$ contains the point $-1$ inside and the contour $\Gamma_\RR$ contains the point $0$ inside. 
The two contours intersect at the point $-d^{-1}$.
See Figure~\ref{fig:contour2}.

We also define the set 
\begin{equation} \label{eq:def_Gamma_perp}
S:= \{w\in \complexC: w(w+1)^{d-1} \text{ is real-valued and } w(w+1)^{d-1}\le -(\sfrr)^d  \}.
\end{equation}
We discuss the shape of this set. 
The equation $w(w+1)^{d-1}= -(\sfrr)^d$ has $d$ solutions and they are on the contour $\Gamma_\LL\cup \Gamma_\RR$. 
It is easy to see that there is a double root at $w=-d^{-1}$. 
It is also easy to check (by the structure of the Bethe roots mentioned before the equation~\eqref{eq:def_rootsRL}) that 
for every real number $0\le a< (\sfrr)^d$, the equation $w(w+1)^{d-1}=-a$ has $1$ solution in the region $\Omega_\RR$ and $d-1$ solution in the region $\Omega_\LL$.
Furthermore, the roots are continuous functions of $a$, and the solution in the region $\Omega_\RR$ converges to the point $w=-d^{-1}$ as $a\to (\sfrr)^d$.
We thus find that the solutions of the equation $w(w+1)^{d-1}= -(\sfrr)^d$ consists of a double root at $w=-d^{-1}$ and $d-2$ points on  $\Gamma_\LL\setminus\{-d^{-1}\}$. 
For a real number $a>(\sfrr)^d$, there are $d$ distinct solutions of the equation $w(w+1)^{d-1}= -a$ and these solutions lie in the region $\complexC\setminus (\Omega_\LL\cup \Omega_\RR)$. 
In conclusion, the set $S$ consists of $(d-1)$ simple contours, one of which intersects the point $w=-d^{-1}$, and the rest contours intersect points on $\Gamma_\LL$. 
The contour that passes the point $w=-d^{-1}$ is symmetric about the real axis and it extends to the infinity at the angle $e^{\frac{ \pi \ii}{d}}$ in one direction and at the angle $e^{-\frac{ \pi \ii}{d}}$ in another direction. 
We denote this contour by $\Gamma^\perp$. 
The other contours extend to infinity at the angle $\frac{(2k+1) \pi \ii}{d}$, $1\le k\le d-2$.
They lie on the left of $\Gamma^\perp$. 
We set the upper part of $\Gamma^\perp$ to be $\Gamma^\perp_\upp$, and the lower half to be $\Gamma^\perp_\down$. 
We let $\Omega_+$ be the region on the right hand side of $\Gamma^\perp$. 
See Figure~\ref{fig:contour2}.
We remark that from the definitions, 
\beq \label{eq:OMRPH}
\Omega_\RR\subset \Omega_+ \subset \{ w\in \complexC : \Re(w)>-d^{-1}\}. 
\eeq

\begin{figure}\centering
	\includegraphics[scale=0.4]{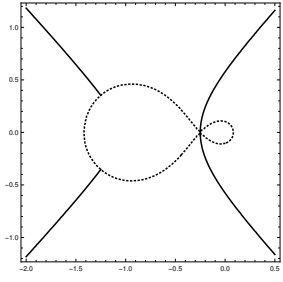}
	\caption{This picture is for $d=4$. The dashed contour has a self-intersection which is $-d^{-1}$. The part of the contour on the left of the self-intersection is $\Gamma_\LL$. The part on the right of the self-intersection is $\Gamma_\RR$. The interiors of these two contours are $\Omega_\LL$ and $\Omega_\RR$ respectively. 
		The union of the solid contours is the set~\eqref{eq:def_Gamma_perp}.
		The solid contour which passes the self-intersection point of the dashed contour is $\Gamma^\perp$. 
		It divides the complex plane to two regions. The region on the right side is $\Omega_+$. 
		Note that $\Omega_+$ contains $\Omega_\RR$ inside.} \label{fig:contour2}
\end{figure}

\begin{proof}[Proof of Lemma~\ref{lm:Omega_plus}]
	Above definitions imply (i), (ii), and (iv). To verify (iii), we note that the map $w\to w(w+1)^{d-1}$ from $\Omega_+$ to  $\complexC\setminus (-\infty, -(\sfrr)^d]$ is analytic and has no critical point. Moreover, it maps the boundary of $\Omega_+$, which is $\Gamma^\perp$, to the boundary of $\complexC\setminus (-\infty, -(\sfrr)^d]$, which is $(-\infty,-(\sfrr)^d]$. Thus it is a bijection. 
	We now check (v). 
	It is sufficiently to show that $\Gamma^\perp=\Gamma_\upp^\perp\cup\Gamma_\down^\perp$ intersects any horizontal line at exactly one point. Using the fact that $w(w+1)^{d-1}$ is strictly increasing on $(-d^{-1},\infty)$, we know $\Gamma^\perp$ intersects the real axis at exactly one point $w=-d^{-1}$. 
	Thus we obtain the result if we show that the derivative of $w(w+1)^{d-1}$ is not a real number for all $w$ in $\Gamma^\perp$ except $w=-d^{-1}$. 
	The derivative of $w(w+1)^{d-1}$ at $w=x+\ii y$ is 
	\beqq
	\begin{split}
		&\left(\frac{1}{w}+\frac{d-1}{w+1}\right) w(w+1)^{d-1}\\
		&= \left(\frac{x}{x^2+y^2}+\frac{(d-1)(x+1)}{(x+1)^2+y^2}\right) w(w+1)^{d-1}
		+\ii y \left(\frac{1}{x^2+y^2}+\frac{d-1}{(x+1)^2+y^2}\right) w(w+1)^{d-1}.
	\end{split}
	\eeqq
	For $w\in\Gamma^\perp$, $w(w+1)^{d-1}$ is a nonzero real number. 
	Hence, the imaginary part of the above expression is non-zero when $y\neq 0$. This completes the proof. 
\end{proof}

\bigskip

The previous lemma is a property depending only on the integer $d\ge 2$. 
We now consider the sets $\rootsL_z$ and $\rootsR_z$ which depend on the parameters $L$ and $N$, and how these sets are related to the contours and regions considered in the previous lemma. 
We assume the assumptions for Theorem~\ref{thm:special_IC} (ii): Recall that 
\begin{equation*} 
L=dN+L_s, \quad{where} \quad L_s = \mu \sqrt{d-1} L^{1/2} + O(1)
\end{equation*}
for fixed $\mu> 0$ and $z^L=(-1)^N\rr^L \mathrm{z}$ with $|\mathrm{z}|$ staying in a compact subset of $(0,1)$.
Define the contours 
\begin{equation*}
\begin{split}
\Lambda_\LL = \{w\in \complexC : |w^N(w+1)^{L-N}| = \rr^L |\mathrm{z}| , \quad  \Re(w) < -\rho\},\\
\Lambda_\RR = \{w\in \complexC : |w^N(w+1)^{L-N}| = \rr^L |\mathrm{z}| , \quad  \Re(w) > -\rho\},
\end{split}
\end{equation*}
which were introduced in \eqref{eq:lammbdaad1}. 
The set $\rootsL_z$ of the left Bethe roots is a discrete subset of $\Lambda_\LL$, and the set  $\rootsR_z$ of the right Bethe roots is a discrete subset of $\Lambda_\RR$. 
The contours $\Lambda_\LL$ and $\Lambda_\RR$ are disjoint and are separated by a distance greater than $cL^{-1/2}$ for some positive constant $c$. 
The following lemma states relationships between $\Lambda_\LL, \Lambda_\RR$ and $\Gamma_\LL, \Gamma_\RR$ defined in \eqref{eq:def_Gamma}. 
See Figure~\ref{fig:contour1}.

\begin{figure} \centering
	\includegraphics[scale=0.6]{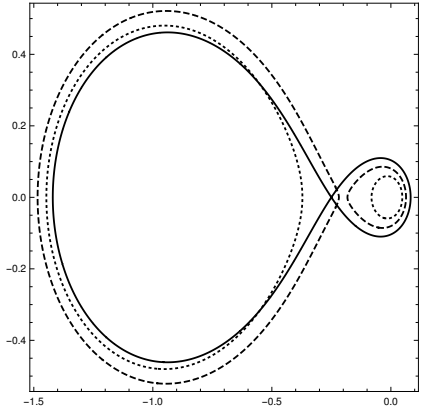}
	\caption{The solid contours are $\Gamma_\LL$ and $\Gamma_\RR$ when $d=4$. 
		The regions $\Omega_\LL$ and $\Omega_\RR$ are the interiors of these contours. 
		Note that the contours $\Gamma_\LL$ and $\Gamma_\RR$ intersect at the point $-d^{-1}$. 
		The dashed contours are $\Lambda_\LL$ and $\Lambda_\RR$ when $e^{-\mu^2/2}<|\mathrm{z}|<1$; in this case, the point $-d^{-1}$ is inside $\Lambda_\LL$.
		The dotted contours are $\Lambda_\LL$ and $\Lambda_\RR$ when $|\mathrm{z}|<e^{-\mu^2/2}$; in this case, the point $-d^{-1}$ is outside $\Lambda_\LL$. 
	}
	\label{fig:contour1}
\end{figure}

\begin{lm}
	\label{lm:locations_Lambda}
	We have the following properties. 
	\begin{enumerate}[(i)]
		\item The contour $\Lambda_\RR$ is in the interior of $\Gamma_\RR$. As a consequence, $\Lambda_\RR\subset \Omega_\RR\subset \Omega_+$. 
		\item The property (i) also hold even if $|\mathrm{z}|=1$. 
		\item If $e^{-\mu^2/2}<|\mathrm{z}|<1$, then the point $-d^{-1}$ is inside $\Lambda_\LL$ for all large enough $L$. In this case, the contour $\Lambda_\LL$ intersects $\Gamma_\RR$ at two points {and does not intersects $\Gamma_\LL$.}
		
		\item If $|\mathrm{z}|<e^{-\mu^2/2}$, then the point $-d^{-1}$ is outside $\Lambda_\LL$ for all large enough $L$. In this case, the contour $\Lambda_\LL$ {intersect $\Gamma_\LL$ at at most two points, but it does not intersect $\Gamma_\RR$.} 
	\end{enumerate}
\end{lm}

\begin{proof}
	We first prove (i). 
	By \eqref{eq:dintermsofrho}, we have $-d^{-1}<-\rho$ for all large enough $L$. 
	Hence, the contour $\Gamma_\RR$ encloses $-\rho$. It also encloses the point $0$. 
	On the other hand, $\Lambda_\RR$ encloses the point $0$ but not $-\rho$.  
	Therefore, the property (i) is obtained if we show that $\Lambda_\RR \cap \Gamma_\RR=\emptyset$. 
	We prove it by contradiction. 
	Suppose that there exists a point $w\in \Lambda_\RR \cap \Gamma_\RR$. By the definitions of $\Lambda_\RR$ and $\Gamma_\RR$, $w$ satisfies 
	\beqq
	|w^N(w+1)^{L-N}|= \rr^L |\mathrm{z}|\quad \text{and} \quad |w(w+1)^{d-1}|=(\sfrr)^d=d^{-1}(1-d^{-1})^{d-1}.
	\eeqq
	Recall that $\rr^L= \rho^N (1-\rho)^{L-N}$. Since $L=dN+L_s$, we have $L-N= (d-1)N+L_s$. Hence, we find that 
	\beqq
	\begin{split}
		|w+1|^{L_s}& = \frac{|w^N(w+1)^{L-N}|}{|w^N(w+1)^{(d-1)N}|}\\
		& =\frac{\rho^N (1-\rho)^{(d-1)N}}{(d^{-1})^N(1-d^{-1})^{(d-1)N}}  (1-\rho)^{L_s}|\mathrm{z}| <(1-\rho)^{L_s} |\mathrm{z}|
	\end{split}
	\eeqq
	since the function $x(1-x)^{d-1}$ is monotone in $(0,d^{-1})$ and $\rho<d^{-1}$. Since $|\mathrm{z}|<1$, the above inequality implies that $|w+1|<1-\rho$. Hence,  $\Re(w)<-\rho$ and this contradicts the definition of $\Lambda_\RR$. Therefore, we obtain the property (i). 
	
	For the part (ii), we note that in the above argument, $|w+1|^{L_s}< (1-\rho)^{L_s} |\mathrm{z}|$ still implies that $|w+1|<1-\rho$ even if $|\mathrm{z}|=1$. Hence, we obtain the same property even when $|\mathrm{z}|=1$.

	The properties (iii) and (iv) are not used in the paper. We only provide a sketch of the proof. 
	We first show that the system of equations
	\begin{equation}
	\label{eq:2019_10_20_01}
	|w^N(w+1)^{L-N}|=\rr^L|\mathrm{z}|\quad \text{and} \quad |w(w+1)^{d-1}|=(\sfrr)^d
	\end{equation}
	have at most two solutions. 
	Indeed, regarding \eqref{eq:2019_10_20_01} as a systems of two equations of two variables 
	$x=|w|$ and $y=|w+1|$, one can check that there is a unique solution since $L>dN$.
	Since there are at most two values of $w$ with given $|w|$ and $|w+1|$ values, we find that the equations~\eqref{eq:2019_10_20_01} have at most two solutions. 
	Therefore, the contour $\Lambda_\LL$ intersects the contour $\Gamma_\LL\cup\Gamma_\RR$ at at most two points. 
	The properties (iii) and (iv) follow by analyzing whether the point $-d^{-1}$, which is the intersection point of $\Gamma_\LL$ and $\Gamma_\RR$, lies outside or inside $\Lambda_\LL$.
	This computation is tedious and we skip the details. 
\end{proof}

\subsubsection{Proof of Lemma~\ref{lm:roots_stof}}
\label{sec:proof_U_v}

Let $v\in \rootsR_z$. 
By Lemma \ref{lm:locations_Lambda} (i), $\rootsR_z$, which is a discrete subset of $\Lambda_\RR$, is a subset of $\Omega_\RR$. 
Hence, by the definition of $\Omega_\RR$, we find that $|v(v+1)^{d-1}|< \sfrr$. 
Now, by the structure of the Bethe roots mentioned before the equation~\eqref{eq:def_rootsRL}, if $|\mathrm{z}|< \sfrr$, then the equation $w(w+1)^{d-1}=\mathrm{z}$ has $d$ distinct roots, one of which is in the region $\Omega_\RR$ and the rest $d-1$ are in the region $\Omega_\LL$. 
Hence, the equation $w(w+1)^{d-1}=v(v+1)^{d-1}$ of $w$ has $d-1$ distinct solutions in $\Omega_\LL$ and the remaining one solution is $w=v$. 
Therefore, the elements of $U(v)$ are the $d-1$ distinct solutions of $w(w+1)^{d-1}=v(v+1)^{d-1}$ in $\Omega_\LL$. 
This proves Lemma~\ref{lm:roots_stof}.

\subsubsection{Proof of Lemma~\ref{lm:energy_stof_asymptotics2}}
\label{sec:proof_assumptionA}

To prove Lemma~\ref{lm:energy_stof_asymptotics2}, we first use Lemma~\ref{lem:sumtointres} to express the left-hand side of \eqref{eq:enegasprod} as a double integral and then take the limit. 
In this computation, the contour for the integral needs to be outside $\Lambda_\RR$ but inside the region $\Omega_+$. 
Such a contour exists by Lemma \ref{lm:locations_Lambda} (i).
In this proof, we explicitly construct the contour, which we denote $C_\oout$ below, to make the asymptotic analysis concrete. 
The same contour $C_\oout$  will also be used in the next subsections.

\begin{figure}
	\centering
	\includegraphics[scale=0.4]{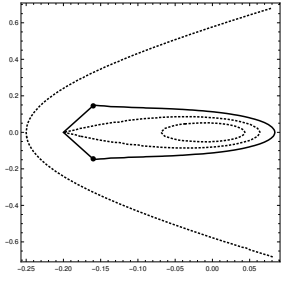}
	\caption{The dotted curves, from outside to inside, are $\Gamma^\perp=\partial\Omega_+$,
		$\{w\in\complexC: \Re(w)\ge -\rho, \, |w|^N|w+1|^{L-N}=\rr^L\}$, 
		and $\Lambda_\RR$. 
		The solid contour is $C_{\oout}$ and it lies between $\Gamma^\perp=\partial\Omega_+$ and $\Lambda_\RR$. The part $C_{\oout}^{(1)}$ is the union of two line segments and the rest is $C_{\oout}^{(2)}$. Two black points are $Q$ and $\overline{Q}$. The corner point of $C_{\oout}^{(1)}$ is $-\rho$.}
	\label{fig:contour_CK}
\end{figure}

Fix $0<\epsilon<1/8$. Consider the contour 
\beq \label{eq:Lambdapr} 
\Lambda':=\{w \in \complexC: |w^N(w+1)^{L-N}|=e^{L^\epsilon} \rr^L |\mathrm{z}|\}. 
\eeq
Since $e^{L^\epsilon} |z|>1$, the contour is a simple closed curve and it encloses the line segment $[-1, 0]$. We will choose two points $Q$ and its conjugation $\bar Q$ on $\Lambda'$. These two points will be given in explicit formulas below (see~\eqref{eq:2019_11_10_02}). Then we define the contour $C_\oout$ as follows. It consists of two parts $C_\oout^{(1)}$ and $C_\oout^{(2)}$: $C_\oout=C_\oout^{(1)}\cup C_{\oout}^{(2)}$,  where
\begin{equation}
\label{eq:def_Cout}
\begin{split}
C_{\oout}^{(1)}&=\{\lambda(-\rho)+(1-\lambda) Q: 0\le \lambda\le 1\}{\, \cup\,} \{\lambda(-\rho)+(1-\lambda) \bar Q: 0\le \lambda\le 1\}, \\
C_{\oout}^{(2)}&=\{w \in \complexC: |w^N(w+1)^{L-N}|=e^{L^\epsilon} \rr^L |\mathrm{z}|, \,  \Re(w)\ge \Re(Q)\}.
\end{split}
\end{equation}
In other words, $C_\oout^{(1)}$ is a union of two line segments which have a shared endpoint $-\rho$ and the other endpoint is either $Q$ or $\bar Q$, and $C_\oout^{(2)}$ is a subset of $\Lambda'$. They are both symmetric about the real axis. See Figure~\ref{fig:contour_CK} for an illustration.

Now we make the choice of $Q$ explicit. Let
\beq
\label{eq:2019_11_10_02}
Q= -\rho+ \zeta''\sqrt{\rho(1-\rho)}L^{-1/2} \hspace{0.12cm} \text{with}\hspace{0.12cm}  \zeta'' = L^{\epsilon/2} + \ii \sqrt{3L^\epsilon + 2\log |\mathrm{z}|}  +O(L^{5\epsilon/2-1/2}).
\eeq
In the following lemma, we prove the existence of such a point on $\Lambda'$ and some properties of the contour $C_\oout^{(2)}$ which will be used in the asymptotic analysis.

\begin{lm}
	\label{lm:properties_Cout}
	There exists a point $Q$ with the form~\eqref{eq:2019_11_10_02} on $\Lambda'$. Moreover, with this $Q$, the contour $C_\oout=C_\oout^{(1)}\cup C_\oout^{(2)}$ is within $\Omega_+$ and $\dist(C_\oout^{(1)},\Gamma^\perp)\ge O(L^{-1/2})$, $\dist(C_\oout^{(2)},\Gamma^\perp)\ge O(L^{(\epsilon-1)/2})$. Here $\Gamma^\perp=\partial\Omega_+$.
\end{lm}

\begin{proof}[Proof of Lemma~\ref{lm:properties_Cout}]
	
	Since $0<\epsilon<1/8$, there exists a constant $\epsilon'$ satisfying $\epsilon<\epsilon'<1/8$. We fix this $\epsilon'$ and define the following disk centered at $-\rho$:
	\begin{equation*}
	\Omegac'=\{w\in\complexC: |w+\rho|\le \sqrt{\rho(1-\rho)}L^{-1/2+\epsilon'}\}.
	\end{equation*}
	We note that $\Omegac'$ is defined in the same way as $\Omegac$ in~\eqref{eq:diskOmgc}, except that we use $\epsilon'$ here but $\epsilon$ in~\eqref{eq:diskOmgc}. 
	To prove the lemma, we first show the following statements. 
	\begin{enumerate}[(a)]
		\item There exists a point $Q$ with the form~\eqref{eq:2019_11_10_02} on $\Lambda'\cap \Omegac'$. This implies the existence of $Q$ and hence the contours $C_\oout$, $C_\oout^{(1)}$, and $C_\oout^{(2)}$ are all well defined.
		\item $C_\oout^{(1)}\subset\Omegac'$. Moreover, $\dist(C_\oout^{(1)},\Gamma^\perp)\ge O(L^{-1/2})$.
		\item $\Omegac'$ contains a part of $C_\oout^{(2)}$. Moreover,
		$\dist(C_\oout^{(2)}\cap \Omegac',\Gamma^\perp)\ge O(L^{(\epsilon-1)/2})$.
		\item $\dist(C_\oout^{(2)}\setminus\Omegac',\Gamma^\perp)\ge O(L^{(\epsilon-1)/2})$.
	\end{enumerate}
	The first three statements need to be verified within the disk $\Omegac'$, hence we write down the explicit formulas of $\Gamma^\perp$ and $\Lambda'$ within $\Omegac'$. For any $w=-\rho + \sqrt{\rho(1-\rho)}\zeta L^{-1/2}\in \Omegac'$ with $|\zeta|\le L^{\epsilon'}$, we have
	\beq \label{eq:mmm1}
	\frac{w^N(w+1)^{L-N}}{\rr^L} = e^{-\zeta^2/2} (1+O(|\zeta|^3L^{-1/2}))
	\eeq
	and
	\beq \label{eq:mmm2}
	\frac{w(w+1)^{d-1}} {(-\rho)(1-\rho)^{d-1}}=1 -\frac{1}{2\rho} ((\zeta+\mu)^2-\mu^2) L^{-1} + O(L^{3\epsilon'-3/2}).
	\eeq
	Therefore, the equations for the contours $\Lambda'$ and $\Gamma^\perp=\partial\Omega_+$ (see \eqref{eq:Lambdapr} and  \eqref{eq:def_Gamma_perp}) in $\Omegac'$ can be written as  
	\beq
	\label{eq:aux_2019_09_01_01}
	\begin{split}
		& \Re(-\zeta^2/2) = L^{\epsilon} +\log |\mathrm{z}| +O(|\zeta|^3L^{-1/2})
	\end{split}
	\eeq
	and
	\beq
	\label{eq:aux_222222}
	\Im ((\zeta+\mu)^2) = O(L^{3\epsilon'-1/2}), \quad \Re((\zeta+\mu)^2)\le O(L^{3\epsilon'-1/2}) 
	\eeq
	respectively.
	From these two equations, we find that the intersection points of $\Lambda'$ and $\partial\Omega_+$ satisfy
	$w=-\rho+ \zeta' \sqrt{\rho(1-\rho)}$ with 
	\beqq
	\zeta'=-\mu \pm \ii\sqrt{2L^\epsilon+2\log|z| +\mu^2} + O(L^{3\epsilon'-(\epsilon+1)/2}).
	\eeqq
	We now set $Q$ and $\zeta''$ as in~\eqref{eq:2019_11_10_02}.
	It is straightforward to check that $\zeta''$ satisfies  \eqref{eq:aux_2019_09_01_01}. 
	Hence, we can choose the $O(L^{5\epsilon/2-1/2})$ term so that $Q\in \Lambda'$. 
	Hence $Q$ is on the contour $\Lambda'$.  This proves (a).
	
	For the first part of (b), we realize that $-\rho$ is the center of the disk $\Omegac'$, and $\dist(Q,-\rho)=O(L^{(\epsilon-1)/2})$ which is far less than $O(L^{\epsilon'-1/2})$, the radius of $\Omegac'$. Therefore the upper half of $C_\oout^{(1)}$, the line segment with end points $-\rho$ and $Q$, lies in $\Omegac'$. By symmetry, the lower half also lies in the same disk. Thus $C_\oout^{(1)}\subset \Omegac'$.
	To show the second part of (b), we need to further analyze~\eqref{eq:aux_222222}. It could be rewritten as $|\Re(\zeta+\mu)|\cdot |\Im(\zeta)|=O(L^{3\epsilon'-1/2})$ and  $|\Im(\zeta)|^2 \ge |\Re(\zeta+\mu)|^2 -O(L^{3\epsilon'-1/2})$. Therefore we have 
	\begin{equation}
	\label{eq:estimate_Gamma_perp}
	\left|\Re(\zeta+\mu)\right| \le O(L^{3\epsilon'-1/2})\quad \text{for }w=-\rho+ \zeta \sqrt{\rho(1-\rho)}L^{-1/2}\in \Gamma^\perp\cap \Omegac'.
	\end{equation}
	Recall the definition of $C_\oout^{(1)}$. For any point $w'\in C_\oout^{(1)}$, we have $\Re(w')\ge -\rho$. Thus by using~\eqref{eq:estimate_Gamma_perp} we have
	\begin{equation*}
	\begin{split}
	\dist(w',\Gamma^\perp\cap \Omegac')&\ge \min_{w\in \Gamma^\perp\cap \Omegac'}\Re(-\rho -w)\\
	&\ge \mu\sqrt{\rho(1-\rho)} L^{-1/2} -O(L^{3\epsilon'-1})\ge O(L^{-1/2}).
	\end{split}
	\end{equation*}
	On the other hand, by using the definition of $\Omegac'$ and $Q$, we have
	\begin{equation*}
	\begin{split}
	\dist(w',\Gamma^\perp\setminus \Omegac')) &\ge \dist(w',\partial\Omegac')\\
	&\ge \sqrt{\rho(1-\rho)}(L^{\epsilon'-1/2}-L^{(\epsilon-1)/2})\ge O(L^{\epsilon'-1/2}).
	\end{split}
	\end{equation*}
	The above two inequalities imply $\dist(C_\oout^{(1)},\Gamma^\perp)\ge O(L^{-1/2})$. We finish the proof of (b).
	
	Now we proceed to prove (c). We note that the above estimates~\eqref{eq:mmm1}, \eqref{eq:mmm2}, \eqref{eq:aux_2019_09_01_01}, \eqref{eq:aux_222222} and~\eqref{eq:estimate_Gamma_perp} still hold for $w\in \Omegac'':=\{w\in\complexC: |w+\rho|\le \sqrt{\rho(1-\rho)}L^{\epsilon''}\}$, with $\epsilon'$ replaced by $\epsilon''$ in the error bounds. Here $\epsilon''$ is a constant satisfying $\epsilon'<\epsilon''<1/8$. We will prove (c) by splitting $\Gamma^\perp$ to $\Gamma^\perp\cap\Omegac''$ and $\Gamma^\perp\setminus\Omegac''$ and then proceeding in the same way as we did for (b).
	Note that if $w'\in C_\oout^{(2)}\cap \Omegac'$, we have $\Re(w')\ge -\rho + \sqrt{\rho(1-\rho)}L^{(\epsilon-1)/2}+O(L^{5\epsilon/2-3/2})$ by the definition of $Q$ and the formula~\eqref{eq:aux_2019_09_01_01} of $\Lambda'$ in $\Omegac'$. Therefore by using~\eqref{eq:estimate_Gamma_perp} with $\epsilon',\Omegac'$ replaced by $\epsilon''$ and $\Omegac''$ respectively, we obtain
	\begin{equation*}
	\begin{split}
	&\dist(w',\Gamma^\perp\cap \Omegac'') \\
	&\ge \min_{w\in \Gamma^\perp\cap \Omegac''}\Re(\sqrt{\rho(1-\rho)}L^{(\epsilon-1)/2}-\rho -w) -O(L^{5\epsilon/2-3/2})\\
	&\ge \sqrt{\rho(1-\rho)}L^{(\epsilon-1)/2} +\mu\sqrt{\rho(1-\rho)} L^{-1/2} -O(L^{3\epsilon''-1})-O(L^{5\epsilon/2-3/2})\\
	&\ge O(L^{(\epsilon-1)/2}).
	\end{split}
	\end{equation*}
	On the other hand, since $w'\in\Omegac'$, we also have
	\begin{equation*}
	\dist(w',\Gamma^\perp\setminus\Omegac'') \ge \dist(\partial\Omegac',\partial\Omega'')= O(L^{\epsilon''-1/2})\ge O(L^{(\epsilon-1)/2}).
	\end{equation*}
	The above two inequalities imply that $\dist(C_\oout^{(2)}\cap \Omegac',\Gamma^\perp)\ge O(L^{(\epsilon-1)/2})$. We complete the proof of (c).
	
	Finally we show (d). We will prove that $C_\oout^{(2)}\setminus\Omegac'$ lies in $\Omega_\RR$, and hence in $\Omega_\RR\setminus\Omegac'$. Note that $\Gamma_\RR$ is a contour independent of $L$, and only intersects $\Gamma^\perp$ at the point $-d^{-1}$, see Figure~\ref{fig:contour2} for an illustration. The two contours have different angles at $-d^{-1}$ by a direction calculation: $\Gamma^\perp$ is vertical while $\Gamma_\RR=\partial\Omega_\RR$ has angle $\pm\pi/4$ at $-d^{-1}$.
	On the other hand $\Omegac'$ is a disk centered at $-\rho$ with radius $O(L^{\epsilon'-1/2})$. Therefore, $\dist(\Omega_\RR\setminus \Omegac',\Gamma^\perp)\ge O(L^{\epsilon'-1/2})$. This implies $\dist(C_\oout^{(2)}\setminus\Omegac',\Gamma^\perp)\ge O(L^{\epsilon'-1/2})$ and hence the statement (d).
	
	{The lemma is complete except that we still have to show that $C_\oout^{(2)}\setminus\Omegac'\subset \Omega_\RR$. Note this also implies  $C_\oout^{(2)}\setminus\Omegac'\subset \Omega_+$ since $\Omega_\RR\subset\Omega_+$. } Suppose $w\in C_\oout^{(2)}\setminus\Omegac'$.  Since $C_\oout^{(2)}\subset\Lambda'$ with $\Lambda'$ defined  in~\eqref{eq:Lambdapr}, we have
	\begin{equation*}
	\begin{split}
	|w^N(w+1)^{L-N}|&=\rho^N(1-\rho)^{L-N} e^{L^\epsilon}|\mathrm{z}| =(\rho(1-\rho)^{d-1})^{N} (1-\rho)^{L_s}e^{L^\epsilon}|\mathrm{z}| \\
	&\le (d^{-1}(1-d^{-1})^{d-1})^N (1-\rho)^{L_s}e^{L^\epsilon}=(\sfrr)^{dN}(1-\rho)^{L_s}e^{L^\epsilon}.
	\end{split}
	\end{equation*}
	On the other hand, by solving~\eqref{eq:aux_2019_09_01_01} at $|\zeta|=L^{\epsilon'}$ we find that the leftmost endpoints of $C_\oout^{(2)}\setminus \Omegac'$ have real parts at least $-\rho + \frac{1}{\sqrt{2}}\sqrt{\rho(1-\rho)}L^{\epsilon'-1/2} +O(L^{\epsilon-\epsilon'-1/2})$. We therefore have
	\begin{equation*}
	|w+1|^{L_s}\ge \left(1-\rho+ \frac{1}{\sqrt{2}}\sqrt{\rho(1-\rho)}L^{\epsilon'-1/2} +O(L^{\epsilon-\epsilon'-1/2})\right)^{L_s}\ge (1-\rho)^{L_s}e^{c_4L^{\epsilon'}}
	\end{equation*}
	for some positive constant $c_4$, where we used the inequality $(1+x^{-1})^{x+1}\ge e$ for all $x>0$, and the fact that $L_s=O(L^{1/2})$.
	Note that $\epsilon'>\epsilon$. The above two inequalities imply
	\begin{equation*}
	\begin{split}
	\left|w(w+1)^{d-1}\right|&=\left|\frac{w^N(w+1)^{L-N}}{(w+1)^{L_s}}\right|^{1/N}\le (\sfrr)^{d}e^{(-c_4L^{\epsilon'}+L^\epsilon)/N}\le (\sfrr)^{d}e^{-c_5L^{\epsilon'-1}}
	\end{split}
	\end{equation*}
	for some positive constant $c_5$ and large enough $L$. Together with the fact that $C_\oout^{(2)}$ lies in the right half plane $\Re(w)>-\rho>-d^{-1}$, we conclude that $w\in\Omega_\RR$. We complete the proof.
\end{proof}

\bigskip

Now we proceed to prove Lemma~\ref{lm:energy_stof_asymptotics2}.
Recall the definition of $U(w)$ in~\eqref{eq:def_U_w}.
By Lemma~\ref{lm:Omega_plus}, for every $w'\in\Omega_+$, the only solution of the equation $w(w+1)^{d-1}= w'(w'+1)^{d-1}$ of $w$ in the set $\Omega_+\cup \Gamma^\perp$ is $w=w'$. 
Hence for $w'\in\Omega_+$, the set $U(w')$ consists of $d-1$ points and these points are outside $\Omega_+\cup \Gamma^\perp$. 
We now define 
\begin{equation}
\label{eq:def_f_stof}
f_{\stof}(w,w'):= \sum_{u\in U(w')}\log (w- u) \qquad \text{for $w\in \Omega_+\cup\Gamma^\perp$ and $w'\in\Omega_+$.} 
\end{equation}
The $\log$ is the principal branch of the logarithm function. The above function is well-defined since $w-u$ does not lie in $\realR_{\le 0}$ by Lemma~\ref{lm:Omega_plus}(v). 
This function is analytic in $w'\in \Omega_+$ and it satisfies 
\beq
\label{eq:aux_2019_11_07_01}
g(w, w')= e^{f_{\stof}(w, w')} \qquad \text{for $w\in\Omega_+\cup\Gamma^\perp$ and $ w'\in \Omega_+$.}
\eeq
Note that since $U(0)$ consists of $d-1$ copies of the point $-1$, we have 
\begin{equation}
\label{eq:aux_2019_08_25_01}
f_{\stof}(w,0) = (d-1)\log (w+1).
\end{equation}

We now express the logarithm of the left-hand side of the equation in Lemma  \ref{lm:energy_stof_asymptotics2} as an integral. 
For fixed $u\in\SU_z$, we apply Lemma~\ref{lem:sumtointres} (a) with $p(w)=\log (w-u) $ and obtain
\beqq
\sum_{v\in\rootsR_z}  \log (v-u) - N \log (-u)  = \frac{Lz^L}{2\pi\ii} \oint_{C_\oout} \frac{\log(w-u) (w+\rho)}{w(w+1) q_z(w)}\dd w.
\eeqq
Taking the sum of the the above formula over all $u\in \SU_z$,
\beqq
\sum_{u\in\SU_z}\sum_{v\in\rootsR_z} \log (v-u) - \sum_{u\in\SU_z} \log (-u) 
=\frac{Lz^L}{2\pi\ii} \oint_{C_\oout} \frac{ \left( \sum_{v\in\rootsR_z}f_{\stof}(w,v) \right) (w+\rho)}{w(w+1) q_z(w)} \dd w.
\eeqq
Applying Lemma~\ref{lem:sumtointres} (a) for the sum in the integrand, we find for each $w\in C_\oout$, 
\beqq
\sum_{v\in\rootsR_z}f_{\stof}(w,v) = Nf_{\stof}(w, 0) + \frac{Lz^{L}}{2\pi\ii} \oint_{C_\oout} \frac{f_{\stof}(w,w') (w'+\rho)}{w'(w'+1) q_{z}(w')}\dd w'. 
\eeqq
We insert this equation to the previous identity. The integral involving $f_{\stof}(w,0)$ becomes, 
using $f_{\stof}(w,0)=(d-1)\log(w+1)$ and the residue theorem, 
\beqq
\frac{NLz^L}{2\pi\ii} \oint_{C_\oout} \frac{f_{\stof}(w,0) (w+\rho)}{w(w+1) q_z(w)} \dd w 
=(d-1)N\sum_{v\in\rootsR_z}\log(w+1).
\eeqq
Hence, we find that 
\begin{equation}
\label{eq:aux_2019_08_22_01}
\begin{split}
&\sum_{u\in\SU_z}\sum_{v\in\rootsR_z}\log (v-u) - N\sum_{u\in\SU_z}\log(-u) -(d-1)N\sum_{v\in\rootsR_z} \log(v+1) \\
&= \frac{L^2z^{2L}}{(2\pi\ii)^2} \oint_{C_\oout}\oint_{C_\oout} \frac{f_{\stof}(w,w') (w+\rho)(w'+\rho)}{w(w+1) q_z(w) w'(w'+1) q_{z}(w')} \dd w \dd w'.
\end{split}
\end{equation}
If we repeat the above calculations with the function $p(w)=\log(w-v)$ replaced by a constant function, we find that the last double integral with $f_{\stof}(w,w')$ replaced by a constant is zero. 
Hence, we may rewrite the right hand side of~\eqref{eq:aux_2019_08_22_01} as
\begin{equation}
\label{eq:aux_2019_08_22_02}
\frac{L^2z^{2L}}{(2\pi\ii)^2}\oint_{C_\oout}\oint_{C_\oout} \frac{(f_{\stof}(w,w')- C')(w+\rho)(w'+\rho)}{w(w+1) q_z(w) w'(w'+1) q_{z}(w')} \dd w \dd w'
\end{equation}
for any constant $C'$. 
We choose  $C'=\sum_{i=1}^{d-2}\log (-\rho-c_i)+ \log\left(L^{-1/2}\sqrt{\rho(1-\rho)}\right)$, where $c_i$ are the points defined in Lemma~\ref{lm:roots_stof_estimate}.

Now we evaluate the limit of~\eqref{eq:aux_2019_08_22_02}.

Recall that $C_\oout^{(2)}\subset\Lambda'$ with $\Lambda'$ defined in~\eqref{eq:Lambdapr}. We have that $|q_z(w)| \ge |z|^L  (e^{L^\epsilon}-1)$ for all $w\in C_\oout^{(2)}$. The same equation~\eqref{eq:Lambdapr} also implies $|w|^N|w+1|^{L-N}<C^L$ for some positive constant $C$. Hence the contour $C_\oout^{(2)}$, and further $C_\oout$, is uniformly bounded. On the other hand, $f_{\stof}(w,w')$ is analytic for $w,w'$ in $\Omega_+$. Thus it is uniformly bounded for $w,w'\in C_\oout$. These discussions imply that the part of the integral~\eqref{eq:aux_2019_08_22_02} when both $w$ and $w'$ are in $C_{\oout}^{(2)}$ gives an $O(e^{-L^{\epsilon}})$ term.

If $w\in C_\oout^{(2)}$ but $w'\in C_\oout^{(1)}$, by Lemma~\ref{lm:roots_stof_estimate}, we have
\begin{equation*}
f_{\stof}(w,w')= \sum_{i=1}^{d-2}\log (w-c_i(w')) +\log(w-\hat w').
\end{equation*}
Using the fact that $\Re(w')\ge -\rho$ for all $w'\in C_\oout^{(1)}$, we have $\Re(\hat w')<-\rho$. We also have $\Re(c_i(w'))<-\rho$, and $\Re(w)\ge -\rho+ O(L^{(\epsilon-1)/2})$. These inequalities imply that $|\log (w-c_i(w'))|$ and $|\log(w-\hat w')|$ are bounded by $O(\log L)$. Thus $|f_{\stof}(w,w')|\le O(\log L)$.
On the other hand, $q_z(w')/z^L\approx e^{-(\zeta')^2/2}\mathrm{z}^{-1}-1+O(L^{(3\epsilon-1)/2})$ for $w'=-\rho+\zeta'\sqrt{\rho(1-\rho)}L^{-1/2}\in C_\oout^{(1)}$. By the definition of $C_\oout^{(1)}$, $\zeta'$ is on the union of two line segments starting from the origin. By ~\eqref{eq:2019_11_10_02}, the slopes of the two line segments are approximately $\pm\sqrt{3}$. This implies $|e^{-(\zeta')^2/2}|\ge 1$ for all $\zeta'$. Thus $|q_z(w')/z^L|\ge |\mathrm{z}^{-1}|-1>0$ for sufficiently large $L$.
Combining the above estimates, and recalling that $|q_z(w)| \ge |z|^L  (e^{L^\epsilon}-1)$ for all $w\in C_\oout^{(2)}$, we obtain that the double integral~\eqref{eq:aux_2019_08_22_02} when $w\in C_{\oout}^{(2)}$ but $w'\in C_{\oout}^{(1)}$ contributes an $O(e^{-L^{\epsilon}})$ term. By symmetry, the contribution when $w\in C_\oout^{(1)}$ and $w'\in C_{\oout}^{(2)}$ in this integral is also at most of order $O(e^{-L^{\epsilon}})$.

Consider the part when both variables are in $C_{\oout}^{(1)}$. 
From the definition of the end point $Q$, 
for $w=-\rho+\zeta\sqrt{\rho(1-\rho)}L^{-1/2}$ and $w'=-\rho+\zeta'\sqrt{\rho(1-\rho)}L^{-1/2}$ on $C_{\oout}^{(1)}$, we have $|\zeta|= O(L^{\epsilon})$ and $|\zeta'|=O(L^{\epsilon})$. For such $w$ and $w'$, we have (cf.  Lemma~\ref{lm:roots_stof_estimate}) 
\begin{equation*}
\begin{split}
&f_{\stof}(w,w')\\
&= \sum_{i=1}^{d-2}\log (w-c_i(w')) +\log(w-\hat w') \\
&=\sum_{i=1}^{d-2}\log (-\rho-c_i) +\log (\zeta +\zeta'+2\mu) + \log\left(L^{-1/2}\sqrt{\rho(1-\rho)}\right)+ O(L^{\epsilon-1/2}).
\end{split}
\end{equation*}
Inserting this formula to~\eqref{eq:aux_2019_08_22_02}, the double integral becomes 
\begin{equation*}
\frac{\mathrm{z}^2}{(2\pi\ii)^2}\int_{\mathrm{C}^{(1)}}\int_{\mathrm{C}^{(1)}} \frac{\zeta\zeta'\log(\zeta +\zeta'+2\mu)} {(e^{-\zeta^2/2}-\mathrm{z})(e^{-(\zeta')^2/2}-\mathrm{z})} \dd\zeta\dd\zeta' +O(L^{\epsilon-1/2}),
\end{equation*}
where the contour $\mathrm{C}^{(1)} = \{\zeta =L^{1/2}(w+\rho)/\sqrt{\rho(1-\rho)}: w\in C_\oout^{(1)}\}$. 
Extending the contours vertically and then deforming the contours to the imaginary axis, we obtain 
\begin{equation*}
\begin{split}
&\sum_{u\in\SU_z}\sum_{v\in\rootsR_z}\log (v-u) - N\sum_{u\in\SU_z}\log(-u) -(d-1)N\sum_{v\in\rootsR_z} \log(v+1) \\
&= \frac{\mathrm{z}^2}{(2\pi\ii)^2}\int_{\ii\realR}\int_{\ii\realR} \frac{\zeta\zeta'\log(\zeta +\zeta'+2\mu)} {(e^{-\zeta^2/2}-\mathrm{z})(e^{-(\zeta')^2/2}-\mathrm{z})} \dd\zeta\dd\zeta' +O(L^{\epsilon-1/2}).
\end{split}
\end{equation*}
This completes the proof of Lemma~\ref{lm:energy_stof_asymptotics2}.

\subsubsection{Proof of Lemma~\ref{lm:assumptionC_03} and Lemma~\ref{lm:assumptionC_02}}
\label{sec:proof_assumptionC}

We first prove the following lemma from which we prove the two lemmas. 
Recall the region $\Omega_+$ and $\Gamma^\perp=\partial \Omega_+$ discussed in Lemma~\eqref{lm:Omega_plus}.

\begin{lm}
	\label{lm:uniform_bounds}
	For every subset $A$ of $\Omega_+\cup\Gamma^\perp$ that is uniformly bounded (with respect to $L$),  there are postive  constants $c$ and $C$  such that
	\begin{equation*}
	c\le 
	\left| \frac{\prod_{v\in\rootsR_z}g(w,v)}{(w+1)^{(d-1)N}}   \right|
	\le C
	\end{equation*}
	for all $w\in A$ and all sufficiently large $L$.
\end{lm}

\begin{proof}
	By \eqref{eq:aux_2019_11_07_01} and \eqref{eq:aux_2019_08_25_01}, we have 
	\begin{equation*}
	\frac{\prod_{v\in\rootsR_z}g(w,v)}{(w+1)^{(d-1)N}}
	=  e^{ \sum_{v\in\rootsR_z} f_{\stof}(w, v) - N f_{\stof}(w, 0)}
	\end{equation*}
	for the function $f_{\stof}$ defined in \eqref{eq:def_f_stof}. 
	Hence, the lemma is proved if we show that there is a constant $C>0$ such that 
	\begin{equation*}
	\left| \sum_{v\in\rootsR_z} f_{\stof}(w, v) - N f_{\stof}(w, 0) \right| \le C
	\end{equation*}
	for all $w\in A$ and sufficiently large $L$. 
	
	We use a similar computation of the previous subsection; we write the sum as an integral and use the method of steepest-descent. 
	As before, Lemma \ref{lem:sumtointres} implies
	\begin{equation*}
	\begin{split}
	\sum_{v\in\rootsR_{z}} f_{\stof}(w, v) - N f_{\stof}(w, 0) 
	= \frac{Lz^L}{2\pi \ii} \oint_{C_\oout} \frac{(v+\rho)f_{\stof}(w, v)}{v(v+1)q_z(v)} \dd v.
	\end{split}
	\end{equation*}
	We note that the integral is zero if we replace $f_{\stof}(v,w)$ by a constant and change the integral to 
	\begin{equation*}
	\frac{Lz^L}{2\pi \ii} \oint_{C_\oout} \frac{(v+\rho)\left(f_{\stof}(w, v)-f_{\stof}(w, -\rho)\right)}{v(v+1)q_z(v)} \dd v
	\end{equation*}
	without changing the value. 
	The contour is $C_\oout=C_{\oout}^{(1)}\cup C_{\oout}^{(2)}$. 
	
	Consider the integral over $C_{\oout}^{(2)}$. 
	We show that this part converges to zero. 
	{Since $|q_z(v)|\ge |z|^L (e^{L^\epsilon}-1)$ for $v\in C_{\oout}^{(2)}$ (see~\eqref{eq:Lambdapr} and~\eqref{eq:def_Cout} above),} it is enough to show that there is $C'>0$ such that 
	\begin{equation}
	\label{eq:aux_2019_11_07_05}
	|f_{\stof}(w, v)| + |f_{\stof}(w, -\rho)| \le C'\log L \qquad \text{for all $w\in A$ and $v\in C_\oout^{(2)}$.}
	\end{equation}
	Recall that $f_{\stof}(w, v)=f_{\stof}(v,w)$ is the sum of $\log(v-u)$ over $u\in U(w)$. 
	By Lemma~\ref{lm:Omega_plus}, for every $w\in A\subset \Omega_+\cup\Gamma^\perp$, all solutions of  $u(u+1)^{d-1}=w(w+1)^{d-1}$, except for $u=w$, are outside $\Omega_+$. 
	Hence, all elements of $U(w)$ are outside $\Omega_+$. 
	The union of $U(w)$ over $w\in A$ is a bounded set since $A$ is bounded. {By Lemma~\ref{lm:properties_Cout},}  $\dist (C_\oout^{(2)},\Gamma^\perp ) \ge O(L^{-1/2})$. {Therefore $|v-u|\ge O(L^{-1/2})$ for all $v\in C_\oout^{(2)}$ and $u\in U(w)$ with $w\in A$.} 
	On the other hand, a point $v\in C_\oout^{(2)}$ satisfies $|v^N(v+1)^{L-N}|=e^{L^\epsilon}\rr^L |\mathrm{z}|$. 
	From this, we find that $v$ remains in a bounded set for all large enough $L$. 
	Hence, $C_\oout^{(2)}$ is in bounded set for all $L$. {Thus $|v-u|$ is bounded above by $O(1)$.}
	Combining the above discussions, we find that there are positive constants $C_2$ and $C_3$ such that  
	\begin{equation*}
	C_2 L^{-1/2}\le \dist (v, U(w))\le C_3
	\end{equation*}
	for all $v\in C_\oout^{(2)}$ and $w\in A$. 
	Since there are $d-1$ elements in the set $U(w)$, there is a constant $C'>0$ such that 
	\begin{equation*}
	\begin{split}
	|f_{\stof}(w, v)|
	&\le \sum_{u\in U(w)}|\log (v-u)| \le C' \log L
	\end{split}
	\end{equation*}
	for all $v\in C_\oout^{(2)}$ and $w\in A$ if $L$ is sufficiently large. 
	The above estimate also holds for $v=-\rho\in\Omega_+$ since $-\rho\ge -d^{-1}+C_4L^{-1/2}$ and it has at least $O(L^{-1/2})$ distance to $\Gamma^\perp$. 
	This completes the proof of~\eqref{eq:aux_2019_11_07_05}.
	
	It remains to show that there is a constant $C>0$ such that 
	\begin{equation}
	\label{eq:aux_2019_11_07_06}
	\left|\frac{Lz^L}{2\pi \ii} \int_{C_\oout^{(1)}} \frac{(v+\rho)\left(f_{\stof}(w, v)-f_{\stof}(w, -\rho)\right)}{v(v+1)q_z(v)} \dd v \right| \le C.
	\end{equation}
	By the definition, $C_{\oout}^{(1)}$ lies in the disk $\Omegac$, which was defined in~\eqref{eq:diskOmgc}.
	We use the change of variables from $v$ to $\xi$ by $v= -\rho+\xi\sqrt{\rho(1-\rho)}L^{-1/2}$. 
	Under this change of variables, we have 
	\beqq
	Lz^L \frac{(v+\rho)}{v(v+1)q_z(v)} \dd v
	= \mathrm{z} \frac{\xi}{e^{-\xi^2/2} -\mathrm{z}} (1+ O(L^{\epsilon-1/2}) \dd \xi,
	\eeqq
	where the error term is uniform in $\xi$.  
	Since $f_{\stof}(w,w')=\sum_{u'\in U(w')}\log(w-u')$, 
	we find from Lemma~\ref{lm:roots_stof_estimate} that the integral \eqref{eq:aux_2019_11_07_06} is equal to 
	\begin{equation}
	\label{eq:2019_11_10_01}
	\begin{split}
	&\sum_{i=1}^{d-2} \frac{\mathrm{z}}{2\pi\ii} \int \frac{\xi \left(\log (w-c_i(v)) -\log (w-c_i(-\rho))\right)}{e^{-\xi^2/2}-\mathrm{z}} (1+O(L^{\epsilon-1/2}))\dd \xi\\
	&\qquad +\frac{\mathrm{z}}{2\pi\ii} \int \frac{\xi  (\log (w-\widehat{v} )-\log (w-\widehat{(-\rho)}))}{e^{-\xi^2/2}-\mathrm{z}} (1+O(L^{\epsilon-1/2}))\dd \xi,
	\end{split}
	\end{equation}
	where $v=v(\xi)= -\rho+\xi\sqrt{\rho(1-\rho)}L^{-1/2}$. 
	Since $C_{\oout}^{(1)}$ is the union of two line segment, so is the contour of the above integrals. 
	By ~\eqref{eq:2019_11_10_02}, the slopes of the two line segments are approximately $\pm\sqrt{3}$.
	This fact has the implication that $e^{-\xi^2/2}-\mathrm{z}$ grows super-exponentially as $|\xi|\to \infty$ along the direction of the contour. 
	
	For the first term of \eqref{eq:2019_11_10_01}, we note that the points $c_i(v)$ in Lemma~\ref{lm:roots_stof_estimate} satisfy $c_i(v)=c_i + O(L^{2\epsilon-1})$.
	Hence $w$ in $A$ has finite distance between $c_i(v)$ uniformly in both $w$ and $v$. 
	Therefore, we see that $|\log (w-c_i(v))|$ is uniformly bounded. The same holds for $|\log (w-c_i(-\rho))|$. 
	Hence, considering the growth of $e^{-\xi^2/2}-\mathrm{z}$, we see that the first term in \eqref{eq:2019_11_10_01} is $O(1)$. 
	
	Now consider the second term of \eqref{eq:2019_11_10_01}.
	By Lemma~\ref{lm:roots_stof_estimate}, 
	\begin{equation*}
	\widehat{v} =-\rho + (-2\mu -\xi)\sqrt{\rho(1-\rho)}L^{-1/2} +O(L^{2\epsilon-1})
	\end{equation*}
	and
	\begin{equation*}
	\widehat{(-\rho)} = -\rho -2\mu\sqrt{\rho(1-\rho)}L^{-1/2} +O(L^{2\epsilon-1}).
	\end{equation*}
	Since $|\log(\alpha)-\log(\beta)|\le 2\pi + |\log(\frac{\alpha}{\beta})|$ for all $\alpha, \beta\in \complexC\setminus (-\infty, 0]$, we have 
	\begin{equation*}
	\begin{split}
	&\left|\log (w-\widehat{v} )-\log (w-\widehat{(-\rho)})\right|\le 2\pi + \left|\log\left(1+\frac{\xi\sqrt{\rho(1-\rho)}}{L^{1/2}(w+\rho) + 2\mu\sqrt{\rho(1-\rho)}}+O(L^{2\epsilon-1/2})\right)\right|.
	\end{split}
	\end{equation*}
	Since $w\in A\subset \Omega_+$, we have from \eqref{eq:OMRPH} that $\Re(w)\ge -d^{-1}$. 
	Using \eqref{eq:dintermsofrho} for the formula of $d^{-1}$, we find that 
	\begin{equation} \label{eq:wlbaa}
	\begin{split}
	&\Re(L^{1/2}(w+\rho) + 2\mu\sqrt{\rho(1-\rho)}) \\
	&
	\ge   \Re(L^{1/2}(-d^{-1}+ \rho) + 2\mu\sqrt{\rho(1-\rho)}) \\
	&\ge \mu \sqrt{\rho(1-\rho)} + O(L^{-1/2}) . 
	\end{split}
	\end{equation}
	This shows in particular that $\arg(L^{1/2}(w+\rho) + 2\mu\sqrt{\rho(1-\rho)}) \in [-\pi/2, \pi/2]$. 
	We also observe that $\xi$ is on the union of line segments whose slopes converge to $\pm \sqrt{3}$ as $L\to \infty$. 
	If two complex numbers satisfy $\arg(\alpha)\in [-\pi/3-\delta, \pi/3+\delta]$ for a sufficiently small $\delta>0$ and $\arg(\beta)\in [-\pi/2, \pi/2]$, then $\arg(\alpha/\beta)\in [-5\pi/6-\delta, 5\pi/6+\delta]$. 
	Since there are positive constants $C_8, C_9$ such that  $|\log(1-z)|\le C_8|z|+C_9$ for all $z$ in $\arg(z)\in [-11\pi/12, , 11\pi/12]$, 
	we conclude that 
	\begin{equation*}
	\begin{split}
	&\left|\log (w-\widehat{v} )-\log (w-\widehat{(-\rho)})\right|
	\le 2\pi + C_9 +  \frac{ C_8 |\xi| \sqrt{\rho(1-\rho)}}{ | L^{1/2}(w+\rho) + 2\mu\sqrt{\rho(1-\rho)} | }+O(L^{2\epsilon-1/2}) . 
	\end{split}
	\end{equation*}
	Now, the inequality \eqref{eq:wlbaa} shows that the absolute value $| L^{1/2}(w+\rho) + 2\mu\sqrt{\rho(1-\rho)} |$ is bounded below uniformly for $w$. 
	Hence, we find that there are positive constants $C_{10}, C_{11}$ such that 
	\begin{equation*}
	\begin{split}
	\left|\log (w-\widehat{v} )-\log (w-\widehat{(-\rho)})\right|
	\le C_{10}|\xi|+C_{11}
	\end{split}
	\end{equation*}
	uniformly in $w\in A$ and $\xi$, for all large enough $L$. 
	Hence, considering the growth of the function $e^{-\xi^2/2}-\mathrm{z}$, we find that the second term of \eqref{eq:2019_11_10_01} is $O(1)$. 
	
	We thus conclude that ~\eqref{eq:2019_11_10_01} is $O(1)$ and we obtain the proof of~\eqref{eq:aux_2019_11_07_06}.
\end{proof}

We now prove Lemma~\ref{lm:assumptionC_03} and Lemma~\ref{lm:assumptionC_02}. 

\begin{proof}[Proof of Lemma~\ref{lm:assumptionC_03}]
	By Lemma~\ref{lm:locations_Lambda}, we have $\rootsR_z\subset \Omega_\RR$. 
	Furthermore, $\rootsR_z$ is uniformly bounded for all $L$. Hence, Lemma~\ref{lm:uniform_bounds} implies the lemma. 
\end{proof}

\begin{proof}[Proof of Lemma~\ref{lm:assumptionC_02}]
	The points $u\in \rootsL_z$ satisfies $\Re(u)<-\rho$. Hence, $|u|\ge \rho$. Using this fact and the Bethe equation, $|u^N(u+1)^{L-N}|=\rr^L |\mathrm{z}|= \rho^N(1-\rho)^{L-N} |\mathrm{z}|$, 
	we find that $|u+1|^{L-N}\le (1-\rho)^{L-N} |\mathrm{z}|\le (1-\rho)^{L-N}$. Hence 
	\beq \label{eq:upol}
	|u+1|\le 1-\rho \qquad \text{for $u\in \rootsL_z$.} 
	\eeq

	By Lemma~\ref{lm:Omega_plus}, for every $u\in \rootsL_z$, there is a unique $\tilde u\in \Omega_+\cup\Gamma^\perp_{\upp}$ such that
	\begin{equation}
	\label{eq:2019_11_10_10}
	\tilde u (\tilde u+1)^{d-1} = u (u+1)^{d-1}.
	\end{equation}
	The estimate $|u+1|\le 1-\rho$ for $u\in\rootsL_z$ implies that $\tilde u$ stays in a bounded subset of $\Omega_+\cup\Gamma^\perp$. Thus, by Lemma~\ref{lm:uniform_bounds}, there is $C>0$ such that 
	\begin{equation}
	\label{eq:2019_11_10_12}
	\left| \frac{\prod_{v'\in\rootsR_z}  g(\tilde u, v')}{(\tilde u+1)^{(d-1)N}}\right|\le C
	\end{equation}
	for all $u\in \rootsL_z$.

	Now we consider three different cases and show that in each case we have 
	\begin{equation}
	\label{eq:2019_11_12_01}
	\left| \frac{\prod_{v'\in\rootsR_z}  g(u, v')}{( u+1)^{(d-1)N}}\right|\le C.
	\end{equation}

	(a) Case 1: $\tilde u=u$. In this case, we clearly obtain \eqref{eq:2019_11_12_01} from~\eqref{eq:2019_11_10_12}. We remark that this case occurs when Lemma~\ref{lm:locations_Lambda} (iii) holds. 
	
	(b) Case 2: $\tilde u\ne u$ and $\Re(\tilde u)\le -\rho$. 
	From the definition of $g$ and \eqref{eq:2019_11_10_10}, we have $g(u, v')= g(\tilde u, v') \frac{\tilde u-v'}{u-v'}$. 
	Hence,
	\beqq
	\frac{ \prod_{v'\in\rootsR_z}  g(u, v')}{( u+1)^{(d-1)N}} = \frac{ \prod_{v'\in\rootsR_z}  g(\tilde u, v')}{( u+1)^{(d-1)N}} \prod_{v'\in\rootsR_z} \frac{\tilde u- v'}{u-v'}. 
	\eeqq
	Using \eqref{eq:2019_11_10_10} one more time, 
	\begin{equation}
	\label{eq:2019_11_12_02}
	\begin{split}
	\frac{ \prod_{v'\in\rootsR_z}  g(u, v')}{( u+1)^{(d-1)N}}
	= \left[ \frac{ \prod_{v'\in\rootsR_z}  g(\tilde u, v')}{(\tilde u+1)^{(d-1)N}} \right]
	\left[ \frac{u^N}{\prod_{v'\in\rootsR_z} (u-v')} \right]  
	\left[ \frac {\prod_{v'\in\rootsR_z} (\tilde u - v')}{\tilde u^N}  \right]. 
	\end{split}
	\end{equation}
	Note that we did not use the condition $\Re(\tilde u)\le -\rho$ in deriving the above formula. 
	The first factor is bounded by \eqref{eq:2019_11_10_12}.
	Noting that $\Re(u)<-\rho$, Lemma~\ref{lem:asymofprod} (the second equation of (i) and (ii)) and Lemma~\ref{lem:hpropperty} (f) imply that the second factor is bounded.
	Since we assume that $\Re(\tilde u)\le -\rho$, the third factor is also bounded by the same lemmas, 
	Thus, we obtain \eqref{eq:2019_11_12_01}.

	(c) Case 3: $\tilde u\ne u$ and $\Re(\tilde u) > -\rho$. 
	Since $\rootsL_z\cup \rootsR_z$ is the set of the roots of the Bethe equation, we have $w^N(w+1)^{L-N}-z^L=\prod_{u'\in \rootsL_z} (w-u') \prod_{v'\in \rootsR_z} (w-v')$. 
	We use this identity to replace the product $\prod_{v'\in\rootsR_z} (\tilde u - v')$ to the product $\prod_{u'\in\rootsL_z} (\tilde u -u')$ in the formula~\eqref{eq:2019_11_12_02}. We have 
	\begin{equation}
	\label{eq:2019_11_10_11}
	\begin{split}
	\frac{ \prod_{v'\in\rootsR_z}  g(u, v')}{( u+1)^{(d-1)N}}
	=& \left[ \frac{ \prod_{v'\in\rootsR_z}  g(\tilde u, v')}{(\tilde u+1)^{(d-1)N}} \right]
	\left[ \frac{u^N}{\prod_{v'\in\rootsR_z} (u-v')} \right] 
	\left[ \frac{(\tilde u+1)^{L-N}} { \prod_{u'\in\rootsL_z} (\tilde u - u')} \right]
	\left( 1- \frac{z^L}{\tilde u^N (\tilde u+1)^{L-N}} \right). 
	\end{split}
	\end{equation}
	As in the case (b), 
	the first factor is bounded by \eqref{eq:2019_11_10_12} and the second factor is bounded by Lemma~\ref{lem:asymofprod} and Lemma~\ref{lem:hpropperty} (f). 
	In this case, we have $\Re(\tilde{u})> -\rho$. Thus, the third factor is bounded by the first equation of (i) and (ii) of Lemma~\ref{lem:asymofprod}, and  Lemma~\ref{lem:hpropperty} (f). 
	We show that the last factor is bounded by $2$. 
	By the Bethe equation, we have $z^L= u^N(u+1)^{L-N}$. Using $L=dN+L_s$ with $L_s>0$ and the identity \eqref{eq:2019_11_10_10}, we have 
	\beqq
	\frac{z^L}{\tilde u^N (\tilde u+1)^{L-N}}
	= \frac{u^N(u+1)^{L-N}}{\tilde u^N (\tilde u+1)^{L-N}}
	= \frac{(u+1)^{L_s}}{(\tilde u+1)^{L_s}} .
	\eeqq
	Since $\Re(\tilde u)\ge -\rho$, we have $|\tilde u+1|\ge 1-\rho$. 
	On the other hand, we have $|u+1|\le 1-\rho$ from \eqref{eq:upol}. 
	Hence, we have $|u+1|\le |\tilde u +1|$, and therefore, the last factor of \eqref{eq:2019_11_10_11} is bounded (by $2$.) 
	Hence, we obtain \eqref{eq:2019_11_12_01} in this case, too. 
\end{proof}

\section{Proof of Theorem~\ref{thm:special_IC} (i)}
\label{sec:proof_special_IC}

Even though the flat initial condition is a special case of the step-flat initial condition, 
we present the proof of the limit theorem separately since the analysis becomes much simpler. 
The simplification is due to the fact that $L=dN$ implies that the Bethe equation $w^N(w+1)^{L-N}=z^L$ can be factored and the solutions satisfy $w(w+1)^{d-1}= z^de^{2\pi \ii k/N}$ for some $1\le k\le N$. 
See Lemma \ref{lm:Berrsfl} below for the structure of the Bethe roots in the flat case. 

We show that the flat initial condition satisfies Assumption~\ref{def:asympstab}. Recall $Y_\flat= (-(N-1)d, \cdots, -2d, -d, 0)$ and
\beqq
\lambda(Y_\flat) = (0, -(d-1), -2(d-1), \cdots, -(N-1)(d-1)). 
\eeqq
They are the same as $Y_\stof$ and $\lambda(Y_\stof)$ for the step-flat initial condition. 
Thus, $\gftn_{\lambda (Y_\flat)}(W)$ is same as $\gftn_{\lambda (Y_\stof)}(W)$ for the step-flat case;  
\beq
\label{eq:aux_087}
\begin{split}
	&\gftn_{\lambda (Y_\flat)}(W)
	=  \left[\prod_{i=1}^N (w_i+1)^{-(N-1)(d-1)}\right]  
	\left[ \prod_{1\le i<j\le N} \frac{w_i(w_i+1)^{d-1} - w_j(w_j+1)^{d-1}}{w_i-w_j} \right] .
\end{split}
\eeq
Therefore, the global energy function and the characteristic function have the same formulas as the step-flat case given in Lemma \ref{lem:energy_stof} and equation \eqref{eq:aux_2019_08_21_07}. 
These formulas can be further simplified in the flat case by the following lemma.

\begin{lm} \label{lm:Berrsfl}
	Let $L=dN$ and let $0<|z|<\rr$. 
	Then, the set $\rootsL_z$ can be partitioned to $N$ subsets $\rootsL_z^{(i)}$, $1\le i\le N$, of $d-1$ elements such that there is a map $M: \rootsL_z \to \rootsR_z$ satisfying $M(u)=M(u')$ for $u$ and $u'$ in the same subset $\rootsL_z^{(i)}$ and, furthermore, $v=M(u)$ satisfies $v(v+1)^{d-1}=u(u+1)^{d-1}$. 
\end{lm}

\begin{proof}
	Fix $v\in \rootsR_z$. Then, $v^N(v+1)^{L-N}= z^L$. Since $L=dN$, we have $|v(v+1)^{d-1}|= |z^d|$.
	Thus, from the structure of the Bethe roots (when $N=1$ and $L=d$ case), the equation $w(w+1)^{d-1}=v(v+1)^{d-1}$ of $w$ has $d-1$ distinct solutions satisfying $\Re(w)<-\rho$ and one solution is given by $w=v$, which satisfies $\Re(w)>-\rho$. 
	Since the polynomial $w(w+1)^{d-1}-v(v+1)^{d-1}$ of $w$ is a factor of the polynomial $w^N(w+1)^{L-N}- v^N (v+1)^{L-N}$, we find that the $d-1$ solutions satisfying $\Re(w)<-\rho$ are subsets of $\rootsL_z$. 
	Thus, for each $v\in \rootsR_z$, there are $d-1$ distinct $u\in \rootsL_z$ such that $u(u+1)^{d-1}=v(v+1)^{d-1}$. 
	The lemma is proved if we show that the sets of such $u$ associated to two different points $v$ and $v'$ in $\rootsR_z$ are disjoint. 
	Note that if $v$ and $v'$ are two different points in $\rootsR_z$, then $v(v+1)^{d-1}\neq v'(v'+1)^{d-1}$ since otherwise, then the equation $w(w+1)^{d-1}=v(v+1)^{d-1}$ has two roots satisfying $\Re(w)>-\rho$, which is a contradiction. 
	Therefore, if $u, u'\in \rootsL_z$ satisfy $u(u+1)^{d-1}=v(v+1)^{d-1}$ and $u'(u'+1)^{d-1}= v'(v'+1)^{d-1}$, then $u$ and $u'$ cannot be equal. This completes the proof. 
\end{proof}

Recall the polynomial 
\beqq
g(w,w')= \frac{w(w+1)^{d-1} -w'(w'+1)^{d-1}}{w-w'}.
\eeqq
In the previous section, for $v\in \rootsR_z$, we denoted by $U(v)$ the set of the roots of $g(w,v)$, and we set $
\SU_z=\cup_{v\in \rootsR_z} U(v)$. 
The above lemma implies that $\SU_z=\rootsL_z$ in the flat case. 
Replacing $\SU_z$ by $\rootsL_z$ in  Lemma \ref{lem:energy_stof}, we obtain the following formula.

\begin{lm}[Global energy function for flat initial condition] \label{lem:energyflatfor}
	We have 
	\begin{equation*} 
	\begin{split}
	\energy_{\mathrm{flat}}(z) 
	&= \frac{\prod_{v\in\rootsR_z} \left(\sqrt{v+1}\right)^{d}}{\prod_{v\in\rootsR_z} \sqrt{dv+1}} 
	\left[\frac{\prod_{v\in\rootsR_z} \prod_{u\in\rootsL_z} \sqrt{v-u}}{\prod_{v\in\rootsR_z} \left(\sqrt{v+1}\right)^{(d-1)N} \prod_{u\in\rootsL_z} \left(\sqrt{-u}\right)^N} \right] .
	\end{split}
	\end{equation*}
	Here the notation $\sqrt{w}$ here denotes the standard square root function with cut $\realR_{\le 0}$. 
\end{lm}

\bigskip

The formula \eqref{eq:aux_2019_08_21_07} of The characteristic function can also be simplified in the flat case. 

\begin{lm}[Characteristic function for flat initial condition] \label{lem:chfflatfor}
	For $v\in\rootsR_z$ and $u\in\rootsL_z$, 
	\begin{equation*} 
	\ich_{\mathrm{flat}} (v,u;z) 
	= \begin{dcases} 
	\frac{q'_{z,\RR}(v) u^{N} (u+1)^{d-1}} {q_{z,\RR}(u) v^N (v+1)^{d-1}} (u-v), \quad & \text{if $u(u+1)^{d-1} = v(v+1)^{d-1}$,}\\
	0, &\text{otherwise},
	\end{dcases}
	\end{equation*}
	where we remind that $q_{z,\RR}(w)=\prod_{v\in\rootsR_{z}} (w-v)$.
\end{lm}

\begin{proof}
	From the definition \eqref{eq:def_ich} of the characteristic function and the equation \eqref{eq:aux_087}, 
	we have the formula, which is same as the step-flat case \eqref{eq:aux_2019_08_21_07}, 
	\beqq
	\ich_{\mathrm{flat}} (v,u;z)  
	= \frac{(v+1)^{(N-1)(d-1)} g(v,v)}{(u+1)^{(N-1)(d-1)} g(u,v)} \prod_{v'\in \rootsR_z} \frac{g(u,v')}{g(v,v')} . 
	\eeqq
	Fix $v\in \rootsR_z$ and $u\in \rootsL_z$. 
	By Lemma \ref{lm:Berrsfl}, there is a unique $v_0\in \rootsR_z$ such that $u(u+1)^{d-1}=v_0(v_0+1)^{d-1}$, i.e. $g(u, v_0)=0$. 
	If $v_0\neq v$, then one of the terms in the product is zero, and hence, $\ich_{\mathrm{flat}} (v,u;z) =0$. 
	Note that $v_0\neq v$ is same condition as $u(u+1)^{d-1}\neq v(v+1)^{d-1}$. 
	On other hand, if $v_0=v$, then $u(u+1)^{d-1}=v(v+1)^{d-1}$. 
	In this case, we see from the formula of $g$ that 
	\beqq
	g(u, w')= g(v,w') \frac{v-w'}{u-w'}
	\eeqq
	for all $w'$. Inserting this formula in the product and noting the factors $g(v,v)$ and $g(u, v)$, we obtain
	\beqq
	\ich_{\mathrm{flat}} (v,u;z)  
	= \frac{(v+1)^{(N-1)(d-1)}}{(u+1)^{(N-1)(d-1)}} \prod_{v'\in \rootsR_z, v'\neq v} \frac{v-v'}{u-v'} 
	\eeqq
	in this case. The last product is equal to $\frac{q_{z, \RR}'(v) (u-v)}{q_{z, \RR}(u)}$ and we obtain the lemma. 
\end{proof}

Note that from the definitions, we have $w^N(w+1)^{L-N}-z^L= q_{z,\LL}(w) q_{z,\RR}(w)$. 
Taking the derivative and evaluating at $w=v\in \rootsR_z$, we find that $L(v+\rho) v^{N-1}(v+1)^{L-N-1}= q_{z,\LL}(v) q'_{z,\RR}(v)$.
The characteristic function can also be written as 
\begin{equation}\label{eq:aux_2018_3_16_02}
\ich_{\mathrm{flat}} (v,u;z) =  \frac{(v+1)^{L-N} u^N}{q_{z,\LL}(v) q_{z,\RR}(u)} 
\left(\frac{u+1}{v+1}\right)^{d-1} 
\frac{L(v+\rho)}{v(v+1)} (u-v) 
\end{equation}
when $u(u+1)^{d-1} = v(v+1)^{d-1}$.

\subsection{Condition (A)} \label{sec:flatAssmA}

We check that the condition (A) of Assumption \ref{def:asympstab}  holds. 
Fix $0<\epsilon<1/2$. Noting that $\rho=1/d$, the second equation of Lemma \ref{lem:asymofprod} (ii) with $\zeta=0$ implies that
\beqq
\prod_{v\in\rootsR_z} \sqrt{dv+1} = e^{\frac12 \mathrm{h} (0,\mathrm{z}) } ( 1+ O(L^{\epsilon-1/2}) ).
\eeqq
The other factors can be evaluated similarly using Lemma \ref{lem:asymofprod} and we find that 
\beqq
\energy_{\mathrm{flat}}(z) = e^{-\frac12 \mathrm{h} (0,\mathrm{z}) - B(z)} ( 1+ O(L^{\epsilon-1/2}) ).
\eeqq
From Lemma \ref{lem:hpropperty} (b), we have $\mathrm{h} (0,\mathrm{z}) = \frac12 \log (1-\mathrm{z})$. 
Hence, $\energy_{\mathrm{flat}}(z) = (1-z)^{-1/4} e^{- B(z)} ( 1+ O(L^{\epsilon-1/2}) )$, and the condition (A) is satisfied.

\subsection{Conditions (B) and (C)}\label{sec:flatAssmBC}

Fix $0<\epsilon<1/8$. 
Let $0<|\mathrm{z}|<1$ and $z^L=(-1)^N\rr^L\mathrm{z}$ as usual.
Consider $(v,u)\in \rootsR_z\times \rootsL_z$ satisfying $u(u+1)^{d-1} = v(v+1)^{d-1}$.
Set 
\beq \label{eq:uvxieatL}
\eta_L=\frac{L^{1/2}(v+\rho)}{\sqrt{\rho(1-\rho)}} \quad \text{and} \quad 
\xi_L=\frac{L^{1/2}(u+\rho)}{\sqrt{\rho(1-\rho)}}.
\eeq
Lemma \ref{lem:asymofprod} (ii) implies that 
\begin{equation*} 
\frac{q_{z,\LL}(v)}{(v+1)^{L-N}} = e^{\mathrm{h}(\eta_L,\mathrm{z})} (1+O(L^{4\epsilon-1/2})) 
\quad 
\text{and}
\quad  \frac{q_{z,\RR}(u)}{u^{N}}= e^{\mathrm{h}(\xi_L,\mathrm{z})} (1+O(L^{4\epsilon-1/2}))
\end{equation*}
for all sufficiently large $L$ if $|\eta_L|, |\xi_L| \le L^{\epsilon}$ and $\Re(\eta_L)\ge 0$, $\Re(\xi_L)\le 0$. 
Hence, \eqref{eq:aux_2018_3_16_02} is equal to 
\beqq
e^{-\mathrm{h}(\xi_L,\mathrm{z})-\mathrm{h}(\eta_L,\mathrm{z})}\eta_L(\eta_L-\xi_L) \left(1+O(L^{4\epsilon-1/2})\right).
\eeqq
Thus, we find that for $u$ and $v$ satisfy $u(u+1)^{d-1} = v(v+1)^{d-1}$, 
\beq \label{eq:ichflatintm}
\ich_{\mathrm{flat}} (v,u;z) =  e^{-\mathrm{h}(\xi,\mathrm{z})-\mathrm{h}(\eta,\mathrm{z})}\eta(\eta-\xi) + O(L^{4\epsilon-1/2})
\eeq	
for $(v,u)\in  \rootsR_z\times \rootsL_z$ satisfying $|v+\rho|= O(L^{\epsilon-1/2})$ and $|u+\rho|= O(L^{\epsilon-1/2})$. Here $\xi=\mathcal{M}_{L,\mathrm{left}} (u)\in \inodesL_{\mathrm{z}}$ and $\eta = \mathcal{M}_{L,\mathrm{right}}(v) \in \inodesR_{\mathrm{z}}$. Recall Lemma~\ref{lm:limiting_nodes}, both $|\xi_L-\xi|$ and $|\eta_L-\eta|$ are of order $O(L^{-1/2+3\epsilon}\log L)$. Since with \eqref{eq:uvxieatL}, the equation $u(u+1)^{d-1}=v(v+1)^{d-1}$ becomes 
\beq \label{eq:bbnnzz}
1- \frac{2 \rho(1-\rho) (\xi_L)^2}{\rho L} + O((\xi_L)^3 L^{-3/2}) 
=  1- \frac{2 \rho(1-\rho) (\eta_L)^2}{\rho L} + O((\eta_L)^3 L^{-3/2}) 
\eeq
which implies $\xi^2=\eta^2$ and further $\xi=-\eta$. 
Hence, \eqref{eq:ichflatintm} can be written as 
\beq \label{eq:ichflattlm22}
\ich_{\mathrm{flat}} (v,u;z) =  e^{-\mathrm{h}(\xi,\mathrm{z})-\mathrm{h}(\eta,\mathrm{z})}\eta(\eta-\xi)\delta_{\eta}(-\xi) + O(L^{4\epsilon-1/2}).
\eeq
On the other hand, suppose that $u(u+1)^{d-1}\neq v(v+1)^{d-1}$. 
Then, $\ich_{\mathrm{flat}} (u,v;z)=0$. By the computation \eqref{eq:bbnnzz}, we have $\xi\neq -\eta$ in this case, and hence the leading term of the right-hand side of \eqref{eq:ichflattlm22} is zero. Thus, \eqref{eq:ichflattlm22} also holds in this case. 
Therefore, we proved that Assumption (B) holds for the flat initial condition.

For Assumption (C), recalling the estimates in Lemma \ref{lem:asymofprod}, we observe that each factor except $\frac{L(v+\rho)}{v(v+1)}$ in the formula~\eqref{eq:aux_2018_3_16_02} is bounded by $O(1)$. Thus $|\ich_{\mathrm{flat}} (v,u;z)|\le O(L)$. This implies~\eqref{eq:easier_tail_estimates} and Assumption (C).

\section{Two random initial conditions}
\label{sec:random_IC}

We consider the PTASEP which starts with a random initial condition. 
Note that there are only a finitely many possible initial configurations since the system is periodic.
If the number of particles is random, then there are $2^L$ possible initial configurations. 
If the number of particles is fixed, then there are ${L\choose N}$ initial configurations. 
All random initial conditions are weighted combinations of these possible configurations, and the multi-point distribution of PTASEP with a random initial condition could be expressed as a weighted sum of the formula~\eqref{eq:multipoint_finite_time} in Theorem \ref{thm:Fredholm} over possible $Y$'s.
This weighted sum can be simplified for two specific random initial conditions. 
We discuss the finite-time formulas of the multi-point distributions for these two cases in this section. 
The limit theorem will be presented in the next section. 

\subsection{Uniform initial condition}

\begin{defn} \label{def:uniformrandom} 
	The $\PTASEP(L,N)$ is said to have the  \emph{uniform (random) initial condition} if 
	\begin{enumerate}[(i)]
		\item $(x_1(0),\cdots,x_N(0))$ is uniformly chosen from $\{(x_1,\cdots,x_N)\in\intZ^N: -L+1\le x_1<\cdots<x_N\le 0\}$, and 
		\item $x_{j+N}(0) =x_j(0) +L$ for all $j\in\intZ$.
	\end{enumerate}
\end{defn}

In other words, each of the ${L\choose N}$ possible initial configurations in one period can be chosen equally likely. 
We have the following finite-time formula for the multi-point distributions. The proof is in Subsection \ref{sec:pfric1}. 

\begin{thm}[Uniform initial condition] \label{thm:Fredholm_Uniform_IC}
	Consider the $\PTASEP(L,N)$ with the uniform initial condition. 
	Fix a positive integer $m$. Let  $(k_1,t_1),\cdots,(k_m,t_m)$ be $m$ distinct points in $\intZ\times[0,\infty)$ satisfying $0\le t_1\le \cdots\le t_m$. 
	Then, for arbitrary integers $a_1,\cdots,a_m$, 
	\beqq 
	\begin{split}
		&\prob \left( \bigcap_{\ell=1}^m \left\{ x_{k_\ell}(t_\ell) \ge a_\ell \right\}\right) = 
		\frac{(-1)^{N+1}}{{L\choose N}} 
		\oint\cdots\oint \frac{1}{z_1^L} \Delta_{\boldsymbol{k}} \left( \scrCstep \left(\boldsymbol{z};\boldsymbol{k}\right) \scrDstep \left(\boldsymbol{z};\boldsymbol{k}\right)\right) 
		\ddbar{z_1}{z_1} \ddbar{z_2}{z_2}\cdots \ddbar{z_m}{z_m},
	\end{split}
	\eeqq
	where $\Delta_{\boldsymbol{k}}$ is the difference operator defined by  
	\beqq
	\Delta_{\boldsymbol{k}}f(\boldsymbol{k}) = f(\boldsymbol{k}^+) -f(\boldsymbol{k}), 
	\text{where $\boldsymbol{k}=(k_1,\cdots,k_m)$ and $\boldsymbol{k}^+=(k_1+1,\cdots,k_m+1)$.}
	\eeqq
	The contours are nested circles satisfying $0<|z_m|<\cdots<|z_1|<\rr$ 
	and the functions $\scrCstep\left(\boldsymbol{z};\boldsymbol{k}\right)$ and $\scrDstep\left(\boldsymbol{z};\boldsymbol{k}\right)$ are defined in
	Subsection \ref{sec:formulaofstepfinite}; here we emphasize the dependence on $\boldsymbol{k}$. 
\end{thm}

When $m=1$, this formula was obtained in \cite[Theorem 3.1]{Liu16}.

\subsection{Partially uniform initial condition}

The second random initial condition the following. 

\begin{defn}[Partially uniform initial condition] \label{def:partialuniform}
	Consider the process $\PTASEP(L,N)$ with $N=N_1+N_2$ where $N_1\ge 0$ and $N_2\ge 1$. 
	Let $Y=(y_1,\cdots,y_{N_2})\in\intZ^{N_2}$ satisfy
	\begin{equation*}
	-L+N_1+1 \le y_{1}<y_2<\cdots<y_{N_2}=0.
	\end{equation*}
	We say that the PTASEP has a \emph{partially uniform initial condition} if
	\begin{enumerate}[(i)]
		\item $(x_{1}(0), \cdots, x_{N_1} (0))$ is uniformly chosen from $\{(x_1,\cdots,x_{N_1})\in\intZ^{N_1}: -L+1\le x_1<\cdots<x_{N_1}\le y_1-1\}$,
		\item $(x_{N_1+1}(0),\cdots,x_{N}(0))=Y$, and 
		\item $x_{j+N}(0) =x_j(0) +L$ for all $j\in\intZ$.
	\end{enumerate}
\end{defn} 

Hence, a part of the period is uniformly random but the other part is deterministic. 
In the above, we impose that the ``last site" in the period $\{-L+1, \cdots, 0\}$ is occupied, i.e., $y_{N_2}=0$. This condition is put in place in order to make the final formula simple. 
However, this condition is not restrictive since if the site $0$ is not occupied, then we may change holes to particles and study the PTASEP with $L-N$ particles where the particles are moving to the left direction.

The multi-point distribution involves $Y$. 
We use a new symmetric function which generalizes $\gftn_{\lambda}(W)$ in~\eqref{eq:def_gftn}. 

\begin{defn}[Another symmetric function]
	Let $N_2$ and $N$ be two integers such that $1\le N_2\le N$. 
	For $\lambda=(\lambda_1,\cdots,\lambda_{N_2})\in \intZ^{N_2}$ satisfying $\lambda_1\ge\cdots\ge\lambda_{N_2}$, we define the symmetric function 
	\beq \label{eq:newsymmf} \begin{split}
		\gftnv_\lambda(W) & =\gftnv_\lambda(W;N_2)\\
		&
		=\frac{\det\left[1_{j\le N_2}w_i^{N-j}(w_i+1)^{\lambda_j}+1_{j>N_2}w_i^{N-j-1}(w_i+1)^{\lambda_{N_2}+1}\right]_{1\le i,j\le N}} {\det\left[w_i^{N-j}\right]_{1\le i,j\le N}},
	\end{split} \eeq 
	where $W=(w_1,\cdots,w_N)$.
\end{defn}

Since $\gftnv_\lambda(W)$ is a symmetric function of $(w_1, \cdots, w_N)$, we can regard $W$ as a set $W=\{w_1, \cdots, w_N\}$ instead of a vector. We use $W$ for either a set or a vector, and interchange the meaning freely. 
The above function is equal to $\gftn_\lambda(W)$ when $N_2=N$. 
Note that when $N_2<N$, it has a pole at $w_i=0$.

Compare the next definition with Definition~\ref{defn:energy_ich}. 

\begin{defn}
	For $Y=(y_1,\cdots,y_{N_2})\in\intZ^{N_2}$ satisfying $-L+N_1+1\le y_1<y_2<\cdots<y_{N_2}=0$,
	set \begin{equation*}
	\lambda(Y) = (y_{N_2}, y_{N_2-1}+1,\cdots, y_{1}+N_2-1).
	\end{equation*}
	Assume that $z\in\complexC$ satisfies $|z|<\rr$. Define
	\beqq 
	\energyv_Y(z) =\gftnv_{\lambda(Y)}(\rootsR_z).
	\eeqq
	When $\gftnv_{\lambda(Y)}(\rootsR_z)\ne 0$, define
	\begin{equation} \label{eq:def_ichv}
	\ichv_Y(v,u;z) = \frac{\gftnv_{\lambda(Y)}(\rootsR_z\cup \{u\} \setminus \{v\})}{\gftnv_{\lambda(Y)}(\rootsR_z)}
	\quad \text{for $v\in\rootsR_z$ and $u\in\rootsL_z$}.
	\end{equation}
\end{defn}

We have the following finite-time formula. Its proof is given in Subsection \ref{sec:ricfl2}. 
We remind the restrictions $N_2\ge 1$ and $y_{N_2}=0$ in our definition of partial uniform initial condition. Thus this result does not cover the statement of Theorem~\ref{thm:Fredholm_Uniform_IC}.

\begin{thm} \label{thm:Fredholm_RandomIC}
	(Partially uniform initial condition) 	
	Consider the process $\PTASEP(L,N)$ with $N=N_1+N_2$ and the partial uniform initial condition described above. 
	Fix a positive integer $m$. 
	Let $(k_1,t_1),\cdots,(k_m,t_m)$ be $m$ distinct points in $\intZ\times[0,\infty)$ satisfying $0\le t_1\le \cdots\le t_m$. 
	Then, for arbitrary integers $a_1,\cdots,a_m$, 
	\beqq
	\prob\left( \bigcap_{\ell=1}^m \left\{x_{k_\ell}(t_\ell) \ge a_\ell \right\}\right) 
	= \frac{1}{{y_1+L-1\choose N_1}} \oint\cdots\oint \tilde\scrC_Y(\bz) \tilde\scrD_Y(\bz) \ddbar{z_1}{z_1}\cdots \ddbar{z_m}{z_m},
	\eeqq
	where the contours are nested circles satisfying $0<|z_m|<\cdots<|z_1|<\rr$. 
	The functions $\tilde\scrC_Y(\bz)$ and $\tilde\scrD_Y(\bz)$ are defined by the same formulas as $\scrC_Y(\bz)$ 
	in~\eqref{eq:def_C0} and $\scrD_Y(\bz)$ in~\eqref{eq:def_D0}, respectively, 
	except that the functions $\energy_Y(\bz)$ and $\ich_Y(v,u;z)$ are replaced by $\energyv_Y(\bz)$ and $\ichv_Y(v,u;z)$.
\end{thm}

\subsection{Proof of Theorem~\ref{thm:Fredholm_Uniform_IC}}\label{sec:pfric1}

Define the set
\beqq
\mathcal{Y}_N(L):=\{(y_1,\cdots,y_N):-L+1\le y_1<\cdots<y_N\le 0\}.
\eeqq
Since the initial condition $Y$ is uniformly chosen from $\mathcal{Y}_N(L)$, we have 
\beqq
\prob \left( \bigcap_{\ell=1}^m \{ x_{k_\ell}(t_\ell) \ge a_\ell \} \right) 
= \frac{1}{{L\choose N}} \sum_{Y\in\mathcal{Y}_N(L)}  \prob_Y \left( \bigcap_{\ell=1}^m \{ x_{k_\ell}(t_\ell) \ge a_\ell \} \right) .
\eeqq
We use the Toeplitz-like formula, Theorem \ref{thm:Toeplitzform}, and take the sum over $Y$. 
The formula involves $\caC(\bz;\boldsymbol{k})$ and $\caD_Y (\bz;\boldsymbol{k})$. Since the first term does not depend on the initial condtion, we have 
\begin{equation}
\label{eq:aux_015}
\begin{split}
&\prob \left( \bigcap_{\ell=1}^m \{ x_{k_\ell}(t_\ell) \ge a_\ell \} \right) 
= \frac{1}{{L\choose N}}\oint \cdots \oint  \caC(\bz;\boldsymbol{k}) \bigg( \sum_{Y\in\mathcal{Y}_N(L)}\caD_Y (\bz;\boldsymbol{k}) \bigg)  \ddbar{z_m}{z_m} \cdots \ddbar{z_1}{z_1} .
\end{split}
\end{equation}
The next lemma simplifies the sum.

\begin{lm} \label{lem:unfrndm}
	We have 
	\begin{equation*}
	\sum_{Y\in\mathcal{Y}_N(L)} \caD_Y(\bz;\boldsymbol{k}) = \caD_{\rm{step}}(\bz;\boldsymbol{k}^+) - (-1)^{N-1}z_1^{-L} \caD_{\rm{step}}(\bz;\boldsymbol{k}),
	\end{equation*}
	where $\boldsymbol{k}^+=(k_1+1,\cdots,k_m+1)$ and $\caD_{\rm{step}}$ is $\caD_{Y}$ with the step initial condition,  $Y=(-N+1, \cdots, -1,0)$. 
\end{lm}

\begin{proof}
	The identity was proved in \cite{Liu16} when $m=1$. 
	The proof extends to general $m$ easily. 
	Applying the Cauchy-Binet identity $m$ times, the formula of $\caD_Y(\bz;\boldsymbol{k})$ given in \eqref{eq:aux_2018_04_06_01} becomes 
	\begin{equation} \label{eq:aux_014}
	\begin{split}
	\caD_Y(\bz;\boldsymbol{k})  = \frac{1}{(N!)^{m}}\sum_{\substack{w_i^{(\ell )}\in \lemS_\ell \\[3pt] 1\le i\le N\\[3pt] 1\le \ell \le m}}
	& \det\left[(w_i^{(1)})^j (w_i^{(1)}+1)^{y_j-j}\right]  \prod_{\ell=1}^{m-1} \det\left[\frac{1}{w_i^{(\ell)}-w_{j}^{(\ell-1)}} \right] \\
	&\qquad  \times  \det\left[(w_i^{(m)})^{-j}\right]  \prod_{\substack{1\le i\le N\\ 1\le \ell \le m}}\mathcal{G}_\ell(w_i^{(\ell)}) ,
	\end{split}
	\end{equation}
	where every determinant is a determinant of an $N\times N$ matrix. 
	Note that $Y$ appears only in the first determinant. The sum over $Y$ of this determinant was computed in Lemma 3.1 of \cite{Liu16}: 
	\begin{equation} \label{eq:aux_018}
	\begin{split}
	&\sum_{Y\in\mathcal{Y}_N(L)} \det\left[(w_i^{(1)})^j (w_i^{(1)}+1)^{y_j-j}\right] \\
	& = \det\left[(w_i^{(1)})^{j-1}(w_i^{(1)}+1)^{-N+1}\right] 
	-(-1)^{N-1}z_1^{-L}\det\left[(w_i^{(1)})^j (w_i^{(1)} +1)^{-N}\right]  \\
	&  = \det\left[(w_i^{(1)})^{j}(w_i^{(1)}+1)^{-N}\right] \prod_{i=1}^N \frac{w_i^{(1)}+1}{w_i^{(1)}}
	-(-1)^{N-1}z_1^{-L}\det\left[(w_i^{(1)})^j (w_i^{(1)} +1)^{-N}\right]  .
	\end{split}
	\end{equation}
	The last two determinants are the same as the first determinant in \eqref{eq:aux_014} when $y_j=j-N$, i.e. when $Y$ is the step initial condition. 
	Recalling the formula~\eqref{eq:aux_2018_04_08_04},
	\beqq
	\mathcal{G}_\ell(w) := \frac{w(w+1)}{L(w+\rho)} \frac{w^{-k_\ell}(w+1)^{-a_\ell+k_\ell} e^{t_\ell w}}{ w^{-k_{\ell-1}} (w+1)^{-a_{\ell-1}+k_{\ell-1}} e^{t_{\ell-1}w}},
	\eeqq
	with $k_0=a_0=t_0=0$, the extra factor $\prod_{i=1}^N \frac{w_i^{(1)}+1}{w_i^{(1)}}$ times the product of $\mathcal{G}_\ell(w_i^{(\ell)}) $ is equal to the product of $\mathcal{G}_\ell(w_i^{(\ell)}) $ with $k_\ell$ changed to $k_\ell+1$. 
	Hence, we obtain the result. 
\end{proof}

From the formula~\eqref{eq:def_caC} of $\caC(\bz;\boldsymbol{k})$, it is easy to see that 
\beqq
\caC(\bz;\boldsymbol{k}^+) = (-1)^{N-1} z_1^{L} \caC(\bz;\boldsymbol{k}).
\eeqq
Thus,
\beqq \begin{split}
	&\caC(\bz;\boldsymbol{k}) \sum_{Y\in\mathcal{Y}_N(L)}\caD_Y (\bz;\boldsymbol{k})  \\
	&= \caC(\bz;\boldsymbol{k}) \left(  \caD_{\rm{step}}(\bz;\boldsymbol{k}^+) - (-1)^{N-1}z_1^{-L} \caD_{\rm{step}}(\bz;\boldsymbol{k}) \right) \\
	&  =  (-1)^{N-1}z_1^{-L} \left( \caC(\bz;\boldsymbol{k}^+) \caD_{\rm{step}}(\bz;\boldsymbol{k}^+) - \caC(\bz;\boldsymbol{k}) \caD_{\rm{step}}(\bz;\boldsymbol{k}) \right) .  \\
\end{split} \eeqq
We thus obtain the result after inserting this formula in \eqref{eq:aux_015} and applying Proposition~\ref{prop:CDlocz}.

\subsection{Proof of Theorem~\ref{thm:Fredholm_RandomIC}}\label{sec:ricfl2}

We denote the initial condition by $ \tilde Y \sqcup Y$ where $Y=(y_1,\cdots,y_{N_2})$ is the deterministic part of the initial condition and $\tilde Y=(\tilde y_1,\cdots,\tilde y_{N_1})$ is the random part. 
By the assumption,
\beqq
-L+N_1+1\le y_1<\cdots<y_{N_2}=0
\eeqq 
and $\tilde Y$ is uniformly chosen from the set
\begin{equation*}
\mathcal{\tilde Y}_{N_1}(L;y_1):=\{(\tilde y_1,\cdots,\tilde y_{N_1})\in\intZ^{N_1} :  -L+1\le \tilde y_1<\cdots<\tilde y_{N_1}\le y_1-1\}.
\end{equation*}
Using the Toeplitz-like formula, Theorem \ref{thm:Toeplitzform}, and taking the sum over $\tilde Y$, 
\begin{equation*}
\prob \left( \bigcap_{\ell=1}^m \{ x_{k_\ell}(t_\ell) \ge a_\ell \} \right) 
= \frac{1}{{y_1+L-1\choose N_1}}\oint \cdots \oint  \caC(\bz) \tilde \caD_Y(\bz)  \ddbar{z_m}{z_m} \cdots \ddbar{z_1}{z_1},
\end{equation*}
where
\beqq
\tilde \caD_Y(\bz) = \sum_{\tilde Y \in\mathcal{\tilde Y}_{N_1}(L;y_1)}\caD_{\tilde Y \sqcup Y} (\bz) .
\eeqq

\begin{lm}
	We have 
	\beq  \label{eq:tildDD}
	\tilde \caD_Y(\bz)
	= \det \left[  \sum_{\substack{w_1\in\roots_{z_1}\\ \cdots \\ w_m\in\roots_{z_m}}}
	\frac{\tilde \lemf_i(w_1)\lemg_j(w_m)}
	{\prod_{\ell=2}^{m} (w_\ell -w_{\ell-1})}
	\prod_{\ell=1}^{m} \lemh_\ell(w_\ell)      
	\right]_{i,j=1}^N,
	\eeq
	where
	\begin{equation*}
	\tilde \lemf_i(w)= w^{N-i}(w+1)^{y_{N_2+1-i} +i-1}1_{i\le N_2} + w^{N-i-1} (w+1)^{y_1+N_2}1_{i>N_2} .
	\end{equation*}
	The functions $\lemg_j$ and $\lemh_\ell$ are the same as~\eqref{eq:aux_020} and~\eqref{eq:def_lemh}. This determinant is same as $\caD_Y(\bz)$ in~\eqref{eq:aux_03_29_03} with $\lemf_i$ changed to $\tilde\lemf_i$. 
\end{lm}

\begin{proof}
	The proof is same as that of Lemma~\ref{lem:unfrndm} except that instead of the identity~\eqref{eq:aux_018}, we use Lemma~\ref{lem:idpfru} below. 
\end{proof}

\begin{lm}\label{lem:idpfru}
	For $w_i^{(1)}\in\roots_{z_1}$, $i=1,\cdots,N$, we have
	\begin{equation}
	\label{eq:aux_019}
	\begin{split}
	&\sum_{Y_1\in\mathcal{\tilde Y}_{N_1}(L;y_1)} \det\left[ (w_i^{(1)})^j (w_i^{(1)}+1)^{\tilde y_j-j} 1_{j\le N_1} + (w_i^{(1)})^j (w_i^{(1)}+1)^{y_{j-N_1}-j} 1_{j> N_1} \right]_{i,j=1}^N\\
	& = \det\left[(w_i^{(1)})^{j-1} (w_i^{(1)}+1)^{y_1-N_1} 1_{j\le N_1} + (w_i^{(1)})^j (w_i^{(1)}+1)^{y_{j-N_1}-j} 1_{j> N_1} \right]_{i,j=1}^N.
	\end{split}
	\end{equation}
\end{lm}

\begin{proof}
	Let  
	\begin{equation*}
	A(i,j)=(w_i^{(1)})^j (w_i^{(1)}+1)^{\tilde y_j-j} 1_{j\le N_1} + (w_i^{(1)})^j (w_i^{(1)}+1)^{y_{j-N_1}-j} 1_{j> N_1}
	\end{equation*}
	and
	\begin{equation*}
	B(i,j)=(w_i^{(1)})^{j-1} (w_i^{(1)}+1)^{y_1-N_1} 1_{j\le N_1} + (w_i^{(1)})^j (w_i^{(1)}+1)^{y_{j-N_1}-j} 1_{j> N_1}
	\end{equation*}
	be the entries of the matrices appearing in the lemma. 
	Note that for $j\le N_1$, the entry $A(i,j)$ depends on $\tilde y_j$ but not on $\tilde y_k$ with $k\ne j$. 
	We evaluate the left-hand side of \eqref{eq:aux_019} by the following repeated sums: 
	\begin{equation*}
	\sum_{\tilde y_1=-L+1}^{y_1-N_1} \sum_{\tilde y_2= \tilde y_1+1}^{y_1-N_1+1} \cdots \sum_{\tilde y_{N_1}=\tilde y_{N_1-1}+1}^{y_1-1} \det[A(i,j)]_{i,j=1}^N.
	\end{equation*}
	Since
	\begin{equation*}
	\sum_{\tilde y_{j-1}+1\le \tilde y_j \le y_1-N_1+j-1} w^j(w+1)^{\tilde y_j-j} =  w^{j-1} (w+1)^{y_1-N_1} -w^{j-1}(w+1)^{\tilde y_{j-1}-(j-1)},
	\end{equation*}
	we find that
	\begin{equation*}
	\sum_{\tilde y_{j-1}+1\le \tilde y_j \le y_1-N_1+j-1} A(i,j)= B(i,j)-A(i,j-1)
	\end{equation*}
	for $2\le j\le N_1$. 
	On the other hand, since 
	\begin{equation*}
	\sum_{-L+1\le \tilde y_1 \le y_1-N_1} w(w+1)^{\tilde y_1-1} =  (w+1)^{y_1-N_1} -(w+1)^{-L},
	\end{equation*}
	we have 
	\begin{equation*}
	\sum_{-L+1\le \tilde y_1 \le y_1-N_1} A(i,1)= B(i,1)- z^{-L}A(i,N).
	\end{equation*}
	The lemma follows from these formulas by taking the sum over $\tilde y_j$ in the $j$th column and using the linearity of determinant on the columns. 
\end{proof}

The formula \eqref{eq:tildDD} of $\tilde \caD_Y(\bz)$ has the same structure as  $\caD_Y(\bz)$ in~\eqref{eq:aux_03_29_03}, except that $\lemf_i$ needs to be replaced by $\tilde\lemf_i$.
This change does not alter the proof of Proposition \ref{prop:CDlocz} and we obtain
\begin{equation*} 
\caC(\bz)\tilde \caD_Y(\bz) =\tilde \scrC_Y(\bz) \tilde \scrD_Y(\bz),
\end{equation*}
where $\tilde \scrC_Y(\bz)$ and $ \tilde \scrD_Y(\bz)$ are same as $\scrC_Y(\bz)$ and $\scrD_Y(\bz)$ except for the functions $\energy_Y(\bz)$ and $\ich_Y(u,v;z)$ replaced by $\energyv_Y(\bz)$ and $\ichv_Y(u,v;z)$, respectively. 
This implies Theorem~\ref{thm:Fredholm_RandomIC}.

\section{Limit theorems for random initial conditions} \label{sec:randomIClimit}

We discuss the large time limit of the multi-point distribution for two random initial conditions. 
We use the same notations as Theorem~\ref{thm:main}.
Recall that $\mb e_1=(1,0)$ and $\mb e_c= (1-2\rho,1)$ are vectors in the space-time coordinates $\R\times \R_{\ge 0}$ and the set $\mathbf{R}_L$ defined by~\eqref{eq:thesetRL}. 

\subsection{Uniform initial condition}

Consider the uniform (random) initial condition in Definition \ref{def:uniformrandom}.

\begin{thm}[Large time limit for uniform initial condition)] \label{thm:Uniform}
	Consider a sequence of $\PTASEP(L,N)$ with the uniform initial condition. 
	Assume that $\rho=N/L$ stays in a compact subset of $(0,1)$. 
	Fix a positive integer $m$. 
	Let $\mathrm{p}_j =(\gamma_j,\tau_j)$ be $m$ points in the region $\mathrm{R}= [0,1] \times \realR_{>0}$
	satisfying $\tau_1<\tau_2<\cdots<\tau_m$. 
	Then, for $\mb p_j = s_j\mb e_1 +t_j\mb e_c\in \mathbf{R}_L$ given by 
	\begin{equation*}
	s_j =\gamma_j L,\qquad t_j= \tau_j \frac{L^{3/2}}{\sqrt{\rho(1-\rho)}},
	\end{equation*}
	and for every $\mathrm{x}_1,\cdots\mathrm{x}_m \in \realR$, 
	we have 
	\begin{equation} \label{eq:uniform_limit}
	\begin{split}
	&\lim_{L\to\infty} \prob \left( \cap_{j=1}^m \left\{ \frac{\height(\mb p_j)  -(1-2\rho) s_j -(1-2\rho+2\rho^2) t_j}{-2\rho^{1/2}(1-\rho)^{1/2} L^{1/2}} \le x_j	\right\}\right) = \FU (\mathrm{x}_1,\cdots, \mathrm{x}_m; \mathrm{p}_1,\cdots, \mathrm{p}_m),
	\end{split}
	\end{equation}
	where 
	\begin{equation*} 
	\begin{split}
	&\FU (\mathrm{x}_1,\cdots, \mathrm{x}_m; \mathrm{p}_1,\cdots, \mathrm{p}_m)
	= -\sqrt{2\pi} \oint \cdots \oint  
	\mathrm{z}_1^{-1}
	\frac{\dd}{\dd \mathrm{\mathbf{x}}} \left( \limCstep (\limz; \mathbf{x}) \limDstep (\limz; \mathbf{x}) \right) 
	\ddbar{\mathrm{z}_1}{\mathrm{z}_1}\cdots \ddbar{\mathrm{z}_m}{\mathrm{z}_m} .
	\end{split} \end{equation*}
	The contours are nested circles satisfying $0<|\mathrm{z}_m|<\cdots < |\mathrm{z}_1|<1$ and the derivative denotes the directional derivative 
	\beqq
	\frac{ \dd }{\dd \mathbf{x} } (f(\mathbf{x}))= \lim_{\epsilon\to 0} \frac{f(\mathbf{x} +\epsilon) - f(\mathbf{x})}{\epsilon}, 
	\quad \text{where $\mathbf{x}+\epsilon = (\mathrm{x}_1+\epsilon,\cdots,\mathrm{x}_m +\epsilon)$.}
	\eeqq
	The convergence of~\eqref{eq:uniform_limit} is locally uniform in $\mathrm{x}_j, \tau_j$, and $\gamma_j$. If $\tau_i=\tau_{i+1}$ for some $i\ge 2$, then~\eqref{eq:uniform_limit} still holds if we assume that $\mathrm{x}_i < \mathrm{x}_{i+1}$.
\end{thm}

The above result when $m=1$ was obtained previously in Theorem 1.1 in \cite{Liu16}. We prove the general $m$ case in Subsection \ref{sec:proof_asympt_uniform} below.

\subsection{Uniform-step initial condition}

\begin{defn}[Uniform-step initial condition] \label{def:uniformstep}
	Consider the $\PTASEP(L,N)$ with $N=N_1+N_2$ where $N_2\ge 1$. 
	We say that the PTASEP has a \emph{uniform-step initial condition} if
	\begin{enumerate}[(i)]
		\item $(x_{1}(0), \cdots, x_{N_1} (0))$ is uniformly chosen from $\{(x_1,\cdots,x_{N_1})\in\intZ^{N_1}: -L+1\le x_1<\cdots<x_{N_1}\le -N_2\}$;
		\item $(x_{N_1+1}(0),\cdots, x_{N-1}(0), x_{N}(0))=(-N_2+1, \cdots, -1, 0)$;
		\item $x_{j+N}(0) =x_j(0) +L$ for all $j\in\intZ$.
	\end{enumerate}
\end{defn}

The uniform-step initial condition is a special case of partially uniform initial condition discussed in Section~\ref{sec:random_IC} when $y_i=i-N_2$. 
We have the following limit theorem. See Subsection~\ref{sec:proof_asympt_uniform_step} for the proof.

\begin{thm}[Large time limit for uniform-step initial condition]	 \label{thm:step_uniform} 
	Consider a sequence of $\PTASEP(L,N)$ with the uniform-step initial condition 
	where $N=N_1+N_2$ and $\rho=N/L$ stays in a compact subset of $(0,1)$. 
	Assume that 
	\begin{equation} \label{eq:def_N2}
	N_2=\alpha \rho^{1/2}(1-\rho)^{-1/2} L^{1/2}
	\end{equation}
	for a positive constant $\alpha>0$. 
	With the same conditions and notations as Theorem \ref{thm:Uniform}, we have 
	\begin{equation} \label{eq:step_uniform_limit}
	\begin{split}
	&\lim_{L\to\infty} \prob \left( \bigcap_{j=1}^m \left\{ \frac{\height(\mb p_j)  -(1-2\rho) s_j -(1-2\rho+2\rho^2) t_j}{-2\rho^{1/2}(1-\rho)^{1/2} L^{1/2}} \le x_j	\right\}\right) 
	= \FUS (\mathrm{x}_1,\cdots, \mathrm{x}_m; \mathrm{p}_1,\cdots, \mathrm{p}_m) .
	\end{split}
	\end{equation}
	The limiting function is given by  
	\begin{equation*} \begin{split}
	\FUS (\mathrm{x}_1,\cdots, \mathrm{x}_m; \mathrm{p}_1,\cdots, \mathrm{p}_m)
	= \oint \cdots\oint 
	\mathrm{C}_{\us}^{(\alpha)} (\mathbf{z})  \mathrm{D}_{\us}^{(\alpha)}  (\mathbf{z})
	\ddbar{\mathrm{z}_m}{\mathrm{z}_m}
	\cdots \ddbar{\mathrm{z}_1}{\mathrm{z}_1},
	\end{split} \end{equation*}
	where the contours are nested circles satisfying $0<|\mathrm{z}_m|<\cdots < |\mathrm{z}_1|<1$.
	The functions $\mathrm{C}_{\us}^{(\alpha)} (\mathbf{z})$ and $\mathrm{}D_{\us}^{(\alpha)} $ are same as $\mathrm{C}_{\mathrm{ic}}(\mathbf{z})$ in~\eqref{eq:def_rmC} and $\mathrm{D}_{\mathrm{ic}}(\mathbf{z})$ in~\eqref{eq:def_rmD},  respectively, with $E_{\mathrm{ic}}$ replaced by
	\begin{equation*}
	E_{\us}^{(\alpha)} (\mathrm{z}_1)
	=\int_{\ii\realR} e^{-\mathrm{h} (\zeta+\alpha,\mathrm{z}_1)+\zeta^2/2} \frac{\dd \zeta}{\ii \sqrt{2\pi} }
	\end{equation*}
	and $\chi_{\mathrm{ic}} (\eta,\xi;\mathrm{z}_1)$ replaced by 
	\begin{equation*}
	\chi_{\us}^{(\alpha)}  (\eta,\xi;\mathrm{z}_1) 
	=\frac{1}{E_{\us}^{(\alpha)} (\mathrm{z_1})} \int_{c-\ii\infty}^{c+\ii\infty} e^{-\mathrm{h} (\zeta+\alpha,\mathrm{z}_1)+\zeta^2/2} \frac{\zeta+\alpha-\xi}{\zeta+\alpha-\eta} \frac{\dd \zeta}{\ii \sqrt{2\pi} } 
	\end{equation*}
	for $\eta\in\inodesR_{\mathrm{z}_1}$ and $\xi\in\inodesL_{\mathrm{z}_1}$,
	where $c$ in the integration contour is any real constant satisfying $c+\alpha>\Re(\eta)>0$. 
	The convergence of~\eqref{eq:step_uniform_limit} is locally uniform in $\mathrm{x}_j, \tau_j$, and $\gamma_j$. If $\tau_i=\tau_{i+1}$ for some $i\ge 2$, then~\eqref{eq:uniform_limit} still holds if we assume that $\mathrm{x}_i < \mathrm{x}_{i+1}$.
\end{thm}

By Lemma \ref{lem:hpropperty} (f), the function $\mathrm{h}(\zeta,\mathrm{z}_1)= O(\zeta^{-1})$ as $\zeta\to\infty$ in $\{\zeta\in\complexC: |\Re(\zeta)|\ge c'\}$ for any given  constant $c'>0$. 
Thus both $E_{\us}^{(\alpha)} (\mathrm{z}_1)$ and $\chi_{\us}^{(\alpha)}  (\eta,\xi;\mathrm{z}_1) $ are well defined. 

The parameter $\alpha$ in \eqref{eq:def_N2} measures the relative amount of the uniform part and the step part of the uniform-step initial condition. 
Hence, we expect that the limiting distribution converges to that of the step initial condition as $\alpha\to \infty$ and that of the uniform initial condition as $\alpha\to 0$. 
As $\alpha\to \infty$, we have $\mathrm{h} (\zeta+\alpha,\mathrm{z}_1)\to 0$ by Lemma \ref{lem:hpropperty} (f) again. 
Since $\frac{\zeta+\alpha-\xi}{\zeta+\alpha-\eta} \to 1$, we find that both  $E_{\us}^{(\alpha)}(\mathrm{z}_1)$ and  $\chi_{\us}^{(\alpha)}(\xi,\eta;\mathrm{z}_1)$ both converge to $1$ as $\alpha\to\infty$. 
Therefore, $\FUS (\mathrm{x}_1,\cdots, \mathrm{x}_m; \mathrm{p}_1,\cdots, \mathrm{p}_m) \to \FS (\mathrm{x}_1,\cdots, \mathrm{x}_m; \mathrm{p}_1,\cdots, \mathrm{p}_m)$ as $\alpha\to \infty$. 
On the other hand, we expect that $\FUS (\mathrm{x}_1,\cdots, \mathrm{x}_m; \mathrm{p}_1,\cdots, \mathrm{p}_m) \to \FU (\mathrm{x}_1,\cdots, \mathrm{x}_m; \mathrm{p}_1,\cdots, \mathrm{p}_m)$ as $\alpha\to 0$, but it is not clear how to see this from the formula directly yet.

\subsection{Proof of Theorem~\ref{thm:Uniform}}
\label{sec:proof_asympt_uniform}

Using the formula, Theorem~\ref{thm:Fredholm_Uniform_IC}, for the finite-time multi-point distributions and translating the particle description for the PTASEP to the height function description (see, for example, (3.15) in \cite{Liu16}), 
it is enough to show the pointwise convergence of  
\begin{equation} \label{eq:aux_031}
\begin{split}
&\frac{(-1)^{N+1}}{{L\choose N}} \frac1{z_1^L}  \Delta_{\boldsymbol{k}} 
\left( \scrCstep \left(\boldsymbol{z};\boldsymbol{k}\right) \scrDstep \left(\boldsymbol{z};\boldsymbol{k}\right)\right) 
\to - \frac{\sqrt{2\pi}}{\mathrm{z}_1} \frac{\dd}{\dd \mathrm{\mathbf{x}}} 
\left( \limCstep (\limz; \mathbf{x}) \limDstep (\limz; \mathbf{x}) \right) 
\end{split}
\end{equation}
and also to show that the integral involving these terms converges 
when the parameters $\boldsymbol{k},\boldsymbol{a},\boldsymbol{t}$ are scaled as in the proof of Theorem~\ref{thm:main}.
It was shown in \cite{Baik-Liu17} that 
\begin{equation*}
\scrCstep\left(\boldsymbol{z};\boldsymbol{k}\right) \to  \limCstep (\mathbf{z}) \quad \text{and} \quad
\scrDstep \left(\boldsymbol{z};\boldsymbol{k}\right) \to \limDstep (\mathbf{z}). 
\end{equation*}
as $L\to\infty$. 
On the other hand, from the scaling~\eqref{eq:aux_028} on the parameters $\boldsymbol{\gamma}$  and $\boldsymbol{\tau}$ fixed, we have
\begin{equation*}
\Delta_{\boldsymbol{k}}\approx \rho^{-1/2}(1-\rho)^{-1/2} L^{-1/2}\frac{\dd}{\dd\mathbf{x}},
\end{equation*}
and by the Stirling's formula, 
\begin{equation*}
\frac{(-1)^{N}}{{L\choose N}} \frac{1}{z_1^L} \approx \frac{\sqrt{2\pi}}{\mathrm{z}_1} \rho^{1/2}(1-\rho)^{1/2} L^{1/2}.
\end{equation*}
These were also shown in (4.80) of \cite{Liu16}. 
These considerations imply the pointwise convergence of \eqref{eq:aux_031}. 
To complete the proof, we need a uniform upper bound on the left-hand side of \eqref{eq:aux_031} in order to use the dominated convergence theorem to prove the convergence of the integral. 
Such estimates were obtained in \cite{Liu16} when $m=1$. 
The case when $m\ge 2$ is similar but tedious. 
Since this computation does not add any new insight, we omit the details.

\subsection{Proof of Theorem~\ref{thm:step_uniform}}
\label{sec:proof_asympt_uniform_step}

The proof is similar to the proof of Theorem~\ref{thm:main}. 
Considering the difference of the finite-time formulas, Theorem \ref{thm:Fredholm_RandomIC} and Theorem \ref{thm:Fredholm}, it is enough to prove the following lemma. Compare this lemma with Assumption \ref{def:asympstab}.

\begin{lm} \label{lm:asymptofsf}
	Assume the same conditions on $L, N_1, N_2$ as Theorem \ref{thm:step_uniform}. Set $z^L= (-1)^N \rr^L \lz$. 
	Then, there are $0<r_2<r_1<1$ such that the following hold. 
	\begin{enumerate}[(i)]
		\item For every $\epsilon\in (0, 1/2)$, we have 
		\begin{equation*}
		\frac{1}{{L-N_2\choose N_1}} \tilde\energy_{\us}(z)
		= E_{\us}(\mathrm{z}) (1+ O(L^{\epsilon-1/2}))
		\end{equation*}
		uniformly for $|\lz|<r_1$ for all sufficiently large $L$.
		
		\item For any fixed $\epsilon\in (0,1/8)$, we have
		\begin{equation*}
		\tilde\ich_{\us} (v,u; z) = \chi_{\us}(\eta,\xi;\mathrm{z}) (1+O(L^{4\epsilon-1/2}))
		\end{equation*}
		uniform for $ r_1<|\mathrm{z}|<  r_2$,  $v \in \rootsR_z^{(\epsilon)}$ and $u\in \rootsL_z^{(\epsilon)}$ for all sufficiently large $L$, where we use the notations from Lemma \ref{lm:limiting_nodes} and we set 
		$\eta = \mathcal{M}_{L,\mathrm{right}} (v)$ and $\xi = \mathcal{M}_{L,\mathrm{left}} (u)$.
		
		\item 
		There are constants $\epsilon\in (0,1/8)$ and $\epsilon',C>0$ such that for all sufficiently large $L$ and for all $r_2<|\mathrm{z}|<r_1$, 
		\begin{equation*}
		|\tilde\ich_{\us} (v,u; z)| \le e^{C\max\{|\eta|,|\xi|\}^{1-\epsilon'}}
		\end{equation*}
		for either $v\in\rootsR_z\setminus\rootsR_z^{(\epsilon)}$ or $u\in\rootsL_z\setminus\rootsL_{z}^{(\epsilon)}$ where 
		\begin{equation*}
		\eta =\begin{dcases}
		\mathcal{M}_{L,\mathrm{right}} (v) & \text{for $v\in \rootsR_z^{(\epsilon)}$}, \\
		\frac{L^{1/2} (v+\rho) }{\sqrt{\rho(1-\rho)}} & \text{for $u\in \rootsR_z \setminus \rootsR_z^{(\epsilon)}$},
		\end{dcases}
		\end{equation*}
		and
		\begin{equation*}
		\xi = \begin{dcases}
		\mathcal{M}_{L,\mathrm{left}} (u)& \text{for $u\in \rootsL_z^{(\epsilon)}$}, \\
		\frac{L^{1/2} (u+\rho) }{\sqrt{\rho(1-\rho)}}& \text{for $u\in \rootsL_z \setminus \rootsL_z^{(\epsilon)}$}.
		\end{dcases}
		\end{equation*}
	\end{enumerate}
\end{lm}

\subsubsection{Proof of Lemma~\ref{lm:asymptofsf} (i)}

Since $y_j=j-N_2$, we have $\lambda(Y) = (0,0, \cdots, 0)$.
When $\lambda=0=(0, 0, \cdots, 0)$, the symmetric function \eqref{eq:newsymmf} becomes
\beqq
\begin{split}
	\gftnv_0(W)
	&= \frac{\det\left[1_{j\le N_2}w_i^{N-j} +1_{j>N_2}w_i^{N-j-1}(w_i+1)\right]_{i,j=1}^N } {\det\left[w_i^{N-j}\right]_{i,j=1}^N }\\
	&
	= \frac{\det\left[ w_i^{N-j} +1_{j>N_2}w_i^{N-j-1} \right]_{i,j=1}^N } {\det\left[w_i^{N-j}\right]_{i,j=1}^N } .
\end{split}
\eeqq
By the linearity of determinants, 
\begin{equation*} \begin{split}
\det\left[ w_i^{N-j} +1_{j>N_2}w_i^{N-j-1} \right]_{i,j=1}^N  
&=\sum_{k  = N_2}^N \det\left[ 1_{j\le k} w_i^{N-j} + 1_{j>k} w_i^{N-j-1} \right]_{i,j=1}^N \\
&= \frac1{w_1\cdots w_N} \sum_{k  = N_2}^N \det\left[  w_i^{N-j+ 1_{j\le k}}  \right]_{i,j=1}^N .
\end{split} \end{equation*}
The last determinant divided by the Vandermonde determinant is the Schur function $s_{(\underbrace{1,\cdots,1}_{k})}(w_1,\cdots,w_N)$. This particular Schur function is equal to the elementary symmetric function $e_k(w_1, \cdots, w_N)$. 
Hence, 
\begin{equation*}
\begin{split}
\gftnv_0(W) &=\frac{1}{w_1\cdots w_N}\sum_{k= N_2}^N e_k(w_1,\cdots,w_N)
= \frac{1}{w_1\cdots w_N} \sum_{k= N_2}^N \sum_{i_1<\cdots<i_k} w_{i_1}\cdots w_{i_k}	 .
\end{split}
\end{equation*}
The last sum can be written as an integral and we obtain
\begin{equation} \label{eq:qpqp2}
\begin{split}
\gftnv_0(W) &= \frac{1}{w_1\cdots w_N}\sum_{k=N_2}^N (-1)^{k} \oint_0  \newz^{-(N-k)} \left[ \prod_{i=1}^N (\newz-w_i) \right] \ddbar{\newz}{\newz}\\
& \quad =  \frac{(-1)^{N_2}}{2\pi \ii w_1\cdots w_N} \oint_0 \left[ \prod_{i=1}^N (\newz-w_i)\right]
\frac{\dd \newz}{(\newz+1)\newz^{N_1+1}}
+ \frac{(-1)^N}{2\pi \ii w_1\cdots w_N} \oint_0 \left[ \prod_{i=1}^N (\newz-w_i)\right] \frac{\dd \newz}{\newz+1 }, 
\end{split}
\end{equation}
where the contours are  small circles around the origin.

Consider $\tilde\energy_{\us}(z)= \gftnv_0(\rootsR_z)$. 
When $W= \rootsR_z$, the product $\prod_{i=1}^N (\newz-w_i)$ becomes $q_{z,\RR}(\newz)$, and hence \eqref{eq:qpqp2} becomes  
\begin{equation*} 
\tilde\energy_{\us}(z)
= \frac{(-1)^{N- N_2}}{2\pi \ii q_{z, \RR}(0)} \oint_0 
\frac{ q_{z,\RR}(\newz) }{(\newz+1)\newz^{N_1+1}}  \dd \newz
+ \frac{1}{2\pi \ii q_{z, \RR}(0)} \oint_0  \frac{q_{z,\RR}(\newz) }{ \newz+1 }  \dd \newz .
\end{equation*}
The integrand of the second integral is analytic at $\newz=0$. Thus the second integral is zero. 
Using the identity $q_{z, \RR}(\newz) q_{z, \LL}(\newz) = \newz^N(\newz+1)^{L-N}- z^L$ and the fact that $N=N_1+N_2$, the first integral is equal to 
\beqq
\oint_0  	\frac{ q_{z,\RR}(\newz) }{(\newz+1)\newz^{N_1+1}}   \dd \newz
= \oint_0 \frac{ \newz^{N_2-1} (\newz+1)^{L-N-1} }{q_{z_1,\LL}(\newz)} \dd \newz 
-z^L \oint_0 \frac{1}{q_{z_1,\LL}(\newz) (\newz+1) \newz^{N_1+1}} \dd \newz .
\eeqq
The first integral is zero since its integrand is analytic at the origin. 
Therefore, we find that 
\begin{equation}	\label{eq:aux_036}
\tilde\energy_{\us}(z) = \frac{(-1)^{N-N_2+1} z^L}{2\pi \ii q_{z, \RR}(0) } \oint_0 \frac{1}{q_{z_1,\LL}(\newz) (\newz+1) \newz^{N_1+1}} \dd \newz.
\end{equation}

We evaluate the asymptotics of the integral. 
We deform the contour to the circle  $|\newz|=\rho-\epsilon'\sqrt{\rho(1-\rho)} L^{-1/2}$ for a small fixed constant $\epsilon'>0$. 
We show that the main contribution to the integral comes from the part of the circle which is of distance of order $L^{-1/2}$ to the point $\newz=-\rho$. 
For $\newz$ satisfying $\newz = -\rho + \sqrt{\rho(1-\rho)} \zeta L^{-1/2}$ with $|\zeta|\le L^{\epsilon/4}$, 
Lemma \ref{lem:asymofprod} (ii) implies that 
\beqq
\frac{(\newz+1)^{L-N}}{q_{z_1,\LL}(\newz)} 
=  e^{- \mathrm{h} (\zeta,\mathrm{z}) } \left( 1+ O(L^{\epsilon-1/2}\log L)\right) 
\eeqq
and a direct calculation using $N_1=N-N_2 = N -\alpha \rho^{1/2}(1-\rho)^{-1/2}L^{1/2}$ implies that 
\beqq
(\newz+1)^{L-N+1} \newz^{N_1+1}
=  (1-\rho)^{L-N+1} (-\rho)^{N_1+1} e^{- \zeta^2/2 + \alpha\zeta} (1+O(L^{\epsilon-1/2} \log L)).
\eeqq 
Together, we obtain the asymptotic formula of the integrand for the part of the contour in a neighborhood of the point $-\rho$. 
On the other hand, for the part of the contour satisfying $|\newz+\rho| \ge L^{\epsilon/4-1/2}$, we use Lemma \ref{lem:asymofprod} (iii) to find that 
\beq
\label{eq:2019_11_30_01}
\left|\frac{1}{q_{z_1,\LL}(\newz) (\newz+1) \newz^{N_1+1}}\right| \le C\frac{1}{\left|\newz+1\right|^{L-N-1}|\newz|^{N_1+1}}.
\eeq
Note $\newz$ lies on a circle centered at $0$ with radius $\rho -\epsilon' \sqrt{\rho(1-\rho)}  L^{-1/2}$, we have
\begin{equation*}
\begin{split}
|\newz|&=\rho-\epsilon'\sqrt{\rho(1-\rho)}L^{-1/2},
\\
|\newz+1|&\ge \sqrt{|\newz+\rho|^2+|1-\rho|^2}=(1-\rho)(1+C'L^{\epsilon/4-1/2})
\end{split}
\end{equation*}
for some positive constant $C'$.
Recall that $N_1=O(L)$. We insert the above estimates in~\eqref{eq:2019_11_30_01} and have
\begin{equation*}
\left|\frac{1}{q_{z_1,\LL}(\newz) (\newz+1) \newz^{N_1+1}}\right| \le \frac{e^{-C''L^{(\epsilon+1)/2}}}{(1-\rho)^{L-N+1} (-\rho)^{N_1+1}}
\end{equation*}
for some positive constant $C''$.

Therefore, we obtain 
\begin{equation}
\label{eq:aux_037}
\begin{split}
&\oint_0 \frac{1}{q_{z_1,\LL}(\newz) (\newz+1) \newz^{N_1+1}} \dd \newz \\
&= \frac{-\sqrt{\rho(1-\rho)}L^{-1/2}}{(1-\rho)^{L-N+1} (-\rho)^{N_1+1}} \left( \int_{\epsilon'-\ii\infty}^{\epsilon'+\ii\infty} e^{-\mathrm{h}(\zeta,\mathrm{z}_1)+\zeta^2/2-\alpha\zeta} \dd \zeta \right)  (1+ O(L^{\epsilon-1/2}\log L))\\
&= \frac{-2\pi \ii \sqrt{\rho(1-\rho)}e^{-\alpha^2/2}}{\sqrt{2\pi L}(1-\rho)^{L-N+1} (-\rho)^{N_1+1}}  E_{\mathrm{\us}}(\mathrm{z})  (1+ O(L^{\epsilon-1/2}\log L))
\end{split}
\end{equation}
for all sufficiently large $L$.  Here we used the following identity for $E_{\us}^{(\alpha)} (\mathrm{z}_1)$:
\begin{equation*}
E_{\us}^{(\alpha)} (\mathrm{z}_1)
=\int_{\epsilon'-\ii\infty}^{\epsilon'+\ii\infty} e^{-\mathrm{h} (\zeta,\mathrm{z}_1)+(\zeta-\alpha)^2/2} \frac{\dd \zeta}{\ii \sqrt{2\pi} },
\end{equation*}
which follows from a simple deformation of contour and change of variable.

Note that setting $\newz=0$ in the identity $q_{z, \RR}(\newz) q_{z, \LL}(\newz) = \newz^N(\newz+1)^{L-N}- z^L$, we find that 
\begin{equation}
\label{eq:aux_038}
- \frac{z^L}{q_{z, \RR}(0)} = q_{z, \LL}(0)= 1 + O(L^{\epsilon-1/2}),
\end{equation}
where we used the first equation of Lemma \ref{lem:asymofprod} (i) with $w=0$. 

Finally, the Stirling's formula implies 
\begin{equation}
\label{eq:aux_039}
{L-N_2\choose N_1} = \frac{e^{-\alpha^2/2} }{\sqrt{2\pi L}(1-\rho)^{L-N+1/2} \rho^{N_1+1/2}}  (1+L^{-1/2}) .
\end{equation}
Therefore, we obtain Lemma~\ref{lm:asymptofsf} (i) from the formula~\eqref{eq:aux_036} using  estimates~\eqref{eq:aux_037},~\eqref{eq:aux_038} and~\eqref{eq:aux_039}.

\subsubsection{Proof of Lemma~\ref{lm:asymptofsf} (ii) and (iii)}

When $W=\rootsR_z\cup \{u\}\setminus \{v\}$ for $u\in\rootsL_z$ and $v\in \rootsR_z$, 
the product $\prod_{i=1}^N (\newz-w_i)$ is equal to $\frac{\newz-u}{\newz-v} q_{z, \rootsR_z}(\newz)$. 
Following the calculation that lead to the formula \eqref{eq:aux_036}, in which we change $q_{z, \rootsR_z}(\newz)$ to $q_{z, \rootsL_z}(\newz)$, we find that the formula \eqref{eq:def_ichv} becomes
\begin{equation*}
\tilde\ich_{\us}(v,u;z) = \frac{(-1)^{N-N_2+1}z^L v }{ 2\pi \ii u q_{z, \RR}(0) \tilde\energy(z) } \oint_0 \frac{\newz-u}{ (\newz-v) q_{z_1,\LL}(\newz) (\newz+1) \newz^{N_1+1}} \dd \newz,  
\end{equation*}
where the contour is a small circle, which satisfies, in particular, $|\newz|<|v|$. 

To prove Lemma~\ref{lm:asymptofsf} (ii), we consider 
\beqq
v = -\rho +\mathrm{v}  \sqrt{\rho(1-\rho)} L^{-1/2} \quad \text{and} \quad  u=  -\rho +\mathrm{u} \sqrt{\rho(1-\rho)} L^{-1/2}
\eeqq
for $\mathrm{v},\mathrm{u}$ satisfying $\Re(\mathrm{v})>0$, $\Re(\mathrm{u})<0$, and $|\mathrm{u}|,|\mathrm{v}|\le L^{\epsilon}$. 
Using the similar calculations as the previous subsubsection, we obtain 
\begin{equation} \label{eq:aux_041}
\tilde\ich_{\us}(v,u;z_1) = \frac{1}{E_{\us}(\mathrm{z_1})} \int_{c-\ii\infty}^{c+\ii\infty} \frac{\zeta-\mathrm{u}}{\zeta-\mathrm{v}}e^{-\mathrm{h} (\zeta,\mathrm{z}_1)+(\zeta-\alpha)^2/2}  \frac{\dd\zeta}{\ii \sqrt{2\pi}} (1+O(L^{4\epsilon-1/2}))
\end{equation}
for any constant $c>\Re(\mathrm{v})$. 
By Lemma~\ref{lm:limiting_nodes}, we may replace  $\mathrm{u}$ and $\mathrm{v}$ by $\xi=\mathcal{M}_{L,\mathrm{left}}(u)$ and $\eta=\mathcal{M}_{L,\mathrm{right}}(v)$ respectively and the error is still bounded by $O(L^{4\epsilon-1/2})$. Finally note the identity
\begin{equation*}
\chi_{\us}^{(\alpha)}  (\eta,\xi;\mathrm{z}_1) 
=\frac{1}{E_{\us}^{(\alpha)} (\mathrm{z_1})} \int_{c-\ii\infty}^{c+\ii\infty} e^{-\mathrm{h} (\zeta,\mathrm{z}_1)+(\zeta-\alpha)^2/2} \frac{\zeta-\xi}{\zeta-\eta} \frac{\dd \zeta}{\ii \sqrt{2\pi} } 
\end{equation*}
for $\eta\in\inodesR_{\mathrm{z}_1}$ and $\xi\in\inodesL_{\mathrm{z}_1}$
by deforming the integral contour and changing the variable.
Thus Lemma~\ref{lm:asymptofsf} (ii) follows immediately.

Lemma~\ref{lm:asymptofsf} (iii) is obtained by modifying the proof of \eqref{eq:aux_041} by first showing that
$\tilde\ich(v,u;z)$ is bounded by an exponential growth. 
This computation is tedious and we skip the details.

\begin{appendix}
	
	\section{Probabilistic argument for the step-flat case when $L_s=O(L)$}
	\label{sec:prob_step_flat}

	In the step-flat initial condition, the parameters satisfy $L=dN+L_s$ for $0<L_s<N$. 
	We evaluated the limit when $L_s=O(\sqrt{L})$ in Section \ref{sec:flatandstepflat}. 
	In this section, we discuss the case when $L_s=O(L)$ and provide a probabilistic argument that the large time limit should be same as the step initial condition. 
	It should be possible to make this argument rigorous but we content to discuss heuristically since this is not a main part of this paper.

	We discuss in terms of the periodic directed last passage percolation (DLPP) model which is well-known to be related to the TASEP.
	Assume that each lattice point $(i,j)\in\intZ^2$ is assigned an exponential random variable $w(i,j)$ with parameter $1$. These $w(i,j)$'s are all independent except for the following periodicity condition
	\begin{equation*}
	w(i,j) = w(i+L-N,j-N),\qquad i,j\in\intZ.
	\end{equation*}
	Let $\Lambda$ be a lattice path; $\Lambda$ consists of connected
	unit horizontal or vertical line segments with vertices in $\intZ^2$, and $\Lambda$ does not intersect any line $y-x=$constant twice. We define the last passage time from $\Lambda$ to a lattice point $\mathrm{p}$ as 
	\begin{equation*}
	\mathcal{L}_\Lambda (\mathrm{p}):= \max_{\Pi:\Lambda\to \mathrm{p}} \sum_{(i,j)\in \Pi} w(i,j),
	\end{equation*}
	where the maximum is over all possible up/right lattice path starting from any lattice point in $\Lambda$ and ending at $\mathrm{p}$. We assume that $\mathcal{L}_\Lambda(\mathrm{p})=-\infty$ if no such path exists. Similarly, one can define $\mathcal{L}_{\mathrm{q}}(\mathrm{p})$ if the lattice path is restricted to start from a given point $\mathrm{q}$.

	Now we consider the step-flat and step initial conditions of the $\PTASEP$. These two initial conditions
	in the language of periodic DLPP, correspond to
	the lattice paths 
	\begin{equation*}
	\begin{split}
	&\Lambda_{\stof}\\
	&=\left(\{(i,j): -(d-1)j \le i \le -(d-1)(j-1), 1\le i\le N\}\cup\{(i,0):0\le i\le L_s\}\right)+ (N,-L+N)\intZ
	\end{split}
	\end{equation*}
	and
	\begin{equation*}
	\Lambda_{\mathrm{step}}=\left(\{(0, j): 0\le j\le N\} \cup \{(i,0): 0\le i\le L-N\}\right)+ (N,-L+N)\intZ,
	\end{equation*}
	respectively. 
	From the well-known connection between the TASEP and the DLPP, the convergence of the $\PTASEP$ with the step-flat initial condition to the step flat initial condition in the large $L$ limit with $L_s=O(L)$ is translated into the following question on the periodic DLPP: 
	Our goal is to show that when (1) $\mathrm{p}$ is far enough, more precisely,
	\begin{equation*}
	\dist \left(\mathrm{p},\Lambda_{\stof}\right) = O(L^{3/2}),
	\end{equation*}
	which corresponds to the relaxation time scale in PTASEP, and (2) $L_s= O(L)$, then 
	\begin{equation}
	\label{eq:aux_084}
	\mathcal{L}_{\Lambda_{\stof}} (\mathrm{p}) = \mathcal{L}_{\Lambda_{\mathrm{step}}} (\mathrm{p}) +o(L^{1/2}) \text{ as $L\to\infty$. }
	\end{equation}
	It is known that in the relaxation time scale, both  $\mathcal{L}_{\Lambda_{\stof}} (\mathrm{p})$ and $\mathcal{L}_{\Lambda_{\mathrm{step}}} (\mathrm{p})$ have $O(\sqrt{L})$ fluctuations. The above estimate implies $\mathcal{L}_{\Lambda_{\stof}} (\mathrm{p})$ and $\mathcal{L}_{\Lambda_{\mathrm{step}}} (\mathrm{p})$ have the same limiting fluctuation. In other words, if $L_s=O(L)$, these two initial conditions yield to the same limiting fluctuations in the relaxation time scale.

	\begin{figure}
		\centering
		\includegraphics[scale=0.25]{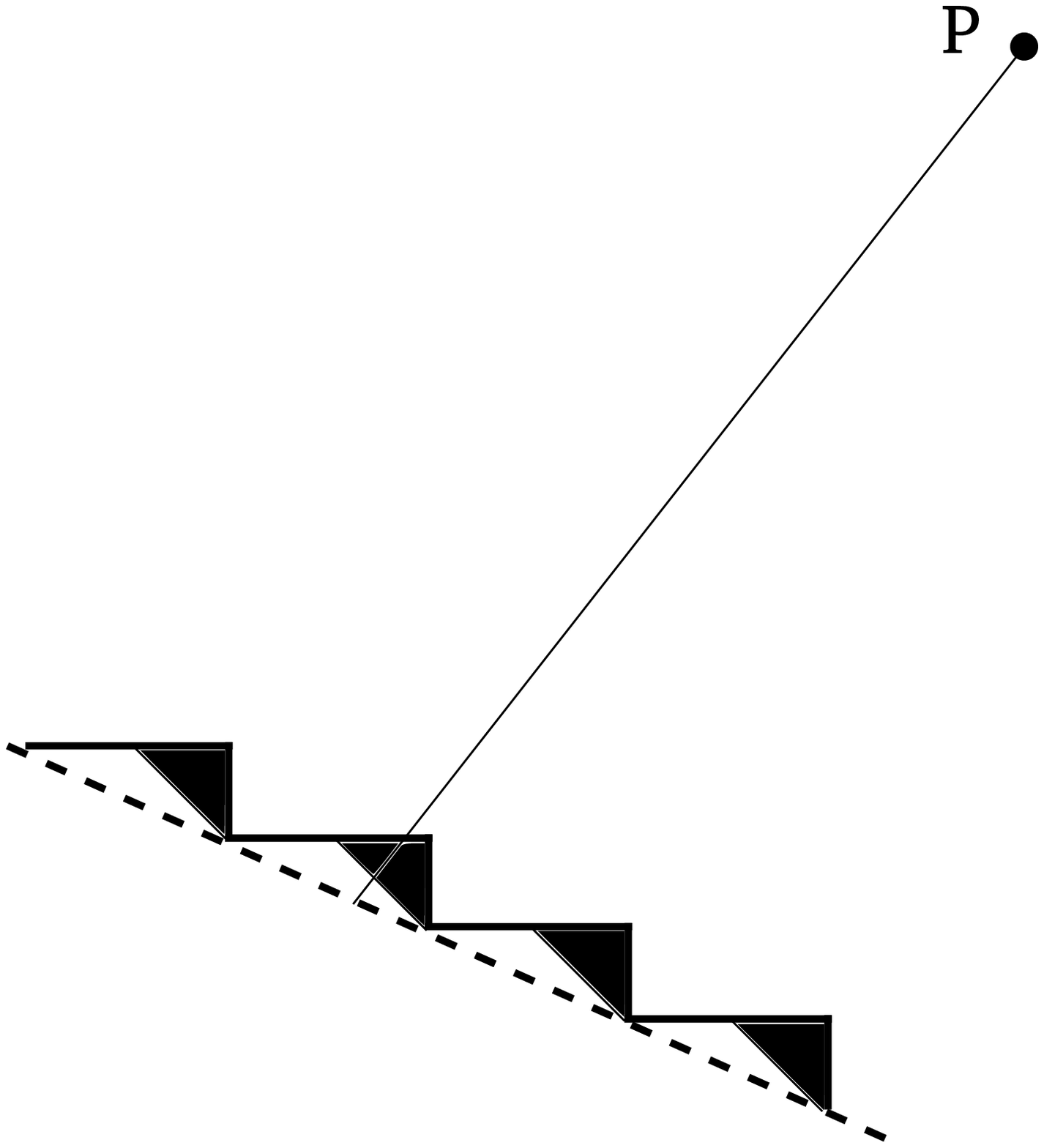}\qquad\includegraphics[scale=0.25]{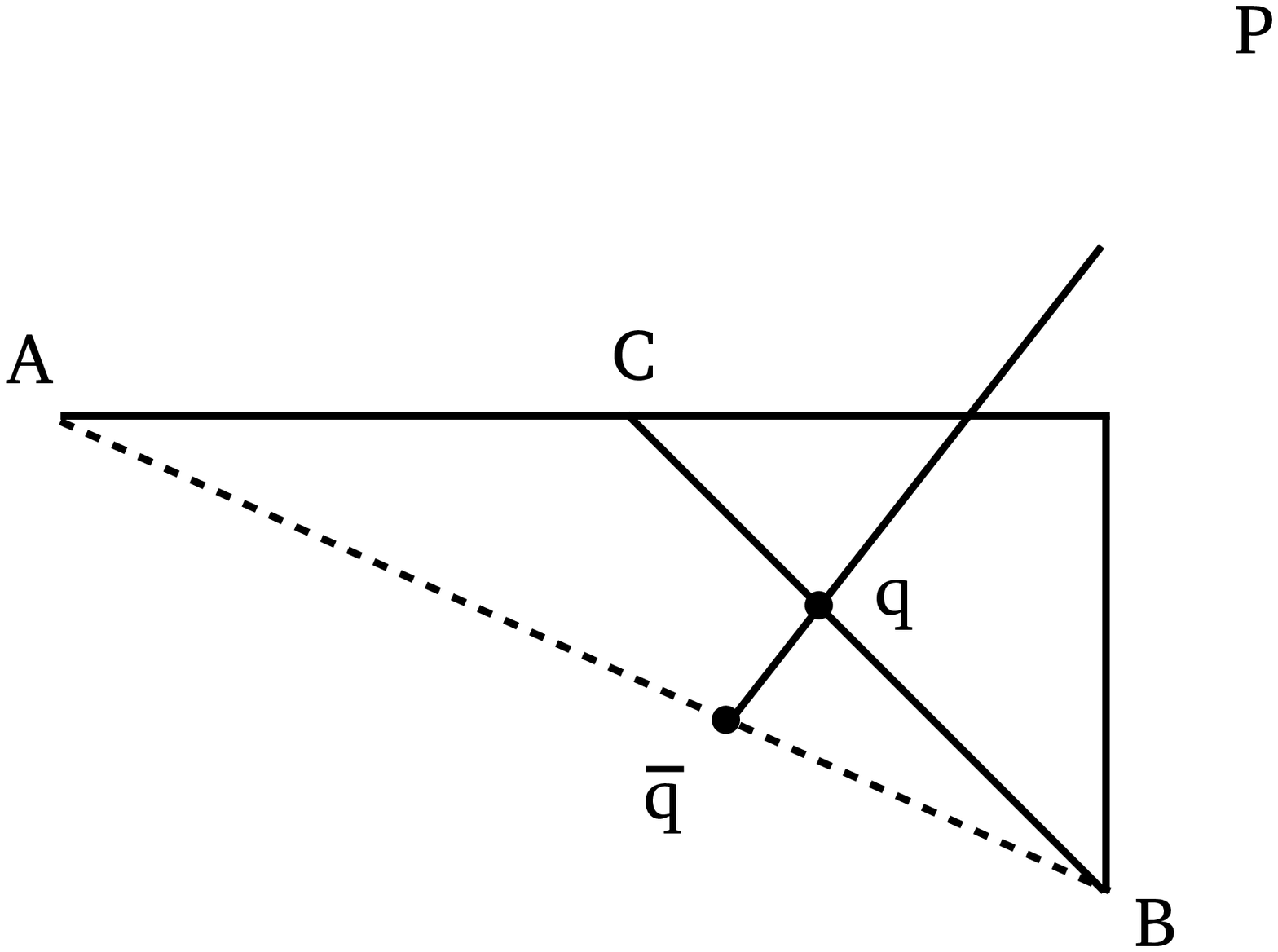}
		\caption{Illustration of periodic DLPP from $\Lambda_{\stof}$ to $\mathrm{p}$. In the figure on the left, the stair-shape path is $\Lambda_{\mathrm{step}}$, the dotted line is $\Lambda_{\mathrm{flat}}$ with particle density $\rho$, and  $\Lambda_{\stof}$ lies between between $\Lambda_{\mathrm{step}}$ and $\Lambda_{\mathrm{flat}}$; It is the path between the white region above $\Lambda_{\mathrm{flat}}$ and the black region below $\Lambda_{\mathrm{step}}$. The figure on the right is a detailed view of the maximal path.}
		\label{fig:DLPP}
	\end{figure}
	
	See Figure~\ref{fig:DLPP} for an illustration of $\Lambda_{\mathrm{step}}$ and $\Lambda_{\stof}$. 
	From the definition, $\Lambda_{\stof}$ lies on the lower left side of $\Lambda_{\mathrm{step}}$. 
	Hence, the last passage time from $\Lambda_{\stof}$ is larger than or equal to the time from $\Lambda_{\mathrm{step}}$. This implies that 
	\begin{equation*}
	\mathcal{L}_{\Lambda_{\stof}} (\mathrm{p}) \ge \mathcal{L}_{\Lambda_{\mathrm{step}}} (\mathrm{p}) . 
	\end{equation*}
	Thus, \eqref{eq:aux_084} follows is we show that 
	\begin{equation}
	\label{eq:aux_082}
	\mathcal{L}_{\Lambda_{\stof}} (\mathrm{p}) \le  \mathcal{L}_{\Lambda_{\mathrm{step}}} (\mathrm{p}) +o(L^{1/2}) \text{ as $L\to\infty$. }
	\end{equation}

	Let 
	\begin{equation*}
	\Lambda_{\mathrm{flat}}:=\{(x,y)\in\realR^2: y=(1-\rho^{-1}) x\}.
	\end{equation*}
	We remark that $\Lambda_{\mathrm{flat}}$ is not necessary a lattice path. See Figure~\ref{fig:DLPP} for $\Lambda_{\mathrm{flat}}$.

	The inequality~\eqref{eq:aux_082} heuristically follows from the slow decorrelation of (periodic) directed last passage percolation, which was discussed in \cite{Corwin-Ferrari-Peche12}. 
	We point out that although this paper was for the DLPP, the same argument extends to the  periodic DLPP. The slow decorrelation implies, if the starting point $\mathrm{q}\in\Lambda_{\stof}$ is not on $\Lambda_{\mathrm{step}}$, say $\mathrm{q}$ is between two corners $A$ and $B$ as shown in Figure~\ref{fig:DLPP}, then
	\begin{equation}
	\label{eq:aux_083}
	\mathcal{L}_{\mathrm{q}} (\mathrm{p}) =
	-c\cdot \dist(\mathrm{q},\mathrm{\bar q}) + \mathcal{L}_{\mathrm{\bar q}} (\mathrm{p}) +o((\dist(\mathrm{p},\mathrm{q}))^{1/3})=-c\cdot \dist(\mathrm{q},\mathrm{\bar q}) + \mathcal{L}_{\mathrm{\bar q}} (\mathrm{p}) +o(L^{1/2}),
	\end{equation}
	where $c>0$ is some constant independent of $L$, and $\mathrm{\bar{q}}$ is the intersection of the line $\mathrm{pq}$ and $\Lambda_{\mathrm{flat}}$.
	
	Note that
	\begin{equation*}
	\mathcal{L}_{\mathrm{\bar q}} (\mathrm{p})\le \mathcal{L}_{\Lambda_{\mathrm{flat}}}(\mathrm{p})=\mathcal{L}_{\Lambda_{\mathrm{step}}}(\mathrm{p})+O(L^{1/2}),
	\end{equation*}
	where the last equality follows from Theorem~\ref{thm:main} (the one point distribution case) and the case we proved for step and flat initial conditions (Theorem~\ref{thm:special_IC})\footnote{We only proved the flat case with $\rho^{-1}\in\{2,3,\cdots\}$. However, heuristically we expect this is true for any $\rho\in(0,1)$.}. We assume that $\mathrm{\bar q}$ is within an $O(L)$ interval where the maximum path from $\Lambda_{\mathrm{flat}}$ to $\mathrm{p}$ is obtained. Otherwise, $\mathcal{L}_{\mathrm{\bar q}} (\mathrm{p})<\mathcal{L}_{\Lambda_{\mathrm{step}}}(\mathrm{p})$ for far enough $\mathrm{\bar q}$. And $\mathcal{L}_{\mathrm{q}} (\mathrm{p}) <\mathcal{L}_{\mathrm{\bar q}} (\mathrm{p}) +o(L^{1/2})<\mathcal{L}_{\Lambda_{\mathrm{step}}} (\mathrm{p}) +o(L^{1/2})$ holds trivially. This assumption means that $\mathcal{L}_{\mathrm{\bar q}} (\mathrm{p})$ and $\mathcal{L}_{B} (\mathrm{p})$ have the same deterministic order terms and they only differ from the fluctuation terms, which is of $O(L^{1/2})$.
	
	We let $C$ be the other intersection point of the line $B\mathrm{q}$ with $\Lambda_{\mathrm{step}}$. It also lies on $\Lambda_{\stof}$. By the definition, $\dist(A,C)=L_s$ has the same order as $\dist(A,B)$. Hence $\dist(\mathrm{q},\mathrm{\bar q})$ has the same order as $\dist(B,\mathrm{\bar q})$.
	
	Now we consider two situations. If $\dist(\mathrm{q},\mathrm{\bar q})\ll O(L)$, then $\dist(B,\mathrm{\bar q})\ll O(L)$. In this case, $\mathcal{L}_{\mathrm{\bar{q}}}(\mathrm{p})$ is asymptotically identical to $\mathcal{L}_{B}(\mathrm{p})$ since the two points $B$ and $\mathrm{\bar q}$ are closer than the correlation length $O(L)$. More precisely, we have
	\begin{equation*}
	\mathcal{L}_{\mathrm{\bar{q}}}(\mathrm{p})=\mathcal{L}_{B}(\mathrm{p})+o(L^{1/2})\le \mathcal{L}_{\Lambda_{\mathrm{step}}}(\mathrm{p})+o(L^{1/2}).
	\end{equation*}
	Together with~\eqref{eq:aux_083} we obtain
	\begin{equation}
	\label{eq:aux_088}
	\mathcal{L}_{\mathrm{q}} (\mathrm{p})\le \mathcal{L}_{\Lambda_{\mathrm{step}}}(\mathrm{p})+o(L^{1/2}).
	\end{equation}
	
	The second situation is that $\dist(\mathrm{q},\mathrm{\bar q})\gg O(L^{1/2})$. In this case, we use the trivial estimate
	\begin{equation*}
	\mathcal{L}_{\mathrm{\bar{q}}}(\mathrm{p}) \le \mathcal{L}_{B}(\mathrm{p})+O(L^{1/2})\le \mathcal{L}_{\Lambda_{\mathrm{step}}}(\mathrm{p})+O(L^{1/2}),
	\end{equation*}
	while the term $O(L^{1/2})$ in this estimate is always dominated by $-c\cdot \dist(\mathrm{q},\mathrm{\bar q})$ in~\eqref{eq:aux_083}. Hence we still have~\eqref{eq:aux_088}.
	
	Note that $\mathrm{q}$ is an arbitrary point on $\Lambda_{\stof}$, the above estimates imply~\eqref{eq:aux_082}.
	
\end{appendix}

\def\cydot{\leavevmode\raise.4ex\hbox{.}}

\end{document}